\documentclass[10pt]{article}
\usepackage{amssymb, amsmath, amsthm}

\usepackage{amsfonts,mathtools,mathrsfs,mathtools,color}
\usepackage[colorlinks]{hyperref}
\usepackage{graphicx,subfigure}
\usepackage{multirow}
\usepackage{soul}
\usepackage[normalem]{ulem}
\usepackage{algorithm}
\usepackage{algpseudocode}
\usepackage{tabularx}
\usepackage{bm}
\usepackage{enumerate}
\usepackage{xcolor}

\DeclareMathOperator*{\argmax}{arg\,max}

\DeclareMathOperator*{\Haar}{Haar}
\DeclareMathOperator*{\image}{image}

\newcommand{\Sym}{\mathrm{Sym}}

\newcommand{\C}{\mathbb{C}}
\newcommand{\E}{\mathbb{E}}

\newcommand{\N}{\mathcal{N}}
\newcommand{\R}{\mathbb{R}}

\newcommand{\Z}{\mathbb{Z}}
\renewcommand{\P}{\mathbb{P}}
\newcommand{\tr}{\mathop{\mathrm{Tr}}}

\newcommand{\bzero}{\mathbf{0}}
\newcommand{\lin}{\mathrm{lin}}

\newcommand{\unif}{\textnormal{Unif}}

\newcommand{\mmse}{\textnormal{MMSE}}
\newcommand{\iden}{\textnormal{Id}}
\newcommand{\id}{\textnormal{id}}
\newcommand{\gs}{\textnormal{gs}}
\newcommand{\qs}{\textnormal{qa}}
\newcommand{\ud}{\mathrm{d}}
\newcommand{\SO}{\mathbb{SO}}
\renewcommand{\SS}{\mathbb{S}}

\newcommand{\U}{\mathbb{U}}

\newcommand{\diag}{\textnormal{diag}}
\newcommand{\overlap}{\mathbf{q}}

\newcommand{\overlapm}{\mathbf{m}}
\newcommand{\ba}{\mathbf{a}}
\newcommand{\bb}{\mathbf{b}}
\newcommand{\bd}{\mathbf{d}}
\newcommand{\be}{\mathbf{e}}
\newcommand{\bg}{\mathbf{g}}
\newcommand{\bh}{\mathbf{h}}
\newcommand{\bbm}{\mathbf{m}} %
\newcommand{\bq}{\mathbf{q}}
\newcommand{\bu}{\mathbf{u}}
\newcommand{\bv}{\mathbf{v}}
\newcommand{\bx}{\mathbf{x}}
\newcommand{\by}{\mathbf{y}}
\newcommand{\bz}{\mathbf{z}}

\newcommand{\ab}{\mathbf{a}}
\newcommand{\db}{\mathbf{d}}
\newcommand{\eb}{\mathbf{e}}

\newcommand{\gb}{\mathbf{g}}
\newcommand{\hb}{\mathbf{h}}

\newcommand{\qb}{\mathbf{q}}

\newcommand{\ub}{\mathbf{u}}
\newcommand{\vb}{\mathbf{v}}

\newcommand{\yb}{\mathbf{y}}
\newcommand{\zb}{\mathbf{z}}

\newcommand{\cF}{\mathcal{F}}
\newcommand{\cG}{\mathcal{G}}
\newcommand{\cH}{\mathcal{H}}
\newcommand{\cI}{\mathcal{I}}

\newcommand{\cK}{\mathcal{K}}
\newcommand{\cL}{\mathcal{L}}
\newcommand{\cM}{\mathcal{M}}
\newcommand{\cN}{\mathcal{N}}

\newcommand{\cQ}{\mathcal{Q}}

\newcommand{\cS}{\mathcal{S}}

\newcommand{\cX}{\mathcal{X}}

\newcommand{\PP}{\mathbb{P}}

\newcommand{\RR}{\mathbb{R}}
\newcommand{\btheta}{\bm{\theta}}

\def\abs#1{\left| #1 \right|}
\newcommand{\norm}[1]{\left\lVert#1\right\rVert}

\newcommand{\inparen}[1]{\left(#1\right)}             
\newcommand{\inbraces}[1]{\left\{#1\right\}}           
\newcommand{\insquare}[1]{\left[#1\right]}             
\newcommand{\inangle}[1]{\left\langle#1\right\rangle} 

\newtheorem{theorem}{Theorem}[section]
\newtheorem{proposition}[theorem]{Proposition}
\newtheorem{lemma}[theorem]{Lemma}
\newtheorem{corollary}[theorem]{Corollary}

\theoremstyle{definition}
\newtheorem{definition}[theorem]{Definition}
\newtheorem{example}[theorem]{Example}
\newtheorem{assumption}[theorem]{Assumption}
\newtheorem{remark}[theorem]{Remark}

\newenvironment{proofof}[1]{\begin{trivlist} \item {\textbf{{Proof of
#1}.~~}}}
  {\qed\end{trivlist}}

\usepackage{subfiles}

\title{Asymptotic mutual information in quadratic estimation problems over compact groups}
\author{Kaylee Y.\ Yang\thanks{These authors contributed equally.\newline
Department of Statistics and Data Science, Yale University\newline
\texttt{yingxi.yang@yale.edu}, \texttt{timothy.wee@yale.edu}, \texttt{zhou.fan@yale.edu}} 
\and Timothy L.\ H.\ Wee\footnotemark[1]
\and Zhou Fan}
\date{}

\begin{document}
\maketitle

\begin{abstract}
Motivated by applications to group synchronization and quadratic assignment on
random data, we study a general problem of Bayesian inference of an
unknown ``signal'' belonging to a high-dimensional compact group, given noisy
pairwise observations of a featurization of this signal. We establish a
quantitative comparison between the signal-observation mutual information in any
such problem with that in a simpler model with linear observations,
using interpolation methods.
For group synchronization, our result proves a replica formula for the
asymptotic mutual information and Bayes-optimal mean-squared-error.
Via analyses of this replica formula, we show
that the conjectural phase transition threshold for
computationally-efficient weak recovery of the signal is determined by
a classification of the real-irreducible components of the observed group
representation(s), and we fully characterize the information-theoretic limits
of estimation in the example of angular/phase synchronization over
$\SO(2)$/$\U(1)$.
For quadratic assignment, we study observations given by a kernel matrix of
pairwise similarities and a randomly permutated and noisy counterpart, and we
show in a bounded signal-to-noise regime that the asymptotic mutual information
coincides with that in a Bayesian spiked model with i.i.d.\ signal prior.
\end{abstract}


\section{Introduction}\label{sec: introduction}

The estimation of a low-rank matrix in a noisy channel is a fundamental problem
in statistical inference, which has received much attention in recent years
\cite{deshpande2015asymptotic,krzakala2016mutual,dia2016mutual,el2018estimation,lelarge2019fundamental,barbier2019adaptive,barbier2020information}.
In this work, we are motivated by two applications that may be viewed as
extensions or variants of this problem:

\begin{itemize}
\item In \emph{group synchronization}, we wish to estimate a collection of
elements $\gb_{*1},\dots, \gb_{*N} \in \cG$ from a known compact group $\cG$,
given noisy observations of their pairwise alignments
\[\yb_{ij} = \gb_{*i}\gb_{*j}^{-1} + \text{Gaussian noise}.\]
Examples include synchronization problems over the binary group $\Z/2\Z$ with
application to community detection in networks \cite{deshpande2015asymptotic},
over $\SO(2)$ (or equivalently $\U(1)$ in the complex domain) with application
to angular and phase synchronization \cite{singer2011angular}, over $\SO(3)$
with application to image registration and cryo-electron microscopy
\cite{bandeira2020non}, and over the symmetric group $\SS_k$ with application
to multi-way matching \cite{pachauri2013solving}.

\item In \emph{quadratic assignment}, we wish to estimate a permutation $\pi
\in \SS_N$ (the symmetric group on $N$ elements) that minimizes a cost function
\[\sum_{1 \leq i<j \leq N} (y_{ij}-a_{\pi(i)\pi(j)})^2\]
for two sets of pairwise similarities $\{a_{ij}\}_{1 \leq i<j \leq N}$
and $\{y_{ij}\}_{1 \leq i<j \leq N}$
between $N$ objects. We study a statistical setting where
$a_{ij}=\kappa(x_i,x_j)$ is the evaluation of a symmetric kernel function
$\kappa(\cdot,\cdot)$ on samples $x_1,\ldots,x_N$, and the above quadratic cost
arises as the log-likelihood in a model
\[y_{ij}=a_{\pi_*(i)\pi_*(j)}+\text{Gaussian noise}\]
for an unknown true
permutation $\pi_* \in \SS_N$. This is a Gaussian-noise analogue of some
recently studied models of graph matching on random geometric
graphs \cite{wang2022random,gong2024umeyama,liu2024random},
here with independent noise for each measurement pair
$(i,j)$ rather than for each underlying sample $x_1,\ldots,x_N$.
\end{itemize}

These two seemingly different problems share a common underlying structure of
inferring an unknown element $G_* \in \cG_N$ of a high-dimensional group
from noisy pairwise observations, where $\cG_N \equiv \cG^N$ is an $N$-fold
product group in synchronization, and $\cG_N \equiv \SS_N$ is the symmetric group in quadratic assignment. Other applications having this structure include
problems of ranking from pairwise comparisons \cite{furnkranz2010preference,negahban2012iterative}, and Procrustes
hyperalignment problems that arise in analyses of functional MRI data
\cite{haxby2011common,lorbert2012kernel}.

In this work, we introduce and study a general formulation for such problems in
a Bayesian setting, where pairwise measurements
\begin{equation}\label{eq:generalmodelintro}
\by_{ij}=\phi(G_*)_i \bullet \phi(G_*)_j+\text{Gaussian noise} \text{ for } 1
\leq i<j \leq N
\end{equation}
are observed corresponding to a featurization $\phi(\cdot)$ of
$G_*$ belonging to a compact group $\cG_N$, assumed to have
Haar-uniform prior distribution. The Hamiltonian of the Bayes posterior law is a
$\cG_N$-indexed Gaussian
process whose mean and covariance are determined by a corresponding overlap
function
\[Q(G,G')=N^{-1} \sum_{i=1}^N \phi(G)_i \otimes \phi(G')_i.\]
We refer to Section \ref{sec: general_model} for details of this setup.
In the context of this general model as well as the aforementioned
synchronization and quadratic assignment applications of interest,
our work makes the following contributions:

\begin{enumerate}
\item We analyze the mutual information between the latent group element
$G_* \in \cG_N$ and the observations $\{\by_{ij}\}_{i<j}$ in general models
of the form~\eqref{eq:generalmodelintro}, showing that this admits an
approximation in terms of the mutual information in a linear observation model
\[\by_i=\overlap^{1/2} \phi(G_*)_i+\text{Gaussian noise} \text{ for }
i=1,\ldots,N\]
defined by a suitable element $\overlap$ of the overlap space.
The approximation error
is small when the overlap space has small covering number, encompassing
scenarios where the signal component of the pairwise measurements
has low effective rank.

Our proof of this result uses interpolation arguments that have been
successfully developed and applied to establish replica formulas in problems of
low-rank matrix estimation
\cite{krzakala2016mutual,el2018estimation,barbier2019adaptive}. In particular, we apply an elegant
method of \cite{el2018estimation} for proving an upper bound on the free energy
(i.e.\ lower bound on the mutual information) by interpolating on the
Franz-Parisi potential at each fixed overlap, adapting this method to settings
where the prior law of $G_* \in \cG_N$ has group symmetry but may not
necessarily decompose as a product of i.i.d.\ components.

\item Specialized to group synchronization, we provide a rigorous
proof of a replica formula for the asymptotic signal-observation mutual
information and Bayes-optimal minimum mean-squared error (MMSE) in a bounded
SNR regime. A version of this replica formula in a model with complex
observations was stated in \cite{PWBM2018}, which also proposed
Approximate Message Passing algorithms for inference.

We obtain a complete characterization of the optimization landscape of the
replica potential for (single-channel) $\SO(2)$/$\U(1)$-synchronization,
implying a characterization of the information-theoretic limits of inference.
More generally, for any group, we analyze the stability of the overlap
$\overlap=\bzero$ as a critical point of the
replica potential, which conjecturally corresponds to the feasibility of
non-trivial signal estimation (i.e.\ weak recovery) by polynomial-time
algorithms \cite{lesieur2017constrained,lelarge2019fundamental}. We show that the phase transition threshold for local
optimality of $\overlap=\bzero$ is
determined by the SNR parameters of an equivalent multi-channel model with
real-irreducible group representations, together with a classification of these 
representations based on their further reduction into complex-irreducible
components.

\item Specialized to quadratic assignment where
$a_{ij}=\kappa(x_i,x_j)$ and $y_{ij}=a_{\pi_*(i)\pi_*(j)}+\text{Gaussian
noise}$, our result implies that the
mutual information is related to that in a linear observation model
\[\by_i=\overlap^{1/2}\phi(x_{\pi_*(i)})+\text{Gaussian noise}\]
where $\phi(\cdot)$ is a feature map defined by eigenfunctions of the kernel
$\kappa(\cdot,\cdot)$.
This linear model, although not independent across components
$i=1,\ldots,N$, is well-studied as an oracle model in the literature on
compound decision problems and empirical Bayes estimation
\cite{hannan1955asymptotic,greenshtein2009asymptotic,jiang2009general,polyanskiy2021sharp}. We deduce from results of this literature that in a bounded SNR
regime, if the empirical distribution of $\{x_i\}_{i=1}^N$ converges to a limit
law $\rho$, then the asymptotic mutual information between 
$\pi_*$ and $\{a_{ij},y_{ij}\}_{i<j}$ coincides with
the mutual information in a low-rank matrix estimation model
having i.i.d.\ prior $\rho$ for its signal components.
\end{enumerate}

We present the detailed setting and results of the general model
in Section \ref{sec: general_model},
the specialization to group synchronization in Section
\ref{sec:synchronization}, and the specialization to random
quadratic assignment in Section \ref{sec:quadraticassignment}.

\subsection{Further related literature}

\paragraph{Interpolation methods and overlap concentration.} Gaussian
interpolation techniques for computing free energies in spin glass models
were brought to prominence by Guerra
\cite{guerra2003broken}, and have since been extended
and applied to characterize the fundamental limits of inference
in many Bayesian statistical problems with planted signals, including
\cite{korada2009exact,krzakala2016mutual, dia2016mutual,
el2018estimation,lelarge2019fundamental,barbier2019adaptive,
barbier2019optimal}. In Bayesian inference problems, obtaining a tight upper
bound for the log-partition-function (i.e.\ free energy) is oftentimes more
intricate than the lower bound; this was achieved
in low-rank matrix estimation problems
using an Aizenman-Sims-Starr scheme in \cite{lelarge2019fundamental} and an
adaptive interpolation method in \cite{barbier2019adaptive}. Our proofs build
upon a different method in \cite{el2018estimation} of analyzing the large
deviations of the overlap between a posterior sample and the planted signal,
by bounding the Franz-Parisi potential \cite{franz1997phase},
i.e.\ the free energy restricted to configurations having overlap values in a
narrow range. For the high-temperature region of the classical
Sherrington-Kirkpatrick model, this is also related to the analysis carried out
in \cite[Theorem 13.4.2]{talagrand2011mean}. Large deviations of the overlap
have also been studied recently for low-rank matrix estimation models outside
the replica-symmetric setting, under a mismatched prior and noise distribution,
in \cite{guionnet2023estimating}.

\paragraph{Group synchronization.} Angular synchronization problems over
$\SO(2)$ were introduced in \cite{singer2011angular}, and subsequently
formulated and studied in settings of general groups in
\cite{bandeira2015convex,bandeira2020non}. The specific examples of
$\Z/2\Z$-synchronization
\cite{deshpande2015asymptotic,bandeira2016low,montanari2016semidefinite,javanmard2016phase,fan2021tap,li2022non,celentano2023local,li2023approximate},
angular/phase synchronization
\cite{boumal2016nonconvex,bandeira2017tightness,liu2017estimation,zhong2018near,gao2019multi,gao2022sdp},
and synchronization problems over the orthogonal and symmetric groups
\cite{cucuringu2012sensor,pachauri2013solving,gao2021exact,ling2022near,ling2022improved,zhang2022exact,gao2023optimal,ling2023solving,nguyen2023novel}
have each received substantial attention in their own right. Much of this
literature focuses on the performance of spectral, semidefinite-programming
(SDP), and/or nonconvex optimization methods for estimation, and their
associated guarantees for exact recovery or optimal estimation rates in regimes 
of growing SNR.

Closer to our work are the (mostly non-rigorous) results
of \cite{javanmard2016phase,PWBM2018} which study synchronization problems in
bounded SNR regimes, the former analyzing Bayesian, maximum-likelihood, and SDP
approaches to inference via the cavity method in the $\Z/2\Z$- and
$\U(1)$-synchronization models, and the latter introducing an Approximate
Message Passing algorithm for Bayes-optimal inference in general
synchronization models with multiple observation channels corresponding to
distinct complex-irreducible group representations. Our results formalize some
of the findings of this latter work \cite{PWBM2018} in a similar model having
real observations, and are also complementary to analyses of the free energy for
general synchronization problems that were carried out
in \cite{perry2016optimality} using a second-moment-method approach.

\paragraph{Quadratic assignment.} The quadratic assignment problem was 
introduced in \cite{koopmans1957assignment}, and its behavior on several
models of random data has been investigated in
\cite{burkard1984quadratic,frenk1985asymptotic,rhee1991stochastic}.
Statistical applications of quadratic assignment and convex relaxations thereof
for estimating latent vertex matchings between random graphs have been studied
more recently in
\cite{zaslavskiy2008path,aflalo2015convex,lyzinski2015graph,fan2023spectralI,fan2023spectralII}, in the context of a broader literature on algorithms and
fundamental limits of inference for random graph matching problems
\cite{cullina2017exact,
cullina2020partial,ganassali2020tree,ding2021efficient,ganassali2022sharp,ganassali2022statistical,wu2022settling,hall2023partial,ding2023matching,mao2023exact,mao2023random}. We study in this work a quadratic assignment problem with
random data matrices of low effective rank, bearing similarity to 
matching problems between random geometric graphs recently considered in
\cite{wang2022random,gong2024umeyama,liu2024random}, and to analyses of related 
linear matching problems in
\cite{collier2016minimax,dai2019database,kunisky2022strong}. Our analyses here
pertain to a bounded SNR regime where weak recovery of the latent
permutation/matching may not be possible, 
and where we instead show an exact
asymptotic equivalence between the signal-observation mutual information
with that in a low-rank matrix estimation model with i.i.d.\ signal prior.

\section{General model and results}\label{sec: general_model}

Consider a compact group $\cG_N$, a ($N$-dependent) featurization
$\phi:\cG_N \to \cH^N$ with ``feature'' space $\cH$, and a bilinear map
$\bullet:\cH \times \cH \to \cK$ with ``observation'' space $\cK$,
where $\cH,\cK$ are finite-dimensional real vector spaces endowed with the
inner-products $\langle \cdot,\cdot\rangle_\cH$ and
$\langle \cdot,\cdot\rangle_\cK$.
We study a general observation model in which, for an unknown parameter
$G_* \in \cG_N$ of interest, we observe
\begin{equation}\label{eq:quadraticmodel}
\by_{ij}=\phi(G_*)_i \bullet \phi(G_*)_j+\sqrt{N}\,\bz_{ij}
\text{ for each } 1 \leq i<j \leq N.
\end{equation}
Here $\{\bz_{ij}\}_{i<j}$ are
i.i.d.\ standard Gaussian noise vectors in $\cK$ (i.e.\ having i.i.d.\
$\cN(0,1)$ components in any orthonormal basis of
$\cK$).\footnote{We fix the noise standard deviation in
(\ref{eq:quadraticmodel}) as $\sqrt{N}$ without loss of generality, absorbing
additional problem scalings into the definition of the $N$-dependent
featurization $\phi$.}

In this work, we will focus on a Bayesian setting where $G_*$ has
Haar-uniform prior $G_* \sim \Haar(\cG_N)$. Denote the combined observations as
$Y=\{\by_{ij}\}_{i<j}$. Bayes-optimal inference for $G_*$ is then based on the
posterior density (with respect to Haar measure)
\begin{align}
\label{eq:posterior-general-model}
p(G \mid Y) \propto \exp \left(-\frac{1}{2N}\sum_{1 \leq i<j \leq N}
\|\by_{ij}-\phi(G)_i \bullet \phi(G)_j\|_{\cK}^2\right)
\propto \exp H(G;Y),
\end{align}
where we expand the square and define the Hamiltonian
\begin{equation}\label{eq:hamiltonian-general-model}
H(G;Y)=-\frac{1}{2N}
\sum_{1 \leq i<j \leq N} \|\phi(G)_i \bullet \phi(G)_j\|_{\cK}^2
+\frac{1}{N}\sum_{1 \leq i<j \leq N} \langle \phi(G)_i \bullet \phi(G)_j,\by_{ij}
\rangle_{\cK}.
\end{equation}
We denote its associated free energy
\begin{equation}\label{eq:generalfreeenergy}
\cF_N=\frac{1}{N}\,\E_{G_*,Z} \log \E_G \exp H(G;Y).
\end{equation}
Here, $\E_G$ is the expectation over a uniform element $G \sim \Haar(\cG_N)$,
and $\E_{G_*,Z}$ is over the independent signal $G_* \sim \Haar(\cG_N)$ and
Gaussian noise vectors $Z=\{\bz_{ij}\}_{i<j}$ which define $Y$. 

We assume a model structure in which there exists a complementary bilinear map
$\otimes:\cH \times \cH \to \cL$ to an ``overlap'' space $\cL$,
such that $\cG_N$, $\phi$, $\bullet$, and $\otimes$ satisfy the following
properties.

\begin{assumption}\label{assumption:general-model}
\begin{enumerate}[(a)]
\item $\cG_N$ is a compact group, $(\cH,\langle\cdot,\cdot\rangle_\cH)$
is a finite-dimensional (real)
inner-product space, and $\phi:\cG_N \to \cH^N$ is a continuous map.
\item $\bullet:\cH \times \cH \to \cK$ and $\otimes:\cH \times \cH \to \cL$ are
bilinear maps from $\cH \times \cH$ to two finite-dimensional (real)
inner-product spaces $(\cK,\langle\cdot,\cdot\rangle_\cK)$ and
$(\cL,\langle\cdot,\cdot\rangle_\cL)$,
satisfying for all $\ab,\bb,\ab',\bb' \in \cH$
the compatibility relation
\begin{equation}\label{eq:compatibility}
\langle \ab \bullet \bb,\ab' \bullet \bb' \rangle_{\cK}
=\langle \ab \otimes \ab',\bb \otimes \bb' \rangle_{\cL}.
\end{equation}
\item Let $B(\cH)$ be the space of linear operators on $\cH$, and define
an inclusion map $\iota:\cL \to B(\cH)$ by
\begin{equation}\label{eq:inclusion}
\langle \ab,\iota(\overlap)\bb \rangle_{\cH}
=\langle \overlap,\ab \otimes \bb \rangle_{\cL}
\text{ for all } \overlap \in \cL \text{ and } \ab,\bb \in \cH.
\end{equation}
Then $\iota$ is injective. Furthermore,
corresponding to any element $\overlap \in \cL$,
there exist elements $|\overlap|,|\overlap^\top| \in
\cL$ such that $\iota(|\overlap|),\iota(|\overlap^\top|)$ are both symmetric
positive-definite, and
\begin{equation}\label{eq:operatorabsolutevalue}
\iota(|\overlap|)=\big(\iota(\overlap)^\top \iota(\overlap)\big)^{1/2},
\qquad
\iota(|\overlap^\top|)=\big(\iota(\overlap)\iota(\overlap)^\top\big)^{1/2},
\qquad \|\overlap\|_\cL=\big\||\overlap|\big\|_\cL
=\big\||\overlap^\top|\big\|_\cL.
\end{equation}
\item Define an overlap map $Q:\cG_N \times \cG_N \to \cL$ by
\begin{equation}\label{eq:overlap}
Q(G,H)=\frac{1}{N}\sum_{i=1}^N \phi(G)_i \otimes \phi(H)_i.
\end{equation}
Then $Q(\cdot,\cdot)$ satisfies the group symmetry
$Q(G,H)=Q(H^{-1}G,\iden)$ for any $G,H \in \cG_N$,
where $\iden$ the identity element of $\cG_N$.
\end{enumerate}
\end{assumption}

We will illustrate how the group synchronization and quadratic assignment
applications fit into this structure in Sections \ref{sec:synchronization}
and \ref{sec:quadraticassignment} to follow. For now, let us observe that
under parts (b) and (d) of this assumption, applying the model definition
\eqref{eq:quadraticmodel}, the Hamiltonian $H(G;Y)$ in
\eqref{eq:hamiltonian-general-model} is approximately a linear combination of
the squared overlaps
\begin{align*}
N\|Q(G,G)\|_\cL^2&=\frac{1}{N}\sum_{i,j=1}^N
\langle \phi(G)_i \bullet \phi(G)_j,
\phi(G)_i \bullet \phi(G)_j \rangle_\cK,\\
N\|Q(G,G_*)\|_\cL^2&=
\frac{1}{N}\sum_{i,j=1}^N \langle \phi(G)_i \bullet \phi(G)_j,
\phi(G_*)_i \bullet \phi(G_*)_j \rangle_\cK,
\end{align*}
and a $\cG_N$-indexed centered Gaussian process $Z(G)$ with covariance kernel
\[\E[Z(G)Z(G')]=N\|Q(G,G')\|_\cL^2
=\frac{1}{N}\sum_{i,j=1}^N \langle \phi(G)_i \bullet \phi(G)_j,
\phi(G')_i \bullet \phi(G')_j \rangle_\cK.\]
This structure mirrors that of the Bayes posterior law in low-rank matrix
estimation problems.
The error of this approximation for the Hamiltonian is\footnote{Here
and throughout, $O(f(N))$ denotes an error bounded in magnitude by $Cf(N)$ for
an absolute constant $C>0$.}
$O(K(\cG_N))$ where
\begin{equation}\label{eq:KGN}
K(\cG_N)=\sup_{G,G' \in \cG_N} \frac{1}{N}\sum_{i=1}^N \|\phi(G)_i \otimes
\phi(G')_i\|_\cL^2
=\sup_{G,G' \in \cG_N} \frac{1}{N}\sum_{i=1}^N \langle \phi(G)_i \bullet
\phi(G)_i,\phi(G')_i \bullet \phi(G')_i \rangle_\cK,
\end{equation}
due to the removal of diagonal terms $i=j$ from the above squared overlap
expressions. The group symmetry of $Q(\cdot,\cdot)$ in part (d) will ensure that the 
law over $Y$ of the $\cG_N$-valued process $\{H(G;Y)\}_{G \in \cG_N}$ is, up to
this $O(K(\cG_N))$ discrepancy, independent of $G_*$. Finally, the inclusion
map $\iota(\cdot)$ in part (c) identifies overlaps
$\overlap \in \cL$ with linear operators $\iota(\overlap)$ on $\cH$, and we
will write the shorthands
\begin{equation}\label{eq:overlapoperator}
\overlap \ab:=\iota(\overlap)\ab, \qquad
\overlap^{1/2} \ab:=\iota(\overlap)^{1/2}\ab,
\end{equation}
the latter being well-defined when $\iota(\overlap)$ is symmetric
positive-semidefinite.

Under this assumption, the main result of this section is a general statement
relating the free energy $\cF_N$ to the following model with linear observations
of $\phi(G_*)$: Let
\begin{equation}\label{eq:Qdef}
\cQ=\Big\{\overlap \in \cL:\iota(\overlap)
\text{ is symmetric positive-semidefinite in } B(\cH)\Big\} \subset \cL.
\end{equation}
Fixing any $\overlap \in \cQ$, consider the linear observation model
with observations
\begin{equation}\label{eq:linearmodel}
\by_i=\overlap^{1/2} \phi(G_*)_i+\bz_i \text{ for each } i=1,\ldots,N,
\end{equation}
where $\overlap^{1/2}$ is identified as a linear operator on the feature space
$\cH$ via \eqref{eq:overlapoperator}, and
$\{\bz_i\}_{i=1}^N$ are i.i.d.\ standard Gaussian noise vectors in $\cH$.
Define a potential function $\Psi_N:\cQ \to \R$ by
\begin{align}
\label{eq:RS-potential-general-model}
\Psi_N(\overlap)=-\frac{1}{4}\|\overlap\|_\cL^2
-\frac{1}{2}\langle \overlap,Q(\iden,\iden) \rangle_\cL
+\frac{1}{N}\,\E_{G_*,Z} \log \E_G \exp\inparen{
N\langle \overlap, Q(G,G_*) \rangle_{\cL}
+\sum_{i=1}^N \langle \overlap^{1/2} \phi(G)_i,\bz_i \rangle_\cH}
\end{align}
where here $\E_Z$ is the expectation over $Z=\{\bz_i\}_{i=1}^N$.
It is readily checked (c.f.\ Appendix \ref{subsec:MI}) that
the signal-observation mutual information in the quadratic
model (\ref{eq:quadraticmodel}) with observations $Y=(\by_{ij})_{i<j}$ is given by
\begin{equation}\label{eq:quadraticMI}
\frac{1}{N}\,I(G_*,Y):=\frac{1}{N}\,\E_{G_*,Z}
\,\log \frac{p(G_*,Y)}{p(G_*)p(Y)}=\frac{1}{4}\|Q(\iden,\iden)\|_\cL^2-\cF_N
+O\left(\frac{K(\cG_N)}{N}\right),
\end{equation}
and the mutual information in the linear model
(\ref{eq:linearmodel}) with observations $Y_\lin=\{\by_i\}_{i=1}^N$ is given by
\begin{equation}\label{eq:linearMI}
\frac{1}{N}\,i(G_*,Y_\lin)=-\frac{1}{4}\|\overlap\|_\cL^2
+\frac{1}{2}\langle \overlap,
Q(\iden,\iden) \rangle_{\cL}-\Psi_N(\overlap).
\end{equation}

Our main result of this section is the following approximation of the free
energy $\cF_N$ in terms of $\Psi_N$.
This then provides a direct relation between the mutual
informations $\frac{1}{N}I(G_*,Y)$ and $\frac{1}{N}i(G_*,Y_\lin)$ via
(\ref{eq:quadraticMI}) and (\ref{eq:linearMI}), which we will spell out in the
later applications of interest.

\begin{theorem}\label{thm:free-energy-general-model}
Denote
\[\image(Q)=\{Q(G,G'):G,G' \in \cG_N\} \subset \cL,\]
let $D(\cG_N)=\max\{\|\overlap\|_\cL:\overlap \in \image(Q)\}$,
and let $L(\epsilon;\cG_N)$ be the metric entropy of $\image(Q)$,
i.e.\ the log-cardinality of the smallest $\epsilon$-cover of $\image(Q)$ in the
norm $\|\cdot\|_\cL$.
Under Assumption~\ref{assumption:general-model}, there exists an absolute
constant $C>0$ such that for any $\epsilon>0$,
    \begin{align*}
\Big|\cF_N-\sup_{\overlap \in \cQ} \Psi_N(\overlap)\Big|
 \le C\bigg(D(\cG_N)\sqrt{\frac{L(\sqrt{\epsilon};\cG_N)}{N}}
+\frac{K(\cG_N)+L(\sqrt{\epsilon};\cG_N)}{N}+\epsilon\bigg).
    \end{align*}
\end{theorem}

Here, $K(\cG_N)$, $D(\cG_N)$, and $L(\sqrt{\epsilon};\cG_N)$ are
finite by compactness of
$\cG_N$ and continuity of $\phi$, and these will all be of constant order
for any fixed constant $\epsilon>0$ in our applications to follow.

\paragraph{Overlap concentration.} Denote by $\langle f(G) \rangle=\E[f(G) \mid
Y]$ the posterior expectation given $Y=\{\by_{ij}\}_{i<j}$
in the quadratic model (\ref{eq:quadraticmodel}). An extension of our
proof of Theorem \ref{thm:free-energy-general-model} will establish that for $G$
sampled from this posterior law, the overlap $Q(G,G_*)$
concentrates on a set defined by near-maximizers of the potential
$\Psi_N(\overlap)$. We give a general statement of this result here, and we will
specialize this to a more interpretable statement in the group synchronization
application of Section \ref{sec:synchronization} to follow.

For any $A \in B(\cH)$, let $p_A$ denote the marginal density of
$Y_\lin=\{\by_i\}_{i=1}^N$ in a linear observation model
$\by_i=A\phi(G_*)_i+\bz_i$ for $i=1,\ldots,N$, similar to \eqref{eq:linearmodel}.
For $\overlapm \in \cL$ such that $\iota(\overlapm) \in B(\cH)$ has
singular value decomposition $\iota(\overlapm)=UDV^\top$, denote
\begin{equation}\label{eq:mUV}
\overlapm_U=D^{1/2}U^\top \in B(\cH), \qquad
\overlapm_V=D^{1/2}V^\top \in B(\cH),
\end{equation}
and recall from Assumption \ref{assumption:general-model}(c) that there exists
$|\overlapm| \in \cQ$ for which
$\iota(|\overlapm|)=(\iota(\overlapm)^\top \iota(\overlapm))^{1/2}
=\overlapm_V^\top\overlapm_V$. Set\footnote{Here, for any $\overlapm \in \cL$,
the element $|\overlapm|$ is unique by the assumed
injectivity of $\iota$. It may also be checked that $D_{\mathrm{KL}}(p_{\overlapm_V} \| p_{\overlapm_U})$ has the same value for any singular value
decomposition $UDV^\top$ of $\iota(\overlapm)$, so $\cL_*(\epsilon)$ is
well-defined.}
\begin{align}
\cL_*(\epsilon)=\inbraces{\overlapm \in \cL:
\frac{1}{N}D_{\mathrm{KL}}(p_{\overlapm_V} \| p_{\overlapm_U}) \leq
\epsilon \text{ and } \sup_{\overlap \in \cQ} \Psi_N(\overlap)
-\Psi_N(|\overlapm|) \leq \epsilon}.
\label{eq:overlap_concentration_set}
\end{align}
Intuitively, any $\overlapm \in \cL_*(\epsilon)$ is such that $|\overlapm|$ is a
near-maximizer of $\Psi_N(\cdot)$, and the condition
$N^{-1}D_{\mathrm{KL}}(p_{\overlapm_V}\|p_{\overlapm_U}) \approx 0$ captures a
class of overlaps $\overlapm$ that are (nearly) equivalent to
$|\overlapm|$ under the group symmetry of the model.

\begin{corollary}\label{corollary:overlap_concentration_general_model}
In the setting of Theorem \ref{thm:free-energy-general-model},
there exist absolute constants $C_0,C,c>0$ such that if
\begin{equation}\label{eq:epsiloncondition}
\epsilon>C_0\bigg(D(\cG_N)\sqrt{\frac{L(\sqrt{\epsilon};\cG_N)}{N}}
+\frac{K(\cG_N)+L(\sqrt{\epsilon};\cG_N)}{N}\bigg)
\end{equation}
then
\begin{align}
\E_{G_*,Z} \Big\langle\mathbf{1}\inbraces{ \norm{Q(G,G_*) -
\cL_*(\epsilon)}_\cL^2>\epsilon}\Big\rangle \leq
C\exp\left({-}cN\min\left(\epsilon,\frac{\epsilon^2}{D(\cG_N)^2}\right)\right)
\label{eq:overlap_concentration_general_model}
\end{align}
where $\norm{Q(G,G_*) - \cL_*(\epsilon)}_\cL := \inf \inbraces{ \norm{Q(G,G_*) -
\overlapm} : \overlapm \in \cL_*(\epsilon)}$.
\end{corollary}

The proofs of Theorem \ref{thm:free-energy-general-model} and Corollary
\ref{corollary:overlap_concentration_general_model} are given in Appendix 
\ref{sec:pf-general_model}. We prove both the lower and upper bounds for
$\cF_N$ in Theorem \ref{thm:free-energy-general-model}
using an interpolation technique, the
upper bound applying a method of \cite{el2018estimation} to perform
the interpolation on the Franz-Parisi potential at each fixed overlap.
Corollary \ref{corollary:overlap_concentration_general_model} then
follows by applying similar arguments to bound a restriction of the free energy.

\begin{remark}\label{remark:MTP}
We thank a reviewer for bringing to our attention that the
model \eqref{eq:quadraticmodel} can be regarded as a special case of the matrix
tensor product (MTP) model in \cite{reeves2020information}, which in turn
extends a multiview matrix model studied in \cite{barbier2020information} and 
was further generalized to higher-order tensors in \cite{chen2022statistical}.
Our model \eqref{eq:quadraticmodel} may be written in MTP form via the
following mapping: In the MTP model one observes
    \begin{equation}\label{eq:MTP}
        y_{ij} = B\inparen{ x_{*,i} \boxtimes x_{*,j}  } + \sqrt{N} z_{ij} \in
\R^m, \qquad i,j = 1,\dots,N,
    \end{equation}
    where $(x_{*,1},\dots, x_{*,N}) \in \inparen{\RR^{d}}^N$ are drawn from some
distribution $P_N$ over $\inparen{\RR^{d}}^N$; $B \in \RR^{m \times d^2}$ is a
known coupling matrix; $\boxtimes$ denotes the Kronecker product; and $z_{ij} \sim
\cN(0, I_m)$ are independent noise variables. In \eqref{eq:quadraticmodel}, set $\cH = \RR^d$ and
$\cK = \RR^m$ and equip both spaces with the standard Euclidean inner product.
Identify the feature map $\phi : \cG_N \rightarrow \inparen{\RR^d}^N$ with
$\inparen{x_{*,1},\dots,x_{*,N}} = \phi(G_*)$ and identify the bilinear map
$\bullet : \RR^d \times \RR^d \rightarrow \RR^m$ with $\ab \bullet \bb =
B\inparen{ \ab \boxtimes \bb  }$ (noting that any such bilinear map
admits a representation of this form). Then set $\cL = \RR^{d \times d} \cong
\R^{d^2}$ equipped with the inner product
    \begin{align*}
        \langle \ub,\vb \rangle_\cL = \tr B^\top B \inparen{\ub \boxtimes \vb}.
    \end{align*}
    The complementary bilinear map $\otimes : \RR^d \times \RR^d \rightarrow
\RR^{d \times d}$ is defined by $\ab \otimes \bb = \ab \bb^\top$. One checks
that $\inangle{\ab \bullet \bb, \ab' \bullet \bb'}_{\RR^m} = \inangle{  \ab
\otimes \ab', \bb \otimes \bb'  }_\cL$. Furthermore, the inclusion map $\iota :
\RR^{d \times d} \rightarrow \RR^{d \times d}$ of \eqref{eq:inclusion} satisfies
$\inangle{ \ab,
\iota(\overlap) \bb  }_{\RR^d} = \tr B^\top B \inparen{  \overlap \boxtimes \ab
\bb^\top  }$ so that $\iota(\overlap)$ may be identified with the operator
$T(\overlap)$ defined in \cite[Lemma 1]{reeves2020information}. This maps our
model \eqref{eq:quadraticmodel} to the MTP model.

The main result of \cite[Theorem 1]{reeves2020information} provides an approximation to the
free energy of the MTP model in a setting where $P_N$ is a product
distribution, but it was also shown in \cite[Theorem 5]{reeves2020information}
and \cite{chen2022statistical} that approximations of this type do not
require this product structure and hold more generally as long as the free
energy in the linear observation model associated to $P_N$ concentrates. We
study a specific such setting where $P_N$ arises as the image of a
uniform element in a compact group under a continuous feature map. Our proof of
Theorem \ref{thm:free-energy-general-model} uses a different approach from
the methods of \cite{reeves2020information} and \cite{chen2022statistical},
and yields a new quantitative bound for the free energy
approximation in terms of the metric entropy of the overlap space.
On the other hand, the analyses of \cite{reeves2020information} pertain to a
more general setting in which the analogue of the bilinear form
$\inangle{\cdot, \cdot}_\cL$ in the MTP model may be indefinite, and the
approximating free energy takes a variational $\sup$-$\inf$ form over
two variables. Under our restriction in
Assumption \ref{assumption:general-model} that $\inangle{\cdot, \cdot}_\cL$
is positive-semidefinite,
our resulting free energy approximation (\ref{eq:RS-potential-general-model})
coincides with that of \cite[Proposition 6]{reeves2020information} in
the positive-semidefinite setting.
\end{remark}

\section{Group synchronization}\label{sec:synchronization}

As a first application of the results in Section \ref{sec: general_model},
we consider a multi-channel group synchronization
model, which is a real analogue of the model studied in \cite{PWBM2018}.
Let $\cG$ be a compact group (fixed and not depending on $N$),
and let $\phi_\ell:\cG \to \R^{k_\ell \times k_\ell}$ for $\ell=1,\ldots,L$
be real orthogonal representations of $\cG$. Throughout, corresponding to
$\gb,\gb',\gb_* \in \cG$, we write the abbreviations
\[\gb_\ell=\phi_\ell(\gb), \qquad \gb_\ell'=\phi_\ell(\gb'), \qquad
\gb_{*\ell}=\phi_\ell(\gb_*).\]

Let $G_*=(\gb_*^{(1)},\ldots,\gb_*^{(N)})$ be an unknown parameter vector of
interest in the product space $\cG^N$, with prior distribution
$\{\bg_*^{(i)}\}_{i=1}^N \overset{iid}{\sim} \Haar(\cG)$.\footnote{For
simplicity of the later notation, in this group synchronization model we will
use superscripts for the sample index $i \in [N]$ and subscripts for the
channel index $\ell \in [L]$.} Consider the observations
\begin{align}
\label{eq:group-synchronization-multi-representation-model}
\begin{cases}
    \yb_1^{(ij)} = \sqrt{\lambda_1}\,\gb_{*1}^{(i)}\gb_{*1}^{(j)\top}
+ \sqrt{N}\,\zb_1^{(ij)} \in \R^{k_1 \times k_1}\\
    \hspace{1in}\vdots \\
    \yb_L^{(ij)} = \sqrt{\lambda_L}\,\gb_{*L}^{(i)}\gb_{*L}^{(j)\top}
+ \sqrt{N}\,\zb_L^{(ij)} \in \R^{k_L \times k_L}
\end{cases} \text{ for all } 1 \leq i<j \leq N
\end{align}
where $\lambda_\ell>0$ are fixed and known signal-to-noise parameters,
and $\{\zb_\ell^{(ij)}\}_{1 \leq i<j \leq N,\,1 \leq \ell \leq L}$
are noise matrices with i.i.d.\ $\cN(0,1)$ entries, independent of each other
and of $G_*$.

We note that any observation model
\eqref{eq:group-synchronization-multi-representation-model}
is equivalent to such a model in a ``canonical'' form where the representations
$\phi_1,\ldots,\phi_L$ are real-irreducible, distinct, and non-trivial;
this canonical form may be a multi-channel model even if the original problem
consists of a single channel $L=1$.
We explain this reduction in Appendix \ref{appendix:GSreduction},
where we also review some relevant background and terminology pertaining to
group representations.

This model falls into the general framework described in Section~\ref{sec:
general_model}. Here, $\cG_N \equiv \cG^N$ is the $N$-fold product of $\cG$.
For $G=(\gb^{(1)},\ldots,\gb^{(N)}) \in \cG_N$,
the feature map $\phi:\cG_N \to \cH^N$ is separable across coordinates
$i=1,\ldots,N$, with components
\begin{equation}\label{eq:GSphi}
\phi(G)_i=(\gb_1^{(i)},\ldots,\gb_L^{(i)}) \in \cH
\equiv \prod_{\ell=1}^L \R^{k_\ell \times k_\ell}.
\end{equation}
We identify the observation and overlap spaces also as
$\cK=\cL=\cH=\prod_{\ell=1}^L \R^{k_\ell \times k_\ell}$, equipped with the
usual Euclidean inner-products
\[\langle \ab,\bb \rangle_{\cH}=\langle \ab,\bb \rangle_{\cK}
=\langle \ab,\bb \rangle_{\cL}=\sum_{\ell=1}^L \tr \ab_\ell^\top \bb_\ell.\]
The pair of bilinear maps $\bullet:\cH \times \cH \to \cK$ and
$\otimes:\cH \times \cH \to \cL$ and the inclusion map $\iota:\cL \to B(\cH)$
are then defined as
\begin{equation}\label{eq:GSbilinearmaps}
\ab \bullet \bb=\Big(\sqrt{\lambda_\ell}\,\ab_\ell\bb_\ell^\top \Big)_{\ell=1}^L,
\qquad \ab \otimes \bb=\Big(\sqrt{\lambda_\ell}\,\ab_\ell^\top\bb_\ell\Big)_{\ell=1}^L,
\qquad \iota(\overlap)\ab=\Big(\sqrt{\lambda_\ell}
\,\ab_\ell\overlap_\ell^\top\Big)_{\ell=1}^L.
\end{equation}
We will check in the proof of Theorem
\ref{theorem:free-energy-group-syn-multi-representation} below that the
structure of Assumption \ref{assumption:general-model} indeed holds under these
definitions.

We write $\Sym^{k \times k},\Sym_{\succeq 0}^{k \times k}$ for the spaces of $k
\times k$ symmetric and symmetric-positive-semidefinite
matrices, respectively, and abbreviate
\[
\qquad \Sym=\prod_{\ell=1}^L \Sym^{k_\ell \times k_\ell},
\qquad \Sym_{\succeq 0}=\prod_{\ell=1}^L \Sym_{\succeq 0}^{k_\ell \times
k_\ell},\]
\[\gb\overlap\gb'=\Big(\gb_\ell\overlap_\ell\gb_\ell'\Big)_{\ell=1}^L
\in \cL
\qquad \text{ for any } \gb,\gb' \in \cG \text{ and }\overlap \in \Sym.\]

\subsection{Asymptotic mutual information and MMSE}

Let $Y=\{\yb_\ell^{(ij)}\}$ be the collection of all observations,
and let $\langle f(G) \rangle=\langle
f(\gb^{(1)},\ldots,\gb^{(N)})\rangle$ denote the average under the posterior
law of $G=(\gb^{(1)},\ldots,\gb^{(N)}) \in \cG^N$ given $Y$. Let
$I(G_*,Y)$ be the mutual information between
$G_*=(\bg_*^{(1)},\ldots,\bg_*^{(N)})$ and $Y$, and let
\begin{equation}\label{eq:GSMMSE}
\mmse_\ell=\frac{1}{{N \choose 2}} \sum_{1 \le i
< j \le N} \E \|\gb_{*\ell}^{(i)\top}\gb_{*\ell}^{(j)} - \langle
\gb_\ell^{(i)\top} \gb_\ell^{(j)} \rangle \|_F^2 \qquad \text{for each } \ell=1,\ldots,L
\end{equation}
be the Bayes-optimal minimum mean-squared-error (MMSE) for estimating
$\{\gb_{*\ell}^{(i)\top}\gb_{*\ell}^{(j)}\}_{1 \leq i<j \leq N}$ in the
$\ell^\text{th}$ channel.

Define the replica potential $\Psi_\gs:\Sym_{\succeq 0} \to \R$ by
\begin{align}\label{eq:Psigsformula}
\Psi_\gs(\overlap) = -\frac{1}{4} \sum_{\ell=1}^L \lambda_\ell
\|\overlap_\ell\|_F^2 - \frac{1}{2} \sum_{\ell=1}^L \lambda_\ell \tr
\overlap_\ell + \E_{\gb_*,\zb} \log \E_{\gb} \exp \inparen{\sum_{\ell=1}^L 
\lambda_\ell \tr \overlap_\ell \gb_\ell^\top\gb_{*\ell}
+\sqrt{\lambda_\ell} \tr \overlap_\ell^{1/2}\gb_\ell^\top\zb_\ell} 
\end{align}
where $\E_\gb$ is the expectation over a single uniformly distributed
group element
$\gb \sim \Haar(\cG)$, and $\E_{\gb_*,\zb}$ is over an independent element
$\gb_* \sim \Haar(\cG)$ and Gaussian noise
$\zb=(\zb_1,\ldots,\zb_L) \in \prod_{\ell=1}^L \R^{k_\ell \times k_\ell}$
with i.i.d.\ $\cN(0,1)$ entries.
By invariance of Haar measure and invariance in law of each $\zb_\ell$
under multiplication by orthogonal matrices, it may be checked that
$\Psi_\gs$ has the group symmetry
\begin{equation}\label{eq:Psisymmetry}
\Psi_\gs(\overlap)=\Psi_\gs(\gb\overlap\gb^{-1}) \text{ for all }
\gb,\gb' \in \cG \text{ and } \overlap \in \Sym_{\succeq 0}.
\end{equation}
In particular, the set of maximizers of $\Psi_\gs$
is closed under the mapping $\overlap
\mapsto \gb\overlap\gb^{-1}$ for all $\gb \in \cG$.
It is also direct to check, similarly to \eqref{eq:linearMI},
that $\Psi_\gs$ satisfies
\begin{equation}\label{eq:scalarMI}
i(\gb_*,\yb)=\sum_{\ell=1}^L \Big({-}\frac{\lambda_\ell}{4}\|\overlap_\ell\|_F^2
+\frac{\lambda_\ell}{2} \tr \overlap_\ell\Big)-\Psi_\gs(\overlap)
\end{equation}
where $i(\gb_*,\yb)$ is the mutual information between the signal
$\gb_* \in \cG$
and observations $\yb=(\yb_1,\ldots,\yb_L)$ in a $\overlap$-dependent
``single-sample'' model
\begin{equation}\label{eq:GSscalarmodel}
\begin{cases} \yb_1=\sqrt{\lambda_1}\,\gb_{*1}\overlap_1^{1/2}+\zb_1 \in \R^{k_1 \times k_1}\\
\hspace{1in}\vdots\\
\yb_L=\sqrt{\lambda_L}\,\gb_{*L}\overlap_L^{1/2}+\zb_L \in \R^{k_L \times k_L}.
\end{cases}
\end{equation}

The following is an application of Theorem~\ref{thm:free-energy-general-model}
and Corollary \ref{corollary:overlap_concentration_general_model},
characterizing the asymptotic mutual information, per-channel MMSE, and
concentration of the posterior overlap with the true signal
in this group synchronization model as $N \to \infty$,
in terms of a maximization of the above replica potential.

\begin{theorem}
\label{theorem:free-energy-group-syn-multi-representation}
Suppose the group $\cG$, representations $\phi_1,\ldots,\phi_L$, and signal
strengths $\lambda_1,\ldots,\lambda_L>0$ are fixed, as $N \to \infty$.
\begin{enumerate}[(a)]
    \item The signal-observation mutual information $I(G_*,Y)$ in the
model~\eqref{eq:group-synchronization-multi-representation-model} satisfies
    \begin{align}
    \label{eq:mutual-information-group-syn-multi-rep}
        \lim_{N \to \infty}\frac{1}{N} I(G_*,Y) = \frac{1}{4}\sum_{\ell=1}^L
\lambda_\ell k_\ell - \sup_{\overlap \in \Sym_{\succeq 0}} \Psi_\gs(\overlap).
    \end{align}

    \item Fixing any $\ell \in \{1,\ldots,L\}$ and positive values
$\{\lambda_\ell'\}_{\ell' \neq \ell}$, set
\[D=\Big\{\lambda_\ell>0:\lambda_\ell \mapsto
\sup_{\overlap \in \Sym_{\succeq 0}} \Psi_\gs(\overlap) \text{ is
differentiable at } \lambda_\ell\Big\}.\]
Then $D$ has full Lebesgue measure in $(0,\infty)$. We have
$\lambda_\ell \in D$ if and only if all maximizers $\overlap_*$ of
$\Psi_\gs(\overlap)$ have the same $\ell$-th component Frobenius norm
$\|\overlap_{*\ell}\|_F$, in which case
    \begin{align}
        \lim_{N \to \infty} \mmse_\ell = k_\ell-\|\overlap_{*\ell}\|_F^2.
    \end{align}

\item Denote
\begin{equation}\label{eq:GSoverlapconcset}
\cL_{*,\gs}=\left\{\gb\overlap_*:
\overlap_* \in \argmax_{\overlap \in \Sym_{\succeq 0}} \Psi_\gs(\overlap) \text{ and }
\gb \in \cG\right\} \subset \cL
\end{equation}
and define a neighborhood $\cL_{*,\gs}(\epsilon)=
\{\overlapm \in \cL:\inf_{\overlapm_*
\in \cL_{*,\gs}} \sum_{\ell=1}^L \lambda_\ell
\|\overlapm_\ell-\overlapm_{*\ell}\|_F^2<\epsilon\}$. Then for any $\epsilon>0$,
there exist constants $C,c>0$ depending only on
$\cG,\phi_1,\ldots,\phi_L,\epsilon$ such that
\[\E \left\langle \mathbf{1}\left\{\left(\frac{1}{N}\sum_{i=1}^N
\gb_\ell^{(i)\top} \gb_{*\ell}^{(i)}\right)_{\ell=1}^L \notin
\cL_{*,\gs}(\epsilon) \right\}\right\rangle \leq Ce^{-cN}.\]
\end{enumerate}
\end{theorem}

We remark that the parameter $G_*=(\gb_*^{(1)},\ldots,\gb_*^{(N)})$ in this
model is identifiable only up to a global rotation
$(\gb_*^{(1)},\ldots,\gb_*^{(N)}) \mapsto
(\gb_*^{(1)}\gb,\ldots,\gb_*^{(N)}\gb)$ by any single group element
$\gb \in \cG$, and the posterior law is invariant under this
transformation. The above set $\cL_{*,\gs}$
may be understood as the set of overlaps that are equivalent to a global
maximizer of $\Psi_\gs$ up to this group invariance of the posterior law, and
part (c) of this theorem shows that the overlap of a posterior sample $G$
with the true signal $G_*$ concentrates near this set
$\cL_{*,\gs}$ in the $N \to \infty$ limit.

\subsection{Critical points and algorithmic phase transition}\label{sec:algphasetransition}

Theorem \ref{theorem:free-energy-group-syn-multi-representation} implies that
the signal-observation mutual information and information-theoretic MMSE
are governed by the global maximizer(s) of the replica potential $\Psi_\gs$.

In contrast, the local optimality of $\overlap=\bzero$ is conjectured to govern
the feasibility of computationally-efficient weak signal recovery 
(i.e.\ of attaining non-zero asymptotic overlap $\lim_{N \to \infty}
\frac{1}{N}\sum_{i=1}^N \widehat{\bg}_\ell^{(i)\top} \gb_{*\ell}^{(i)}$ for some channel
$\ell \in \{1,\ldots,L\}$ by a polynomial-time estimator $\widehat{\bg}$).
In particular, Approximate Message Passing (AMP) algorithms of the form
developed in \cite{PWBM2018} are expected to achieve weak signal recovery from a
random initialization whenever $\overlap=\bzero$ \emph{is not} a local
maximizer of $\Psi_\gs$, and conversely it is conjectured that no
polynomial-time algorithm can achieve weak signal recovery when
$\overlap=\bzero$ is a local maximizer of $\Psi_\gs$. We refer to
\cite{lesieur2017constrained,lelarge2019fundamental} for discussion of
this conjecture in related low-rank matrix estimation problems.

In this section, for general synchronization problems, we provide a
criterion for the local optimality of $\overlap=\bzero$ for maximizing
$\Psi_\gs$, in terms of the signal strengths and classifications
of the real-irreducible components of the observation channels.

The following proposition first derives general forms for the gradient and
Hessian of $\Psi_\gs$.
We write $\nabla \Psi_\gs(\bq)$ and $\nabla^2 \Psi_\gs(\bq)$ for this
gradient and Hessian as linear and bilinear forms on $\Sym$. When 
$\bq$ is in the strict interior of $\Sym_{\succeq 0}$, these are defined by the
Taylor expansion
\[\Psi_\gs(\bq+\bx)=\Psi_\gs(\bq)+\nabla \Psi_\gs(\bq)[\bx]+\frac{1}{2}\nabla^2
\Psi_\gs(\bq)[\bx,\bx]+o(\|\bx\|^2)
\qquad \text{ for } \bx \in \Sym
\text{ with } \|\bx\| \to 0,\]
and we extend these definitions of $\nabla \Psi_\gs$ and $\nabla^2 \Psi_\gs$ by
continuity to the boundary of $\Sym_{\succeq 0}$.

\begin{proposition}\label{prop:GSderivatives}
Let $\langle f(\gb) \rangle_\qb$ be the mean under the posterior
law of $\gb$ in the single-sample model \eqref{eq:GSscalarmodel}.
\begin{enumerate}[(a)]
\item For any $\bq=(\bq_1,\ldots,\bq_L) \in \Sym_{\succeq 0}$ and
$\bx=(\bx_1,\ldots,\bx_L) \in \Sym$,
\begin{equation}\label{eq:GradPsi}
\nabla \Psi_\gs(\bq)[\bx]=\sum_{\ell=1}^L -\frac{\lambda_\ell}{2}
\tr \bx_\ell\Big(\bq_\ell-\E_{\bg_*,\bz} \bg_{*\ell}^\top
\langle \bg_\ell \rangle_\qb \Big)
=\sum_{\ell=1}^L -\frac{\lambda_\ell}{2}
\tr \bx_\ell\Big(\bq_\ell-\E_{\bg_*,\bz} \langle \bg_\ell \rangle_\qb^\top
\langle \bg_\ell \rangle_\qb\Big).
\end{equation}
In particular, $\nabla \Psi_\gs(\bq)=0$ if and only if 
$\bq_\ell=\E_{\bg_*,\bz} \langle \bg_\ell \rangle_\qb^\top \langle \bg_\ell
\rangle_\qb$
for every $\ell=1,\ldots,L$.
\item For any $\bq=(\bq_1,\ldots,\bq_L) \in \Sym_{\succeq 0}$,
$\bx=(\bx_1,\ldots,\bx_L) \in \Sym$, and
$\bx'=(\bx_1',\ldots,\bx_L') \in \Sym$,
\begin{equation}\label{eq:HessPsi}
\begin{aligned}
\nabla^2 \Psi_\gs(\bq)[\bx,\bx']
&=\sum_{\ell=1}^L -\frac{\lambda_\ell}{2}\tr \bx_\ell\bx_\ell'
+\sum_{\ell,\ell'=1}^L \frac{\lambda_\ell\lambda_{\ell'}}{2}
\E_{\bg_*,\bz}\Big[\big\langle \tr \bx_\ell \bg_{*\ell}^\top\bg_\ell
\tr \bx_{\ell'}'\bg_{*\ell'}^\top\bg_{\ell'}\big\rangle_\qb\\
&\quad
-2\tr \bx_\ell\bg_{*\ell}^\top  \langle \bg_\ell\rangle_\qb
\tr \bx_{\ell'}' \bg_{*\ell'}^\top \langle \bg_{\ell'}\rangle_\qb
+\tr \bx_\ell \langle \bg_\ell \rangle_\qb^\top \langle \bg_\ell
\rangle_\qb \tr \bx_{\ell'}'\langle \bg_{\ell'}
\rangle_\qb^\top \langle \bg_{\ell'} \rangle_\qb\Big].
\end{aligned}
\end{equation}
\end{enumerate}
\end{proposition}

If a representation $\gb_\ell=\phi_\ell(\bg)$ is \emph{not} real-irreducible, then applying an orthogonal change-of-basis so that the matrices
$\{\phi_\ell(\gb):\gb \in \cG\}$ are simultaneously block-diagonal (c.f.\ Theorem
\ref{thm:realreprdecomp}),
part (a) of this proposition implies that
$\nabla \Psi_\gs(\overlap)=0$ can only hold when $\overlap_\ell$ has this same
block-diagonal structure. Maximization of $\Psi_\gs(\overlap)$ may then
be restricted to $\overlap_\ell$ having
this structure, in agreement with the reduction in
Appendix \ref{appendix:GSreduction} of the
model \eqref{eq:group-synchronization-multi-representation-model} 
to a canonical form having only real-irreducible representations.

Assuming such a canonical form, the next result
characterizes the phase transition threshold for $\overlap=\bzero$ to locally
maximize $\Psi_\gs(\overlap)$. We recall in Appendix \ref{appendix:realrepr}
that any real-irreducible representation $\phi_\ell$ can be categorized as being
of ``real type'' if $\phi_\ell$ is also $\C$-irreducible, of
``complex type'' if $\phi_\ell \cong \psi \oplus \bar\psi$ where
$\psi,\bar\psi$ are $\C$-irreducible complex-conjugate sub-representations
with $\psi \not\cong \bar\psi$, or of ``quaternionic type'' if
$\phi_\ell \cong \psi \oplus \psi$ where $\psi$ is $\C$-irreducible
and $\psi \cong \bar\psi$; the type of $\phi_\ell$ may be determined from the
value of $\rho_\ell:=\E_\bg [(\tr \bg_\ell)^2]$.

\begin{proposition}\label{prop:GSlocaloptimality}
Suppose $\phi_1,\ldots,\phi_L$ are real-irreducible, distinct, and non-trivial
representations of $\cG$. Let
\begin{equation}\label{eq:irredcharacters}
\rho_\ell:=\E_\bg [(\tr \bg_\ell)^2]
=\begin{cases} 1 & \text{ if } \bg_\ell \text{ is of real type} \\
2 & \text{ if } \bg_\ell \text{ is of complex type}\\
4 & \text{ if } \bg_\ell \text{ is of quaternionic type}
\end{cases}
\end{equation}
and set $\tilde \lambda_\ell=\lambda_\ell \rho_\ell/k_\ell$.
Then at $\overlap=\bzero$, we have $\nabla \Psi_\gs(\bzero)=0$. Furthermore,
\begin{enumerate}[(a)]
\item If $\max_{\ell=1}^L \tilde \lambda_\ell<1$, then $\nabla^2 \Psi_\gs(\bzero)$
is negative-definite, and $\bq=\bzero$ is a local maximizer of $\Psi_\gs(\overlap)$.
\item If $\max_{\ell=1}^L \tilde \lambda_\ell>1$, then $\nabla^2 \Psi_\gs(\bzero)$
has a positive eigenvalue, and $\bq=\bzero$ is not a local maximizer of
$\Psi_\gs(\overlap)$.
\end{enumerate}
\end{proposition}

Let us spell out the implication of this result for four specific examples
of synchronization problems over rotation and permutation groups,
due to their particular interest in
applications \cite{singer2011angular,pachauri2013solving,bandeira2020non}.

\begin{example}[Multi-channel angular synchronization]
\label{example:multichannelSO2}
Let
\begin{equation}\label{eq:SO2}
\cG=\SO(2)=\left\{\begin{pmatrix} \cos \theta & -\sin\theta
\\ \sin\theta & \cos\theta \end{pmatrix}:\theta \in [0,2\pi)\right\}.
\end{equation}
Identifying $\gb \in \cG$ with its rotation angle $\theta \in
[0,2\pi)$, consider the multi-channel observation model
\eqref{eq:group-synchronization-multi-representation-model} with
the representations
\begin{equation}\label{eq:SO2reprs}
\gb_\ell=\begin{pmatrix} \cos \ell\theta & -\sin\ell\theta \\
\sin\ell\theta & \cos\ell\theta \end{pmatrix} \in \R^{2 \times 2}
\text{ for } \ell=1,\ldots,L.
\end{equation}
(The setting of single-channel angular synchronization corresponds to $L=1$.)

Here, the representations $\gb_\ell$ are distinct, and
each representation $\gb_\ell$ is real-irreducible: Indeed,
for any two unit vectors $\bu,\bv \in \R^2$ there exists $\theta \in
[0,2\pi)$ for which $\gb_\ell$ defined by \eqref{eq:SO2reprs} satisfies
$\bu=\gb_\ell\bv$, so no subspace of $\R^2$ is invariant.
We have $\E[(\tr \gb_\ell)^2]=\E_{\theta \sim
\unif([0,2\pi))}[(2\cos \ell\theta)^2]=2$,
so each representation $\bg_\ell$ is of complex type.

Then Proposition \ref{prop:GSlocaloptimality} implies that
$\overlap=\bzero$ is a local maximizer of $\Psi_\gs$ if
$\max_{\ell=1}^L \lambda_\ell<1$, and it is not a local
maximizer if $\max_{\ell=1}^L \lambda_\ell>1$.
\end{example}

\begin{example}[Rotational synchronization]\label{example:SOk}
Let $\cG=\SO(k)$, and
consider a single-channel model with the standard representation
$\phi(\gb)=\gb \in \R^{k \times k}$ given by its rotational action on $\R^k$.
That is, for an unknown signal vector
$G_*=(\gb_*^{(1)},\ldots,\gb_*^{(N)})$ with prior distribution
$\{\gb_*^{(i)}\}_{i=1}^N \overset{iid}{\sim} \Haar(\SO(k))$, we observe
\[\yb^{(ij)}=\sqrt{\lambda}\,{\gb_*^{(i)\top}}\gb_*^{(j)}
+\sqrt{N}\,\zb^{(ij)} \in \R^{k \times k} \quad \text{ for } 1 \leq i<j \leq N\]
where $\zb^{(ij)}$ are independent
noise matrices with i.i.d.\ $\cN(0,1)$ entries.

This representation is real-irreducible, for the same reason as in 
Example \ref{example:multichannelSO2}. For $k=2$, it is of complex type as shown
in Example \ref{example:multichannelSO2}.
For $k \geq 3$, we have $\E[g_{ii}g_{jj}]=0$ for all $i \neq j$ and
$\E[g_{ij}^2]=\frac{1}{k^2}\E \tr \gb^\top \gb=\frac{1}{k}$ for
all $i,j \in \{1,\ldots,k\}$, by the invariances in law of $\SO(k)$ under
negations and transpositions of rows and columns. Thus
$\E[(\tr \gb)^2]=\E[\sum_{i=1}^k g_{ii}^2]=1$, so the representation is
of real type (i.e.\ it is also $\C$-irreducible).

Set
\[\lambda_c:=\begin{cases} 1 & \text{ if } k=2 \\
k & \text{ if } k \geq 3. \end{cases}\]
Proposition~\ref{prop:GSlocaloptimality} then implies that
$\overlap=\bzero$ is a local maximizer of $\Psi_\gs$ for
$\lambda<\lambda_c$, and it is not a local
maximizer for $\lambda>\lambda_c$.
\end{example}

\begin{example}[Cyclic permutation synchronization]\label{example:Ck}
Let $\cG=\Z/k\Z$ be the cyclic group of size $k$, with elements
$\{\iden,\bh,\bh^2,\ldots,\bh^{k-1}\}$, and consider its
action on $\R^k$ by cyclic permutations of coordinates. We note
that the span of $\eb=(1,1,\ldots,1) \in \R^k$
is a trivial invariant subspace of this action,
which carries no information about the permutation. Hence,
let us consider the model defined by
$\phi(\gb) \in \R^{(k-1) \times (k-1)}$ representing the restriction of this
action to the subspace orthogonal to $\eb$, under any choice of orthonormal
basis for this subspace. That is, for an unknown signal vector
$G_*=(\gb_*^{(1)},\ldots,\gb_*^{(N)})$ with prior distribution
$\{\gb_*^{(i)}\}_{i=1}^N \overset{iid}{\sim} \Haar(\Z/k\Z)$,
we observe
\[\yb^{(ij)}=\sqrt{\lambda}\,\phi(\gb_*^{(i)})^\top\phi(\gb_*^{(j)})
+\sqrt{N}\,\zb^{(ij)} \in \R^{(k-1) \times (k-1)}
\quad \text{ for } 1 \leq i<j \leq N\]
where $\zb^{(ij)}$ are again independent noise matrices with
i.i.d.\ $\cN(0,1)$ entries.

Suppose (for simplicity of discussion) that $k \geq 3$ is odd. Then
$\{\phi(\gb):\gb \in \Z/k\Z\}$ leaves invariant the
2-dimensional subspaces $\{S_\ell\}_{\ell=1}^{(k-1)/2}$ of $\R^k$ orthogonal
to $\eb$, spanned by the pairs of vectors
\[\Big(1,\cos \frac{2\pi \ell}{k},\cos \frac{4\pi \ell}{k},\ldots,\cos
\frac{2\pi \ell(k-1)}{k}\Big),
\quad \Big(0,\sin \frac{2\pi \ell}{k},\sin \frac{4\pi \ell}{k},\ldots,\sin
\frac{2\pi \ell(k-1)}{k}\Big).\]
The sub-representation of $\phi(\cdot)$ restricted to each subspace
$S_\ell$ is isomorphic to the representation given by
\[\phi_\ell(\bh^j)=\begin{pmatrix} \cos (2\pi \ell j/k) &
{-}\sin (2\pi \ell j/k) \\ \sin (2\pi \ell j/k)
& \cos (2\pi \ell j/k)  \end{pmatrix} \in \R^{2 \times 2},\]
so (c.f.\ Appendix \ref{appendix:GSreduction})
this model is equivalent to a multi-channel model in which we observe
\[\yb_\ell^{(ij)}=\sqrt{\lambda}\,\gb_{*\ell}^{(i)\top}\gb_{*\ell}^{(j)}
+\sqrt{N}\,\zb_\ell^{(ij)} \in \R^{2 \times 2}\]
for all $\ell=1,\ldots,(k-1)/2$ and $1 \leq i<j \leq N$,
with $\gb_{*\ell}^{(i)}=\phi_\ell(\gb_*^{(i)})$. These representations
$\phi_\ell$ are distinct, real-irreducible, and of complex type by a similar
argument as in Example \ref{example:multichannelSO2}.

Thus, Proposition \ref{prop:GSlocaloptimality} shows that $\overlap=\bzero$ is
a local maximizer of $\Psi_\gs$ if $\lambda<1$, and it is not a local maximizer
if $\lambda>1$.
\end{example}

\begin{example}[Permutation synchronization]\label{example:Sk}
Consider now the full symmetric group $\cG=\SS_k$, with action by permutation of
coordinates on $\R^k$. Again, as $\eb=(1,1,\ldots,1) \in \R^k$ spans a trivial
invariant subspace, we consider the model defined by the standard representation
$\phi(\bg) \in \R^{(k-1) \times (k-1)}$ representing this action on the subspace
orthogonal to $\eb$. That is, for an unknown signal vector
$G_*=(\gb_*^{(1)},\ldots,\gb_*^{(N)})$ with prior distribution
$\{\gb_*^{(i)}\}_{i=1}^N \overset{iid}{\sim} \Haar(\SS_k)$,
we observe
\[\yb^{(ij)}=\sqrt{\lambda}\,\phi(\gb_*^{(i)})^\top\phi(\gb_*^{(j)})
+\sqrt{N}\,\zb^{(ij)} \in \R^{(k-1) \times (k-1)}
\quad \text{ for } 1 \leq i<j \leq N\]
where $\zb^{(ij)}$ are again independent noise matrices with
i.i.d.\ $\cN(0,1)$ entries. Here, the standard representation $\phi(\cdot)$ is
$\C$-irreducible \cite[Proposition 3.12]{fulton2013representation},
and hence also real-irreducible of real type.

Thus, Proposition \ref{prop:GSlocaloptimality} shows that $\overlap=\bzero$ is
a local maximizer of $\Psi_\gs$ if $\lambda<k-1$, and it is not a local
maximizer if $\lambda>k-1$.
\end{example}

More generally, the algorithmic phase transition threshold for local
optimality of $\overlap=\bzero$ in any such example may be deduced by first
reducing the model to a canonical form as described in Appendix
\ref{appendix:GSreduction}, then determining the type of each real-irreducible
component, and finally determining the thresholds for their corresponding signal
strengths from Proposition \ref{prop:GSlocaloptimality}.

\begin{remark}\label{remark:PWMB}
The work \cite{PWBM2018} studied a version of this model of the form
\begin{align}
\label{eq:group-synchronization-complex}
\begin{cases}
    \yb_1^{(ij)} = \sqrt{\lambda_1}\,\gb_{*1}^{(i)}\gb_{*1}^{(j)*}
+ \sqrt{N}\,\zb_1^{(ij)} \in \C^{k_1 \times k_1}\\
    \hspace{1in}\vdots \\
    \yb_L^{(ij)} = \sqrt{\lambda_L}\,\gb_{*L}^{(i)}\gb_{*L}^{(j)*}
+ \sqrt{N}\,\zb_L^{(ij)} \in \C^{k_L \times k_L}
\end{cases} \text{ for all } 1 \leq i<j \leq N
\end{align}
where $\bg_{*\ell}=\phi_\ell(\bg_*) \in \C^{k_\ell \times k_\ell}$
for $\ell=1,\ldots,L$
correspond to distinct $\C$-irreducible representations of $\cG$, and
$\{\bz_\ell^{(ij)}\}_{i<j}$ are the sub-blocks of $k_\ell N \times k_\ell N$
noise matrices distributed
according to the GOE, GUE, or GSE depending on the type of the representation
$\phi_\ell$. For such a model, \cite{PWBM2018} developed and analyzed an AMP
algorithm for Bayes-optimal inference, and stated also a replica
formula for the free energy that is similar to \eqref{eq:Psigsformula}.
It was argued (more heuristically) in \cite[Section 6.6]{PWBM2018}
that $\overlap=\bzero$ is a stable fixed point of this AMP algorithm
if and only if $\max_{\ell=1}^L \lambda_\ell/k_\ell<1$. Our results of
Theorem \ref{theorem:free-energy-group-syn-multi-representation}
and Proposition \ref{prop:GSlocaloptimality} thus provide rigorous proofs of
analogous statements in a real version of this model.

One difference between our analyses and those of \cite{PWBM2018}---in addition
to the real vs.\ complex distinction---is that the replica formula stated in
\cite{PWBM2018} implicitly assumes that the maximization of
$\Psi_\gs(\overlap)$ may be restricted to overlaps
$\overlap=(\overlap_1,\ldots,\overlap_L)$ where each $\overlap_\ell$ is
a scalar multiple of the identity matrix.
The analyses of AMP state evolution
in \cite{PWBM2018} also assume an initialization at overlaps
$\overlap$ of this form, and stability of the state evolution at
$\overlap=\bzero$ is analyzed under this restriction of $\overlap$.
We conjecture that this restriction of the replica formula for the
free energy is correct, but have not found a general proof outside the
setting of abelian groups (c.f.\ Proposition \ref{prop:abelian} below); see also
Remark \ref{remark:abelian} for more discussion on the non-abelian case.
Due to this reason, and also due to interest in AMP algorithms whose initial
overlap may not be of this restricted form,
we have defined the replica potential $\Psi_\gs$ over all
symmetric positive-semidefinite overlaps $\overlap$, and carry out in
Proposition \ref{prop:GSlocaloptimality} an analysis of the local optimality of
$\overlap=\bzero$ with respect to optimization of $\Psi_\gs$
over this full overlap space.

When $\overlap=\bzero$ is not a local maximizer of $\Psi_\gs$,
we believe that a variant of the AMP algorithm proposed
in \cite[Algorithm 3.1]{PWBM2018} with an appropriate spectral 
initialization may achieve the optimal matrix mean-squared-error among
computationally efficient algorithms. However, the initial overlap $\overlap$
under such a spectral initialization is not necessarily of the above restricted
form comprised of scalar multiples of the identity. This would require a
modification of the algorithm and state evolution discussed in \cite{PWBM2018}
to a larger overlap space,
and we leave an investigation of such an algorithm to future work.
\end{remark}

\subsection{Mutual information and MMSE for angular synchronization}

A complete characterization of the information-theoretic limits of inference
and the possible existence of computationally-hard SNR regimes may be obtained
via an analysis of the global optimization landscape of $\Psi_\gs$. To our
knowledge, this has been carried out only for the
$\Z/2\Z$-synchronization example, in \cite{deshpande2015asymptotic}.

In this section, we develop a similar characterization for
single-channel angular synchronization over $\SO(2)$,
corresponding to Example \ref{example:multichannelSO2} with $L=1$
and Example \ref{example:SOk} with $k=2$. This is a model with
the pairwise observations
\begin{equation}\label{eq:model-SO2}
\by^{(ij)}=\sqrt{\lambda}\bg_*^{(i)\top}\bg_*^{(j)}+\sqrt{N}\,\bz^{(ij)}
=\sqrt{\lambda}\begin{pmatrix}
\cos (\theta_j-\theta_i) & -\sin (\theta_j-\theta_i) \\
\sin (\theta_j-\theta_i) & \cos (\theta_j-\theta_i) \end{pmatrix}
+\sqrt{N}\,\bz^{(ij)} \quad \text{ for } 1 \leq i<j \leq N
\end{equation}
where $\theta_1,\ldots,\theta_N \overset{iid}{\sim} \unif([0,2\pi))$.
Averaging the two measurements of $\cos(\theta_j-\theta_i)$ and
$\sin(\theta_j-\theta_i)$ in each observation $\by^{(ij)}$,
the model is equivalent to the phase-synchronization model 
\cite{boumal2016nonconvex} over $\U(1)$ with complex observations
\[y_{ij}=\sqrt{\lambda} e^{i(\theta_j-\theta_i)}+\sqrt{N}\,z_{ij} \in \C,
\qquad \Re z^{(ij)},\Im z^{(ij)} \overset{iid}{\sim} \cN(0,\tfrac{1}{2}).\]
Some partial analyses of the replica potential in this model were carried
out in \cite[Section 7.2.2]{javanmard2016phase}, and it was
shown in \cite[Theorem 6.11]{perry2016optimality} using an alternative
second-moment-method calculation that $\lambda_c=1$ is the
(information-theoretic) threshold for contiguity with the null model
$\by^{(ij)}=\sqrt{N}\,\bz^{(ij)}$. Here, we extend these results by showing that
the replica potential has a single unique local maximizer $\overlap \in
\Sym_{\succeq 0}^{2 \times 2}$, which is non-zero if and only if $\lambda>1$,
thus providing a full characterization of the Bayes-optimal MMSE and confirming 
the absence of a statistical-computational gap in this model for any
positive $\lambda$.

We begin with a general statement that optimization of the replica potential
$\Psi_\gs$ may be restricted to overlaps
$\overlap=(\overlap_1,\ldots,\overlap_L)$ where each $\overlap_\ell$ is a
positive multiple of the identity, if $\cG$ is abelian and each $\phi_\ell$ is
real-irreducible. (Each representation must then take values in
$\R^{1 \times 1}$ or $\R^{2 \times 2}$, c.f.\ Corollary \ref{cor:twodim}, so
this statement pertains to the structure of $\overlap_\ell$ corresponding to
the 2-dimensional representations $\phi_\ell$.)

\begin{proposition}\label{prop:abelian}
Suppose $\cG$ is any abelian group, and $\phi_1,\ldots,\phi_L$ are
real-irreducible. If $\overlap=(\overlap_\ell)_{\ell=1}^L
\in \Sym_{\succeq 0}$ is a critical point or
local maximizer of $\Psi_\gs$, then each $\overlap_\ell$ is
a scalar multiple of the identity.
\end{proposition}

Restricting to scalar multiples of the identity $\overlap=q\,I_{2 \times 2}$
with $q \geq 0$, the replica potential takes the form
\begin{align}
\label{eq:free-energy-to-mutual-information-so2}
\Psi_\gs(q\,I_{2 \times 2})=-\frac{\lambda q^2}{2} + \lambda q - i(\lambda q).
\end{align}
Here, by \eqref{eq:scalarMI}, $i(\gamma)$ is the mutual information between
$\gb_*$ and $\yb$ in a single-letter model
\begin{align}\label{eq:SO2singleletter}
    \yb=\sqrt{\gamma}\,\gb_*+\zb
\end{align}
with $\gb_* \sim \Haar(\SO(2))$ and a noise matrix $\zb \in \R^{2 \times 2}$
having i.i.d.\ $\cN(0,1)$ entries. By the i-mmse relation \cite{guo2005mutual},
critical points $q_*$ of $\Psi_\gs(q\,I_{2 \times 2})$ correspond to solutions of the
fixed-point equation
\begin{align}
\label{eq:critical-point-so2}
    q_* = 1 - \frac{1}{2}\,\textnormal{mmse}(\lambda q_*)
\end{align}
where $\textnormal{mmse}(\gamma)=\E\|\gb_*-\E[\gb \mid \by]\|_F^2
=2-\E\|\E[\gb \mid \by]\|_F^2$
is the minimum mean-squared-error for estimating $\gb_*$
in the single-letter model \eqref{eq:SO2singleletter}.

The following is our main result for $\SO(2)$-synchronization.

\begin{theorem}\label{thm:phase-transition-so2}
For the $\SO(2)$-synchronization model \eqref{eq:model-SO2},
all critical points of $\Psi_\gs(\overlap)$ are given by
\begin{equation}\label{eq:diagonalreduction}
\{\overlap \in \Sym_{\succeq 0}^{2 \times 2}:\nabla \Psi_\gs(\overlap)=0\}
=\{q_*\,I_{2 \times 2}:q_*
\text{ solves \eqref{eq:critical-point-so2}}\}.
\end{equation}
If $\lambda \leq \lambda_c:=1$, then 0
is the only solution of \eqref{eq:critical-point-so2}, $\overlap=\bzero$
is the unique global maximizer of $\Psi_\gs$ over $\Sym_{\succeq 0}^{2 \times 2}$, and
\begin{align*}
    \lim_{n \to \infty} \frac{1}{N}\,I(\Theta_*,Y)=\frac{\lambda}{2}
\quad \text{and} \quad \lim_{N \to \infty} \mmse=2,
\end{align*}
and for any $\epsilon > 0$, there exists constants $C, c > 0$ depending only on $\epsilon$ such that
\begin{align*}
    \E\!\inangle{  \mathbf{1}\!\inbraces{ \norm{ \frac{1}{N} \sum_{i = 1}^{N} \gb^{(i) \top} \gb_*^{(i)} }_{F}^{2}  > \frac{\epsilon}{\lambda}  }   } \leq C e^{-cN}.
\end{align*}
If $\lambda>\lambda_c:=1$,
then there exists a unique positive solution $q_*>0$ of
\eqref{eq:critical-point-so2}, $\overlap=q_*\,I_{2 \times 2}$ is the unique global
maximizer of $\Psi_\gs$ over $\Sym_{\succeq 0}^{2 \times 2}$, and
\[\lim_{N \to \infty} \frac{1}{N}I(\Theta_*,Y)=\frac{\lambda}{2} -
\Psi_\gs(q_*\,I_{2 \times 2})
\quad \text{and} \quad \lim_{N \to \infty} \mmse=2-2q_*^2,\] and for any $\epsilon > 0$, there exists constants $C, c > 0$ depending only on $\epsilon$ such that
\begin{align*}
    \E\!\inangle{  \mathbf{1}\!\inbraces{ \norm{ \frac{1}{N} \sum_{i = 1}^{N} \gb^{(i) \top} \gb_*^{(i)} - q_* \hb }_{F}^{2} > \frac{\epsilon}{\lambda} \text{ for all } \hb \in \SO(2)   }  } \leq C e^{-cN}.
\end{align*}
\end{theorem}

\begin{remark}\label{remark:abelian}
The reduction \eqref{eq:diagonalreduction} to diagonal overlap
matrices $\overlap$ is possible because $\SO(2)$ is abelian.
The analogous statement in the non-abelian setting of $\SO(k)$-synchronization
for $k \geq 3$ is false: We show in Proposition \ref{prop:SOkcounterexample}
that for any $k \geq 3$ and $\lambda>\lambda_c:=k$, $\Psi_\gs$ has a
critical point that is not a multiple of the identity.
This suggests that analyses of the global optimization landscape of
$\Psi_\gs$ may need to be multivariate in nature, and it remains an open
question to fully characterize this landscape for $\SO(k)$-synchronization when
$k \geq 3$ or, more generally, for any non-abelian group $\cG$.
\end{remark}

\section{Quadratic assignment}\label{sec:quadraticassignment}

We transition to a second application of the general results in Section
\ref{sec: general_model}, and study a quadratic assignment model for inference
over the symmetric group. Here, the signal and signal prior do
not have the ``product structure'' of group synchronization.

Let $\cX$ be a compact space and $\kappa:\cX \times \cX \to \R$ a pairwise
similarity kernel, both independent of $N$. We observe samples
$x_1,\ldots,x_N \in \cX$, together with noisy pairwise similarities of
a permutation of these samples,
\begin{equation}\label{eq:QAmodel}
y_{ij}=\kappa(x_{\pi_*(i)},x_{\pi_*(j)}) + \sqrt{N} z_{ij}
\text{ for each } 1 \leq i<j \leq N.
\end{equation}
Here $z_{ij} \overset{iid}{\sim} \cN(0,1)$,
and $\pi_* \in \SS_N$ is an unknown permutation of
interest, assumed to have uniform prior on the symmetric group $\SS_N$
of all permutations of $N$ elements.

We will characterize the asymptotic mutual information $I(\pi_*,Y)$ between the
latent permutation $\pi_*$ and observations $Y=(y_{ij})_{i<j}$, under the
following assumptions on $\cX$, $\kappa$, and $x_1,\ldots,x_N$
as $N \to \infty$.

\begin{assumption}\label{asmpt:kernel}
\begin{enumerate}[(a)]
\item $\cX$ is a compact space, and
$\kappa:\cX \times \cX \rightarrow \R$ is a continuous
positive-semidefinite kernel, i.e.\ $\kappa(x_i,x_j)_{i,j=1}^m \in \R^{m \times m}$
is positive-semidefinite for any $m \geq 1$ and $x_1,\ldots,x_m \in \cX$.
\item There exists a probability distribution $\rho$ on $\cX$ such that, as $N
\to \infty$, the empirical law $\frac{1}{N}\sum_{i=1}^N \delta_{x_i}$ converges
weakly to $\rho$.
\end{enumerate}
\end{assumption}

Under Assumption \ref{asmpt:kernel}, $\kappa$ is a Mercer kernel admitting the
following approximation by eigenfunctions, see
e.g.\ \cite[Theorem 12.20]{wainwright2019high}.

\begin{theorem}[Mercer's theorem]\label{thm:mercer}
Suppose Assumption \ref{asmpt:kernel} holds. Then there exists an orthonormal
basis of eigenfunctions
$\{f_\ell\}_{\ell=1}^\infty$ of $L^2(\cX,\rho)$ and eigenvalues $\mu_1
\geq \mu_2 \geq \mu_3 \geq \ldots \geq 0$ such that
$\int_{\cX} \kappa(x,y) f_\ell(y) \, \rho(\textnormal{d}y) = \mu_\ell f_\ell(x)$.
Furthermore, $\kappa$ admits the expansion
\begin{align*}
\kappa(x,y) = \sum_{\ell = 1}^\infty \mu_\ell f_\ell(x) f_\ell(y)
\end{align*}
where this series converges absolutely and uniformly over all $x,y \in \cX$.
\end{theorem}

We define from this eigenfunction expansion, for each $L \geq 1$,
the truncated kernel
\begin{align}\label{eq:truncated-kernel}
\kappa^L(x,y) \coloneqq \sum_{\ell=1}^L \mu_\ell f_\ell(x) f_\ell(y).
\end{align}
The model defined by $\kappa^L$ in place of $\kappa$ falls into the
framework of Section \ref{sec: general_model}, where $\cG_N \equiv \SS_N$ is the
symmetric group, and the feature map $\phi:\SS_N \to \cH^N$ has components
\begin{equation}\label{eq:QAfeaturemap}
\phi(\pi)_i=(\sqrt{\mu_1}f_1(x_{\pi(i)}),\ldots,\sqrt{\mu_L}f_L(x_{\pi(i)})) \in \cH \equiv \R^L.
\end{equation}
The bilinear maps $\bullet:\cH \times \cH \to \cK \equiv \R$ and
$\otimes: \cH \times \cH \to \cL \equiv \R^{L \times L}$ and the
inclusion map $\iota:\cL \to B(\cH)$ are given by
\begin{equation}\label{eq:QAbilinearforms}
\ab \bullet \bb=\ab^\top\bb,
\qquad \ab \otimes \bb=\ab\bb^\top,
\qquad \iota(\overlap)\ab=\overlap\ab,
\end{equation}
where we equip $\cH,\cK,\cL$ with their usual Euclidean inner-products. Our
analyses will first characterize the asymptotic mutual information
between $\pi_*$ and observations $Y^L=\{y_{ij}^L\}_{i<j}$ defined with the
truncated kernel $\kappa^L$, and then take a limit $L \to \infty$ to describe the
mutual information for the original observations $Y$.

\paragraph{Asymptotic mutual information.}

Fixing any $L \geq 1$ and an overlap matrix
$\overlap \in \Sym_{\succeq 0}^{L \times L}$, denote
$\bu(x)=(\sqrt{\mu_1}f_1(x),\ldots,\sqrt{\mu_L}f_L(x)) \in \R^L$ and consider a linear observation model
\begin{equation}\label{eq:QAlinearmodel}
\by_i=\overlap^{1/2} \ub(x_{\pi_*(i)})+\bz_i \in \R^L
\text{ for } i=1,\ldots,N
\end{equation}
where $\pi_* \sim \Haar(\SS_N)$ and $\{\bz_i\}_{i=1}^N
\overset{iid}{\sim} \cN(0,I_{L \times L})$. Consider also the single-letter
model
\begin{equation}\label{eq:QAscalarmodel}
\by=\overlap^{1/2}\ub(x)+\bz \in \R^L
\end{equation}
where $x \sim \rho$ is a single random sample in $\cX$, and $\bz \sim \cN(0,I)$.
It is direct to check, as in \eqref{eq:linearMI}, that the mutual information
in the model \eqref{eq:QAscalarmodel} is given by
$i(x,\by)=-\frac{1}{4}\|\overlap\|_F^2+\frac{1}{2} \E_{x_*}\tr \ub(x_*)^\top
\overlap \ub(x_*) -\Psi_\qs^L(\overlap)$,
for the potential function
\begin{align}
\Psi_\qs^L(\overlap)=-\frac{1}{4}\|\overlap\|_F^2
+ \E_{x_*,\bz} \log \E_x \exp\inparen{-\frac{1}{2}\,\ub(x)^\top \overlap
\ub(x)+\ub(x)^\top \overlap \ub(x_*) + \ub(x)^\top \overlap^{1/2}\zb}.
    \label{eq:qua-assgn-RS-potential-decoupled}
\end{align}
Here, $\E_x$ is over $x \sim \rho$ in $\cX$, and $\E_{x_*,\bz}$ is over
independent $x_* \sim \rho$ and $\bz \sim \cN(0,I)$.

Inference in the model \eqref{eq:QAlinearmodel} may be understood as the
task of estimating $\btheta_i=\overlap^{1/2}\bu(x_{\pi_*(i)}) \in \R^L$ for
$i=1,\ldots,N$ from observations $\by_i=\btheta_i+\bz_i$, given only the empirical
distribution of the values $\{\btheta_i\}_{i=1}^N$ but not their ordering.
In contrast, inference in the model \eqref{eq:QAscalarmodel} is the task of
estimating $\btheta_i$ from an observation $\by_i=\btheta_i+\bz_i$
assuming a Bayesian prior for $\btheta_i$. Comparisons between these two tasks
underlie the classical literature on empirical Bayes estimation in compound
decision problems; in particular, the efficiency of coordinate-separable
decision rules within the class of all decision rules for the former model
\eqref{eq:QAlinearmodel} has been investigated in
\cite{hannan1955asymptotic,greenshtein2009asymptotic}.

Leveraging the main result of \cite{greenshtein2009asymptotic}, the following
lemma first shows that the signal-observation mutual information in the linear
observation model \eqref{eq:QAlinearmodel} coincides, to leading asymptotic
order, with that in the scalar model \eqref{eq:QAscalarmodel}.

\begin{lemma}\label{lemma:QAlinearMI}
Suppose Assumption \ref{asmpt:kernel} holds, and fix any $L \geq 1$ and
$\overlap \in \Sym_{\succeq 0}^{L \times L}$. Let $i(\pi_*,Y_\lin)$ be the
mutual information between $\pi_* \in \SS_N$ and $Y_\lin=\{\by_i\}_{i=1}^N$
in the model \eqref{eq:QAlinearmodel}. Then
\[\lim_{N \to \infty} \frac{1}{N} i(\pi_*,Y_\lin)=i(x,\by)
:=-\frac{1}{4}\|\overlap\|_F^2+\frac{1}{2}\,\E_{x_*}\tr \ub(x_*)^\top \overlap \ub(x_*) -\Psi_\qs^L(\overlap).\]
\end{lemma}

The general framework of Section \ref{sec: general_model} then allows us to
relate $i(\pi_*,Y_\lin)$ with the mutual information $I(\pi_*,Y)$ in the
quadratic assignment model \eqref{eq:QAmodel}, yielding the following main
result of this section.

\begin{theorem}\label{thm:qua-assgn-mutual-info}
Suppose Assumption \ref{asmpt:kernel} holds. Then there exists a finite limit
\[\Psi_\infty
=\lim_{L \to \infty} \sup_{\overlap \in \Sym_{\succeq 0}^{L \times L}}
\Psi_\qs^L (\overlap),\]
and the mutual
information $I(\pi_*,Y)$ between $\pi_* \sim \SS_N$ and $Y=\{y_{ij}\}_{i<j}$
in the model \eqref{eq:QAmodel} satisfies
\begin{equation}\label{eq:QAmutualinfolimit}
\lim_{N \to \infty} \frac{1}{N}\,I(\pi_*,Y)
=\frac{1}{4}\,\E_{x,x' \overset{iid}{\sim} \rho}[\kappa(x,x')^2]-\Psi_\infty.
\end{equation}
\end{theorem}
Similar to Theorem \ref{theorem:free-energy-group-syn-multi-representation}(2), 
our results on asymptotic mutual information also imply an asymptotic limit
for the
Bayes-optimal error for estimating $(\kappa(x_{\pi_*(i),\pi_*(j)}))_{i,j=1}^N$.
To apply an I-MMSE relation, let us introduce the following model with 
$\sqrt{\lambda} \kappa$ in place of $\kappa$ in \eqref{eq:QAmodel}:
\begin{equation}
\label{eq:QAmodel_with_lambda}
        y_{ij}^{(\lambda)} = \sqrt{\lambda} \kappa(x_{\pi_*(i)}, x_{\pi_*(j)}) + \sqrt{N} z_{ij} 
        \text{ for each } 1 \leq i<j \leq N,
\end{equation}
and let
\begin{equation}\label{eq:QAMMSE}
    \mmse_\qs(\lambda) = \frac{1}{N^2} \E \sum_{i,j=1}^N \inparen{\kappa(x_{\pi_*(i)}, x_{\pi_*(j)}) 
    - \inangle{\kappa(x_{\pi(i)}, x_{\pi(j)})}}^2.
\end{equation}
It is clear that Theorem \ref{thm:qua-assgn-mutual-info} still holds with
$\sqrt{\lambda}\kappa$ in place of $\kappa$ for any $\lambda \in (0, \infty)$.
\begin{corollary}
\label{corollary:QAMMSE}
Let $$D = \{\lambda > 0 \colon \lambda \mapsto \Psi_\infty \text{ is differentiable at }\lambda\}.$$
Then $D$ has full Lebesgue measure in $(0, \infty)$. For any $\lambda \in D$,
all maximizers $\overlap_*^L$ of $\Psi_\qs^L(\overlap)$ have the same 
Frobenius norm $q_*^L:=\|\overlap_*^L\|_F$. Furthermore,
$q_* \coloneqq \lim_{L \to \infty} q_*^L$ exists, and
    \begin{equation}
        \lim_{N \to \infty} \mmse_\qs(\lambda) = \E_{x,x'\overset{iid}{\sim}\rho}[\kappa(x,x')^2] - q_*^2.
    \end{equation}
\end{corollary}
The limit value in \eqref{eq:QAmutualinfolimit} may be understood as the
mutual information between a signal vector $(x_{*1},\ldots,x_{*N})$
having i.i.d.\ prior $x_{*i} \overset{iid}{\sim} \rho$, and observations
\[y_{ij}=\kappa(x_{*i},x_{*j})+\sqrt{N}\,z_{ij} \text{ for } 1 \leq i<j \leq
N.\]
In the setting where $\kappa$ has a finite expansion into eigenfunctions, i.e.\
$\kappa^L=\kappa$ for some finite $L$, this is the mutual
information in a usual low-rank matrix estimation model with i.i.d.\ signal
prior. Informally, Theorem~\ref{thm:qua-assgn-mutual-info}
shows that in a bounded SNR regime of the model \eqref{eq:QAmodel}
where the kernel eigenvalues
$\mu_1,\mu_2,\ldots$ are fixed independently of $N$, observing the exact
sample points $x_1,\ldots,x_N$ is asymptotically no more informative for
estimating $\kappa(x_{\pi_*(i)},x_{\pi_*(j)})_{i<j}$ than
knowing the ``prior distribution'' $\rho$ corresponding to the limit of
their empirical law. This suggests that in the asymptotic limit
$N \to \infty$, correct recovery of a non-vanishing fraction of rows of $\pi_*$
 may be impossible in this regime, although we will not develop a formal
statement of this impossibility result here.

\section{Conclusion}

In this work, we have studied the two models of group synchronization and
quadratic assignment on pairs of noisy positive-semidefinite kernel matrices
observed with Gaussian noise. These problems
share a common structure of estimating a latent element $G_*$ of a
high-dimensional group from pairwise observations.
Assuming a Bayesian setting with Haar-uniform prior for $G_*$,
we have derived under a common framework the limit of the signal-observation
mutual information in both models, in an asymptotic regime of bounded SNR.
For group
synchronization, we have given a complete characterization of the algorithmic
phase transition threshold for $\overlap=\bzero$ to locally optimize the replica
potential in general groups.
For quadratic assignment, we have shown that the signal-observation
mutual information is asymptotically equivalent to that in a low-rank matrix
estimation model with i.i.d.\ signal prior.

The framework developed here is fairly general, and may apply to other
Bayesian inference problems of this form, where the underlying group $\cG_N$
does not necessarily have a product structure. We have analyzed two examples in
which the linear observation model (to which the original quadratic model is
compared) admits a reasonably simple direct analysis. In applications with other
group structures, as well as in other regimes of SNR, the linear model itself
may exhibit other types of asymptotic behaviors, and we believe this may be
interesting to investigate in future work.

\section*{Acknowledgments}

We would like to thank Alex Wein for helpful conversations about
\cite{PWBM2018}, and Yihong Wu for helpful conversations and pointers to the
works \cite{greenshtein2009asymptotic,polyanskiy2021sharp}
on empirical Bayes estimation. This research was supported in part by
NSF DMS-2142476.

\pagebreak

\appendix

\section{Proofs for the general model}\label{sec:pf-general_model}

Throughout this section, we write as shorthand
\[\phi=\phi(G), \qquad \phi'=\phi(G'), \qquad \phi_*=\phi(G_*)\]
and abbreviate $\cF \equiv \cF_N$, $\Psi \equiv \Psi_N$. Define the Hamiltonian
\begin{equation}\label{eq:mainhamiltonian}
\begin{aligned}
\widetilde H(G;G_*,Z)
&=\sum_{1 \leq i<j \leq N} -\frac{1}{2N}\|\phi_i \bullet \phi_j\|_\cK^2
+\frac{1}{N}\langle \phi_i \bullet \phi_j,\phi_{*i} \bullet \phi_{*j}
\rangle_\cK+\frac{1}{\sqrt{N}}\langle \phi_i \bullet \phi_j,\bz_{ij} \rangle_\cK\\
&\qquad+\sum_{i=1}^N {-}\frac{1}{4N}\|\phi_i \bullet \phi_i\|_\cK^2
+\frac{1}{2N}\langle \phi_i \bullet \phi_i,\phi_{*i} \bullet \phi_{*i}
\rangle_\cK+\frac{1}{\sqrt{2N}}\langle \phi_i \bullet \phi_i,\bz_{ii}
\rangle_\cK
\end{aligned}
\end{equation}
where $\{\bz_{ii}\}_{i=1}^N$ are additional standard Gaussian noise vectors in
$\cK$, independent of $G_*$ and of $\{\bz_{ij}\}_{i<j}$.
Then $\exp \widetilde H(G;G_*,Z)$ is
proportional to the posterior density of $G$ in the model
\eqref{eq:quadraticmodel} with additional observations
\[\by_{ii}=\phi_i \bullet \phi_i+\sqrt{2N}\,\bz_{ii} \qquad
\text{ for } i=1,\ldots,N.\]
We will establish lower and upper bounds for the perturbed free energy
\[\widetilde \cF=\frac{1}{N}\,\E_{G_*,Z} \log \E_G \exp \widetilde H(G;G_*,Z)\]
and remove this perturbation at the conclusion of the proof.

\subsection{Free energy lower bound}

We first prove the following lower bound for $\widetilde \cF$.

\begin{lemma}\label{thm:lower-bound-general-bound}
Under Assumption \ref{assumption:general-model},
\[\widetilde \cF \ge \sup_{\overlap \in \cQ} \Psi(\overlap).\]
\end{lemma}

Fixing any $\overlap \in \cQ$, for every $0 \le t \le 1$,
consider the observations
\begin{align}
\label{eq:lower-bound-interpolation}
    \begin{cases}
        \by_{ij}^{(t)}=\sqrt{t}\phi_{*i} \bullet \phi_{*j} + \sqrt{N}\,\bz_{ij}
\text{ for all } 1 \leq i<j \leq N\\
\by_{ii}^{(t)}=\sqrt{t}\phi_{*i} \bullet \phi_{*i} + \sqrt{2N}\,\bz_{ii}
\text{ for all } i=1,\ldots,N\\
        \by_i^{(t)}=\sqrt{(1-t)}\overlap^{1/2}\phi_{*i}+\bz_i
\text{ for all } i=1,\ldots,N
    \end{cases}
\end{align}
where $\{\bz_{ij}\}_{i \leq j}$ and $\{\bz_i\}_{i=1}^N$ are standard Gaussian noise
vectors in $\cK$ and $\cH$ respectively, independent of each other and of
$G_* \sim \Haar(\cG_N)$.
The posterior distribution of $G$ given the joint
observations~\eqref{eq:lower-bound-interpolation} is proportional to $\exp
\widetilde H_t(G;G_*,Z)$ for the interpolating Hamiltonian
\begin{equation}
\label{eq:interpolating-hamiltonian-general-model-lower-bound}
\begin{aligned}
    \widetilde H_t(G;G_*,Z)&=\sum_{1 \leq i<j \leq N} -\frac{t}{2N}\|\phi_i \bullet
\phi_j\|_\cK^2+\frac{t}{N}\langle \phi_i \bullet \phi_j,
\phi_{*i} \bullet \phi_{*j} \rangle_\cK+\sqrt{\frac{t}{N}}\langle 
\phi_i \bullet \phi_j, \bz_{ij} \rangle_\cK\\
&\qquad + \sum_{i=1}^N -\frac{t}{4N}\|\phi_i \bullet \phi_i\|_\cK^2
+\frac{t}{2N}\langle \phi_i \bullet \phi_i, \phi_{*i} \bullet \phi_{*i}
\rangle_\cK+\sqrt{\frac{t}{2N}}\langle \phi_i \bullet \phi_i,\bz_{ii}
\rangle_\cK\\
    &\qquad +\sum_{i=1}^N
-\frac{(1-t)}{2}\|\overlap^{1/2}\phi_i\|_\cH^2
+(1-t)\langle \overlap^{1/2}\phi_i,\overlap^{1/2}\phi_{*i}
\rangle_\cH+\sqrt{1-t}\,\langle \overlap^{1/2}\phi_i,\bz_i \rangle_\cH.
\end{aligned}
\end{equation}

For $0\le t \le 1$, we denote the posterior mean
$\langle f(G) \rangle_t=\frac{\E_G[f(G)\exp \widetilde H_t(G;G_*,Z)]}{\E_G[\exp
\widetilde H_t(G;G_*,Z)]}$
(not to be confused with the inner-product notations $\langle \cdot,\cdot
\rangle_\cK$ and $\langle \cdot,\cdot \rangle_\cH$). The Nishimori identity
holds for $\E \langle \cdot \rangle_t$ by Bayes' rule in the model 
\eqref{eq:lower-bound-interpolation}. Define the interpolating free energy
\begin{align*}
\widetilde \cF(t) = \frac{1}{N}\,\E_{G_*,Z}\log
\E_G \exp \widetilde H_t(G;G_*,Z)
\end{align*}
where $\widetilde \cF(1)=\widetilde \cF$ is the free energy of interest. At $t=0$,
applying the identity
\begin{equation}\label{eq:qQidentity}
\sum_{i=1}^N \langle \overlap^{1/2}\phi_i,\overlap^{1/2}\phi_i'
\rangle_\cH=\sum_{i=1}^N \langle \phi_i,\overlap \phi_i' \rangle_\cH
=\sum_{i=1}^N \langle \overlap,\phi_i \otimes \phi_i' \rangle_\cL
=N\langle \overlap,Q(G,G') \rangle_\cL
\end{equation}
for any $G,G' \in \cG_N$, and the group symmetry $Q(G,G)=Q(\iden,\iden)$, we have
\begin{align}
\widetilde \cF(0)&=-\frac{1}{2}\langle \overlap,Q(\iden,\iden)\rangle_\cL
+\frac{1}{N}\,\E_{G_*,Z} \log \E_{G} \exp \inparen{N\langle \overlap,Q(G,G_*)
\rangle_\cL+\sum_{i=1}^N \langle \overlap^{1/2}\phi_i,\bz_i \rangle_\cH}\nonumber\\
&=\Psi(\overlap)+\frac{1}{4}\|\overlap\|_\cL^2
\label{eq:varphi0-general-model}
\end{align}
A calculation based on Gaussian integration by parts shows the derivative of
$\widetilde \cF(t)$.
\begin{proposition}
\label{prop:derivative-varphi-general-model}
\begin{align*}
    \widetilde \cF'(t) =
-\frac{1}{4}\|\overlap\|_\cL^2+\frac{1}{4}\E_{G_*,Z}
\inangle{\|Q(G,G_*)-\overlap\|_\cL^2}_t.
\end{align*}
\end{proposition}
\begin{proof}
First note that
\begin{align*}
\widetilde \cF'(t)=\frac{1}{N}\,\E_{G_*,Z}\inangle{\frac{\ud}{\ud
t}\widetilde H_t(G;G_*,Z)}_t
\end{align*}
where we have
\begin{align*}
    \frac{\ud}{\ud t}\widetilde H_t(G;G_*,Z)&=\sum_{i<j}
-\frac{1}{2N}\|\phi_i \bullet \phi_j\|_\cK^2
+\frac{1}{N}\langle\phi_i \bullet \phi_j,\phi_{*i} \bullet
\phi_{*j}\rangle_\cK
+\frac{1}{2\sqrt{tN}}\langle \phi_i \bullet \phi_j,\bz_{ij} \rangle_\cK
\nonumber \\
&\qquad+\sum_{i=1}^N -\frac{1}{4N}\|\phi_i \bullet \phi_i\|_\cK^2
+\frac{1}{2N}\langle \phi_i \bullet \phi_i,\phi_{*i} \bullet \phi_{*i}
\rangle_\cK
+\frac{1}{2\sqrt{2tN}}\langle \phi_i \bullet \phi_i,\bz_{ii} \rangle_\cK\\
    &\qquad +\sum_{i=1}^N \frac{1}{2} \|\overlap^{1/2}\phi_i\|_\cH^2
-\langle\overlap^{1/2}\phi_i,\overlap^{1/2}
\phi_{*i}\rangle_\cH-\frac{1}{2\sqrt{1-t}}\langle
\overlap^{1/2}\phi_i,\bz_i \rangle_\cH.
\end{align*}

Applying Gaussian integration by parts and denoting $\phi_i'=\phi(G')_i$ for an 
independent sample $G'$ from the posterior law, 
\begin{align*}
\E_Z \inangle{\frac{1}{2\sqrt{tN}}\langle \phi_i \bullet
\phi_j,\bz_{ij}\rangle_\cK }_t
&=\frac{1}{2N}\,\E_{G_*,Z}\Big\langle\|\phi_i \bullet \phi_j\|_\cK^2
-\langle\phi_i \bullet \phi_j, \phi_i' \bullet
\phi_j'\rangle_\cK\Big\rangle_t,\\
\E_Z\inangle{\frac{1}{2\sqrt{2tN}}\langle \phi_i \bullet \phi_i,\bz_{ii}
\rangle_\cK}_t
&=\frac{1}{4N}\,\E_{G_*,Z}\Big\langle\|\phi_i \bullet \phi_i\|_\cK^2
-\langle\phi_i \bullet \phi_i, \phi_i' \bullet
\phi_i'\rangle_\cK\Big\rangle_t,\\
\E_Z\inangle{\frac{1}{2\sqrt{1-t}}\langle
\overlap^{1/2}\phi_i,\bz_i \rangle_\cH}_t
&=\frac{1}{2}\,\E_{G_*,Z}\inangle{\|\overlap^{1/2}\phi_i\|_\cH^2
-\langle \overlap^{1/2}\phi_i,\overlap^{1/2}\phi_i' \rangle_\cH}_t.
\end{align*}
Hence
\begin{align*}
\widetilde \cF'(t)&=\frac{1}{N}\,\E_{G_*,Z} \Bigg\langle\frac{1}{N}\sum_{i<j}
\langle \phi_i \bullet \phi_j,\phi_{*i} \bullet \phi_{*j} \rangle_\cK
-\frac{1}{2N}\sum_{i<j}
\langle \phi_i \bullet \phi_j,\phi_i' \bullet \phi_j' \rangle_\cK
+\frac{1}{2N}\sum_{i=1}^N
\langle \phi_i \bullet \phi_i,\phi_{*i} \bullet \phi_{*i} \rangle_\cK\\
&\hspace{1in}-\frac{1}{4N}\sum_{i=1}^N
\langle \phi_i \bullet \phi_i,\phi_i' \bullet \phi_i' \rangle_\cK
-\sum_{i=1}^N \langle \overlap^{1/2}\phi_i,
\overlap^{1/2} \phi_{*i} \rangle_\cH
+\frac{1}{2}\sum_{i=1}^N \langle \overlap^{1/2}\phi_i,
\overlap^{1/2} \phi_i' \rangle_\cH\Bigg\rangle_t\\
&=\E_{G_*,Z} \Bigg\langle\frac{1}{2N^2}\sum_{i,j=1}^N
\langle \phi_i \bullet \phi_j,\phi_{*i} \bullet \phi_{*j} \rangle_\cK
-\frac{1}{4N^2}\sum_{i,j=1}^N
\langle \phi_i \bullet \phi_j,\phi_i' \bullet \phi_j' \rangle_\cK\\
&\hspace{1in}-\frac{1}{N} \sum_{i=1}^N \langle \overlap^{1/2}\phi_i,
\overlap^{1/2} \phi_{*i} \rangle_\cH
+\frac{1}{2N} \sum_{i=1}^N \langle \overlap^{1/2}\phi_i,
\overlap^{1/2} \phi_i' \rangle_\cH\Bigg\rangle_t.
\end{align*}

By Assumption \ref{assumption:general-model}, for any $G,G' \in \cG_N$, we have
\begin{equation}\label{eq:Q2identity}
\sum_{i,j=1}^N \langle \phi_i \bullet \phi_j,\phi_i' \bullet \phi_j'
\rangle_\cK
=\sum_{i,j=1}^N \langle \phi_i \otimes \phi_i',\phi_j \otimes \phi_j'
\rangle_\cL=N^2\|Q(G,G')\|_\cL^2.
\end{equation}
Applying (\ref{eq:qQidentity}) and (\ref{eq:Q2identity}) to the above gives
\begin{equation}\label{eq:lowerboundderiv}
\widetilde \cF'(t)=\E_{G_*,Z} \inangle{
\frac{1}{2}\|Q(G,G_*)\|_\cL^2-\frac{1}{4}\|Q(G,G')\|_\cL^2
-\langle \overlap,Q(G,G_*) \rangle_\cL
+\frac{1}{2}\langle \overlap,Q(G,G') \rangle_\cL}_t.
\end{equation}
Finally, by Nishimori's identity,
$\E_{G_*,Z} \inangle{f(G,G')}_t=\E_{G_*,Z} \inangle{f(G,G_*)}_t$, so
\[\widetilde \cF'(t)=\E_{G_*,Z} \inangle{
\frac{1}{4}\|Q(G,G_*)\|_\cL^2-\frac{1}{2}\langle \overlap,Q(G,G_*)
\rangle_\cL}_t
=-\frac{1}{4}\|\overlap\|_\cL^2
+\frac{1}{4}\E_{G_*,Z} \inangle{\|Q(G,G_*)-\overlap\|_\cL^2}_t.\]
\end{proof}

\begin{proofof}{Lemma \ref{thm:lower-bound-general-bound}}
For any $\overlap \in \cQ$, applying (\ref{eq:varphi0-general-model}) and
Proposition \ref{prop:derivative-varphi-general-model} with
$\|Q(G,G_*)-\overlap\|_\cL^2 \geq 0$, we have
$\widetilde \cF=\widetilde \cF(0)+\int_0^1 \widetilde \cF'(t)\ud t
\geq \Psi(\overlap)$,
and the result follows upon taking a supremum over $\overlap \in \cQ$.
\end{proofof}


\subsection{Free energy upper bound via the Franz-Parisi potential}
\label{sec:RS_upperBound_general_model}

In this section, we now prove the following upper bound for $\widetilde \cF$.

\begin{lemma}\label{thm:upper-bound-general-model}
In the setting of Theorem \ref{thm:free-energy-general-model},
for any $\epsilon>0$,
\begin{align}
\widetilde \cF \leq \sup_{\overlap \in \cQ} \Psi(\overlap)
+D(\cG_N)\sqrt{\frac{L(\epsilon^{1/2};\cG_N)}{N}}
+\frac{L(\epsilon^{1/2};\cG_N)}{N}+\frac{\epsilon}{2}.
\end{align}
\end{lemma}

Recall that $Q(G,G_*)=N^{-1}\sum_{i=1}^N \phi(G)_i \otimes \phi(G_*)_i \in \cL$. 
For any $\overlapm \in \cL$ and $\epsilon>0$, define the Franz-Parisi potential
\begin{align}
\label{eq:truncated-free-energy}
    \widetilde \Phi_\epsilon(\overlapm) = \frac{1}{N}\, \E_{G_*,Z} \log \E_G
\insquare{\mathbf{1}\{\|Q(G,G_*)-\overlapm\|_\cL^2 \le \epsilon\} \exp \widetilde
H(G;G_*,Z)}.
\end{align}
This is the restriction of the free energy to samples $G$ for which $Q(G,G_*)$
falls close to $\overlapm$. It is clear that $\widetilde \cF \geq
\widetilde \Phi_\epsilon(\overlapm)$ for all $\overlapm \in \cL$;
the following lemma provides a complementary upper bound.

\begin{lemma}\label{lem:upper-bound-Fn-Phi-general-model}
In the setting of Theorem \ref{thm:free-energy-general-model},
for any $\epsilon>0$,
\begin{align}
\label{eq:upper-bound-Fn-1-general-model}
\widetilde \cF \le \sup_{\overlapm \in \cL} \widetilde
\Phi_\epsilon(\overlapm) + D(\cG_N)\sqrt{\frac{L(\epsilon^{1/2};\cG_N)}{N}}
+\frac{L(\epsilon^{1/2};\cG_N)}{N}.
\end{align}
\end{lemma}
\begin{proof}
Let $\cM$ be a $\sqrt{\epsilon}$-cover of $\operatorname{image}(Q)$
in the norm $\|\cdot\|_\cL$ with cardinality
$\log |\cM|=L(\epsilon^{1/2};\cG_N)$.
Then for any $G,G_* \in \cG_N$, some point of $\cM$ must be within
$\sqrt{\epsilon}$-distance of $Q(G,G_*) \in \cL$, so we have
\begin{align}
    \widetilde \cF &\leq \frac{1}{N}\,\E_{G_*, Z} \log \sum_{\overlapm \in
\cM} \E_G \insquare{\mathbf{1}\{\|Q(G,G_*)-\overlapm \|_\cL^2 \le \epsilon\} \exp \widetilde H(G;G_*,Z)} \nonumber \\
    &\leq  \E_{G_*,Z} \max_{\overlapm \in \cM} \underbrace{\frac{1}{N}\log \E_G
\insquare{\mathbf{1}\{\|Q(G,G_*) -\overlapm \|_\cL^2 \le \epsilon\} \exp
\widetilde H(G;G_*,Z)}}_{:=\widetilde \Phi_\epsilon(G_*,Z;\overlapm)}
+\frac{\log |\cM|}{N}.
    \label{eq:upper-bound-Fn-before-conc-measure}
\end{align}

We apply concentration over $Z=\{\bz_{ij}\}_{i \leq j}$
to pass $\E_Z$ inside $\max_{\overlapm \in \cM}$. Define
$\widetilde \Phi_\epsilon(G_*,Z;\overlapm)$ as in
\eqref{eq:upper-bound-Fn-before-conc-measure}, and denote the corresponding
Gibbs average
$\langle f(G) \rangle=\frac{\E_G[f(G)\mathbf{1}\{\|Q(G,G_*)-\overlapm
\|_\cL^2 \le \epsilon\} \exp \widetilde H(G;G_*,Z)]}
{\E_G[\mathbf{1}\{\|Q(G,G_*)-\overlapm\|_\cL^2 \le \epsilon\}
\exp \widetilde H(G;G_*,Z)]}$ (again not to be confused with the inner-products
$\langle \cdot,\cdot \rangle_\cK$ and $\langle \cdot,\cdot \rangle_\cL$).
Then, differentiating
(\ref{eq:mainhamiltonian}) and applying
\eqref{eq:Q2identity} and Jensen's inequality,
\begin{align}
&\sum_{i<j} \|\nabla_{\bz_{ij}} \widetilde \Phi_\epsilon(G_*,Z;\overlapm)\|_\cK^2
+\sum_{i=1}^N \|\nabla_{\bz_{ii}} \widetilde \Phi_\epsilon(G_*,Z;\overlapm)\|_\cK^2
\nonumber\\
&=\sum_{i<j} \left\|\frac{1}{N}\Big\langle \nabla_{\bz_{ij}} \widetilde H(G;G_*,Z)
\Big\rangle\right\|_\cK^2
+\sum_{i=1}^N \left\|\frac{1}{N}\Big\langle \nabla_{\bz_{ii}} \widetilde
H(G;G_*,Z) \Big\rangle\right\|_\cK^2\nonumber\\
&=\frac{1}{2N^3}\sum_{i,j=1}^N \left\|\big\langle \phi_i \bullet \phi_j
\big\rangle\right\|_\cK^2
\leq \frac{1}{2N^3}\inangle{
\sum_{i,j=1}^N \left\|\phi_i \bullet \phi_j \right\|_\cK^2}
=\frac{1}{2N}\inangle{\|Q(G,G)\|_\cL^2} \leq \frac{D(\cG_N)^2}{2N}.
\label{eq:PhiLipschitz}
\end{align}
Therefore, $\widetilde \Phi_\epsilon(G_*,Z;\overlapm)$ is $D(\cG_N)/\sqrt{2N}$-Lipschitz,
so $\E_Z \,e^{\lambda
\inparen{\widetilde \Phi(G_*,Z;\overlapm)-\E_Z\widetilde \Phi(G_*,Z;\overlapm)}}
\leq e^{\lambda^2 D(\cG_N)^2/4N}$ for any $\lambda>0$ by Gaussian concentration of
measure \cite[Theorem 5.5]{boucheron2003concentration}.
Thus, applying also Jensen's inequality,
\begin{align*}
    &\E_Z \max_{\overlapm \in \mathcal{M}}
\inparen{\widetilde \Phi_\epsilon(G_*,Z;\overlapm) - \E_Z \widetilde \Phi_\epsilon(G_*,Z;\overlapm)}
\le \E_Z\,\frac{1}{\lambda} \log 
\sum_{\overlapm \in \mathcal{M}}
e^{\lambda\inparen{\widetilde \Phi_\epsilon(G_*,Z;\overlapm) - \E_Z \widetilde \Phi_\epsilon(G_*,Z;\overlapm)}} \\
    &\hspace{1in}
\le \frac{1}{\lambda} \log \sum_{\overlapm \in \mathcal{M}} \E_Z e^{\lambda
\inparen{\widetilde \Phi_\epsilon(G_*,Z;\overlapm) - \E_Z \widetilde \Phi_\epsilon(G_*,Z;\overlapm)}}
\le \frac{1}{\lambda} \log |\cM|+\frac{\lambda D(\cG_N)^2}{4N}.
\end{align*}
Optimizing over $\lambda$ and applying this to
\eqref{eq:upper-bound-Fn-before-conc-measure}, we get
\begin{equation}\label{eq:before-conc-gstar}
    \widetilde \cF \le \E_{G_*} \max_{\overlapm \in \mathcal{M}} 
\E_Z \widetilde
\Phi_\epsilon(G_*,Z;\overlapm)+D(\cG_N)\sqrt{\frac{\log|\cM|}{N}}
+\frac{\log |\cM|}{N}.
\end{equation}

Next, we claim by the group symmetry $Q(G,G_*)=Q(G_*^{-1}G,\iden)$ that
$\E_Z \widetilde \Phi_\epsilon(G_*,Z;\overlapm)$
has the same value for all $G_* \in \cG_N$, and hence equals
$\widetilde \Phi_\epsilon(\overlapm)$. Indeed, denoting
\[Z(G)=\sum_{i<j} \frac{1}{\sqrt{N}}\langle \phi_i \bullet \phi_j,\bz_{ij}
\rangle_\cK+\sum_{i=1}^N \frac{1}{\sqrt{2N}}\langle \phi_i \bullet \phi_i,\bz_{ii}
\rangle_\cK\]
and applying the definition of $\widetilde H(G;G_*,Z)$ from
\eqref{eq:mainhamiltonian} and the identity \eqref{eq:Q2identity}, we have
\begin{align*}
\widetilde H(G;G_*,Z)
&=-\frac{N}{4}\,Q(G,G)+\frac{N}{2}\,Q(G,G_*)+Z(G)
=-\frac{N}{4}\,Q(\iden,\iden)+\frac{N}{2}\,Q(G_*^{-1}G,\iden)+Z(G).
\end{align*}
Here, $\{Z(G)\}_{G \in \cG_N}$ is a mean-zero Gaussian process
with covariance
\[\E[Z(G)Z(G')]=\sum_{i<j} \frac{1}{N}\langle \phi_i \bullet \phi_j,
\phi_i' \bullet \phi_j' \rangle_\cK
+\sum_{i=1}^N \frac{1}{2N}\langle \phi_i \bullet \phi_i,
\phi_i' \bullet \phi_i' \rangle_\cK=\frac{N}{2}\,Q(G,G').\]
For any fixed $G_* \in \cG_N$, the process $\{Z(G)\}_{G \in \cG_N}$ is then
equal in law to $\{Z(G_*^{-1}G)\}_{G \in \cG_N}$,
because the latter process is also mean-zero with the same covariance
\[\E[Z(G_*^{-1}G)Z(G_*^{-1}G')]
=\frac{N}{2}\,Q(G_*^{-1}G,G_*^{-1}G')=\frac{N}{2}\,Q(G,G').\]
Then, applying also the invariance of Haar measure, this implies
\begin{align}
&\E_Z \widetilde \Phi_\epsilon(G_*,Z;\overlapm)\nonumber\\
&=\frac{1}{N}\,\E_Z \log \E_G
\insquare{\mathbf{1}\{\|Q(G_*^{-1}G,\iden)-\overlapm\|_\cL^2 \le \epsilon\}
\exp\left(-\frac{N}{4}Q(\iden,\iden)+\frac{N}{2}Q(G_*^{-1}G,\iden)
+Z(G_*^{-1}G)\right)}\nonumber\\
&=\frac{1}{N}\,\E_Z \log \E_G
\insquare{\mathbf{1}\{\|Q(G,\iden)-\overlapm\|_\cL^2 \le \epsilon\}
\exp\left(-\frac{N}{4}Q(\iden,\iden)+\frac{N}{2}Q(G,\iden)+Z(G)\right)}
=\E_Z \widetilde \Phi_\epsilon(\iden,Z;\overlapm)\label{eq:PhiInvariant}
\end{align}
so $\E_Z\widetilde \Phi_\epsilon(G_*,Z;\overlapm)=\widetilde \Phi_\epsilon(\overlapm)$ for all $G_* \in
\cG_N$, as claimed. Then \eqref{eq:before-conc-gstar} is equal to the desired
upper bound for $\widetilde \cF$, completing the proof.
\end{proof}

Next, we apply an interpolation argument to upper bound the Franz-Parisi
potential~\eqref{eq:truncated-free-energy}. For any $\overlapm \in \cL$ and
$\overlap \in \cQ \subset \cL$, define
\begin{align*}
\Psi(\overlap,\overlapm)&=\frac{1}{4}\|\overlap\|_\cL^2
-\frac{1}{2}\|\overlapm\|_\cL^2-\frac{1}{2}\langle \overlap,Q(\iden,\iden)
\rangle_\cL\\
&\hspace{1in}
+\frac{1}{N}\,\E_{G_*,Z}\log \E_G \exp\left(N\langle \overlapm,Q(G,G_*)
\rangle_\cL+\sum_{i=1}^N \langle \overlap^{1/2}\phi(G)_i,\bz_i \rangle_\cH\right)
\end{align*}
Note that $\Psi(\overlap,\overlap)=\Psi(\overlap)$ as defined in
\eqref{eq:RS-potential-general-model}.

\begin{lemma}\label{lemma:FPupperbound}
For any $\overlapm \in \cL$, $\overlap \in \cQ$, and $\epsilon>0$,
\[\widetilde \Phi_\epsilon(\overlapm) \leq \inf_{\overlap \in \cQ}
\Psi(\overlap,\overlapm)+\frac{\epsilon}{2}.\]
\end{lemma}
\begin{proof}
Fix any $\overlapm \in \cL$ and $\overlap \in \cQ$. For $0 \leq t
\leq 1$, define now the interpolating Hamiltonian
\[\begin{aligned}
\widetilde H_t(G;G_*,Z)&=\sum_{1 \leq i<j \leq N} -\frac{t}{2N}\|\phi_i \bullet
\phi_j\|_\cK^2+\frac{t}{N}\langle \phi_i \bullet \phi_j,
\phi_{*i} \bullet \phi_{*j} \rangle_\cK+\sqrt{\frac{t}{N}}\langle 
\phi_i \bullet \phi_j, \bz_{ij} \rangle_\cK\\
&\qquad + \sum_{i=1}^N -\frac{t}{4N}\|\phi_i \bullet \phi_i\|_\cK^2
+\frac{t}{2N}\langle \phi_i \bullet \phi_i, \phi_{*i} \bullet \phi_{*i}
\rangle_\cK+\sqrt{\frac{t}{2N}}\langle \phi_i \bullet \phi_i,\bz_{ii}
\rangle_\cK\\
    &\qquad +\sum_{i=1}^N
-\frac{(1-t)}{2}\|\overlap^{1/2}\phi_i\|_\cH^2
+(1-t)\langle \phi_i,\overlapm \phi_{*i} \rangle_\cH
+\sqrt{1-t}\,\langle \overlap^{1/2} \phi_i,\bz_i \rangle_\cH
\end{aligned}\]
which differs from
\eqref{eq:interpolating-hamiltonian-general-model-lower-bound} only by the term
$(1-t)\langle \phi_i,\overlapm \phi_{*i} \rangle_\cH$ in place of
$(1-t)\langle \overlap^{1/2} \phi_i,\overlap^{1/2}\phi_{*i} \rangle_\cH$.

Let
\begin{align}
    \widetilde \Phi_\epsilon(t;\overlapm)=\frac{1}{N}\,\E_{G_*,Z} \log \E_G
\insquare{\mathbf{1}\{\|Q(G,G_*)-\overlapm\|_\cL^2 \le \epsilon\} \exp
\widetilde{H}_t(G;G_*,Z)}
\end{align}
and let $\langle f(G) \rangle_t=\frac{\E_G[f(G)\mathbf{1}\{\|Q(G,G_*)-\overlapm
\|_\cL^2 \le \epsilon\} \exp \widetilde H_t(G;G_*,Z)]}
{\E_G[\mathbf{1}\{\|Q(G,G_*)-\overlapm\|_\cL^2 \le \epsilon\}
\exp \widetilde H_t(G;G_*,Z)]}$ be the corresponding Gibbs average.
Note that $\widetilde \Phi_\epsilon(1;\overlapm)=\widetilde \Phi_\epsilon (\overlapm)$, and when $t=0$ we
have the trivial bound analogous to \eqref{eq:varphi0-general-model}
\begin{align*}
    \widetilde \Phi_\epsilon(0;\overlapm) &\le \frac{1}{N}\,\E_{G_*,Z} \log \E_G \exp\widetilde
H_0(G;G_*,Z)=\Psi(\overlap,\overlapm)-\frac{1}{4}\|\overlap\|_\cL^2
+\frac{1}{2}\|\overlapm\|_\cL^2.
\end{align*}
Applying the same calculation as in the proof of
Proposition~\ref{prop:derivative-varphi-general-model}, we have also analogous
to \eqref{eq:lowerboundderiv} that
\begin{align*}
\widetilde \Phi_\epsilon'(t;\overlapm)
&=\E_{G_*,Z}\inangle{\frac{1}{2}\|Q(G,G_*)\|_\cL^2
-\frac{1}{4}\|Q(G,G')\|_\cL^2-\langle \overlapm,Q(G,G_*) \rangle_\cL
+\frac{1}{2}\langle \overlap,Q(G,G') \rangle_\cL}_t\\
&=-\frac{1}{4}\,\E_{G_*,Z} \inangle{  \|Q(G,G') - \overlap\|_\cL^2 }_t
+\frac{1}{2}\,\E_{G_*,Z} \inangle{ \|Q(G,G_*) - \overlapm\|_\cL^2}_t
+\frac{1}{4}\|\overlap\|_\cL^2-\frac{1}{2}\|\overlapm\|_\cL^2.
\end{align*}
Upper bounding the first negative term by 0 and applying
$\|Q(G,G_*)-\overlapm\|_\cL^2 \leq \epsilon$ with probability 1 under the Gibbs
measure defining $\langle \cdot \rangle_t$, we obtain
\begin{align*}
    \widetilde \Phi_\epsilon'(t;\overlapm) \le
\frac{1}{4}\|\overlap\|_\cL^2-\frac{1}{2}\|\overlapm\|_\cL^2
+\frac{\epsilon}{2}.
\end{align*}
Thus $\Phi_\epsilon(\overlapm)=
\widetilde \Phi_\epsilon(0;\overlapm)+\int_0^1 \widetilde
\Phi_\epsilon'(t;\overlapm)\,\ud t
\leq \Psi(\overlap,\overlapm)+\epsilon/2$, and the lemma follows upon taking
the infimum over $\overlap \in \cQ$.
\end{proof}

Finally, appealing to the closure properties of Assumption
\ref{assumption:general-model}(c), denote by
$|\overlapm|,|\overlapm^\top| \in \cQ$ the elements for which
$\iota(|\overlapm|)^2=\iota(\overlapm)^\top\iota(\overlapm)$ and
$\iota(|\overlapm^\top|)^2=\iota(\overlapm)\iota(\overlapm)^\top$.
The following lemma will allow us to
pass the maximization over $\overlapm \in \cL$ to $|\overlapm| \in \cQ$.

\begin{lemma}\label{lem:psi-positive-general-model}
For any $\overlapm \in \cL$,
\[\Psi(|\overlapm^\top|,\overlapm) \leq \Psi(|\overlapm|,|\overlapm|).\]
\end{lemma}

\begin{proof}
Consider the linear observation model
\[\by_i=A\phi_{*i} + \bz_i \text{ for } i=1,\ldots,N\]
indexed by a parameter $A \in B(\cH)$, where $\phi_{*i}=\phi(G_*)_i$,
$G_* \sim \Haar(\cG_N)$, and $\{\bz_i\}_{i=1}^N$ are i.i.d.\ standard Gaussian
noise vectors in $\cH$. Writing $\phi_i=\phi(G)_i$,
the marginal log-likelihood of $Y_\lin=\{\by_i\}_{i=1}^N$ is given by
\begin{align*}
    \log p_A(Y_\lin)&=\log \E_G \exp \inparen{-\frac{1}{2}\sum_{i=1}^N
\langle A\phi_i,A\phi_i \rangle_\cH +\sum_{i=1}^N \langle A\phi_i,\by_i \rangle_\cH}
-\frac{1}{2}\sum_{i=1}^N\|\by_i\|_\cH^2-\frac{N\dim(\cH)}{2} \log 2\pi.
\end{align*}
Then for $A,B \in B(\cH)$, the Kullback-Liebler divergence in this model is
\begin{align*}
D_{\mathrm{KL}} \inparen{p_A\|p_B}
&:=\E_{Y_\lin \sim p_A} \insquare{\log p_A(Y_\lin)-\log p_B(Y_\lin)}\\ 
    &=\E_{G_*,Z} \log \E_G \exp \inparen{ - \frac{1}{2}\sum_{i=1}^N
\langle A\phi_i,A\phi_i \rangle_\cH
+\sum_{i=1}^N \langle A\phi_i,A\phi_{*i} \rangle_\cH
+\sum_{i=1}^N \langle A\phi_i,\bz_i \rangle_\cH}\\
    &\qquad-\E_{G_*,Z} \log \E_G \exp \inparen{ - \frac{1}{2}\sum_{i=1}^N
\langle B\phi_i,B\phi_i \rangle_\cH
+\sum_{i=1}^N \langle B\phi_i,A\phi_{*i} \rangle_\cH
+\sum_{i=1}^N \langle B\phi_i,\bz_i \rangle_\cH}
\end{align*}

For any $\overlapm \in \cL$, let us write the singular value decomposition
$\iota(\overlapm)=UDV^\top$ and specialize the above to $A=\overlapm_V$
and $B=\overlapm_U$ as defined in \eqref{eq:mUV}.
Then $\iota(\overlapm)=B^\top A$, $\iota(|\overlapm|)=A^\top A$, and
$\iota(|\overlapm^\top|)=B^\top B$, so
\begin{align*}
D_{\mathrm{KL}} \inparen{p_{\overlapm_V}\|p_{\overlapm_U}}
&=\E_{G_*,Z} \log \E_G \exp \inparen{ - \frac{N}{2}\langle |\overlapm|,
Q(G,G)\rangle_\cL+N\langle |\overlapm|,Q(G,G_*) \rangle_\cL
+\sum_{i=1}^N \langle |\overlapm|^{1/2}\phi_i,V^\top \bz_i \rangle_\cH}\\
    &\qquad-\E_{G_*,Z} \log \E_G \exp \inparen{ - \frac{N}{2}\langle
|\overlapm^\top|,Q(G,G)\rangle_\cL+N\langle \overlapm,Q(G,G_*) \rangle_\cL
+\sum_{i=1}^N \langle |\overlapm^\top|^{1/2}\phi_i,U^\top \bz_i \rangle_\cH}.
\end{align*}
Applying that $\{V^\top \bz_i\}_{i=1}^N$ and $\{U^\top \bz_i\}_{i=1}^N$ are both
equal in law to $\{\bz_i\}_{i=1}^N$, that $Q(G,G)=Q(\iden,\iden)$,
and that $\|\overlapm\|_\cL^2=\||\overlapm|\|_\cL^2=\||\overlapm^\top|\|_\cL^2$,
we get $D_{\mathrm{KL}} \inparen{p_{\overlapm_V}\|p_{\overlapm_U}}
=\Psi(|\overlapm|,|\overlapm|)-\Psi(|\overlapm^\top|,\overlapm)$, and the
lemma follows from non-negativity of Kullback–Leibler divergence.
\end{proof}

\begin{proofof}{Lemma \ref{thm:upper-bound-general-model}}
Combining Lemmas \ref{lemma:FPupperbound} and
\ref{lem:psi-positive-general-model}, for any $\overlapm \in \cL$ and
$\epsilon>0$,
\[\widetilde \Phi_\epsilon(\overlapm)
\leq \Psi(|\overlapm^\top|,\overlapm)+\frac{\epsilon}{2}
\leq \Psi(|\overlapm|,|\overlapm|)+\frac{\epsilon}{2}
=\Psi(|\overlapm|)+\frac{\epsilon}{2}.\]
The result follows upon
applying this to Lemma \ref{lem:upper-bound-Fn-Phi-general-model}, and
upper bounding the supremum over $\{|\overlapm|:\overlapm \in \cL\}$ by that
over $\overlap \in \cQ$.
\end{proofof}

Finally, we conclude the proof of Theorem \ref{thm:free-energy-general-model}
by comparing the perturbed free energy $\widetilde \cF$ with the original
free energy $\cF$.

\begin{proofof}{Theorem \ref{thm:free-energy-general-model}}
Define
\[H_t(G;G_*,Z)=H(G;Y)+\sum_{i=1}^N {-}\frac{t}{4N}\|\phi_i \bullet
\phi_i\|_\cK^2+\frac{t}{2N}\langle \phi_i \bullet \phi_i,\phi_{*i} \bullet
\phi_{*i}\rangle_\cK+\sqrt{\frac{t}{2N}}\langle \phi_i \bullet \phi_i,\bz_{ii}
\rangle_\cK,\]
which equals $H(G;Y)$ at $t=0$ and $\widetilde H(G;G_*,Z)$ at $t=1$. 
Set $\cF(t;G_*,Z)=N^{-1}\log \E_G \exp H_t(G;G_*,Z)$, and write $\langle \cdot
\rangle_t$ for the average over the corresponding law of $G$. Then a
calculation similar to that in Proposition
\ref{prop:derivative-varphi-general-model} using Gaussian integration-by-parts
(omitted for brevity) shows, for any fixed $G_* \in \cG_N$,
\[\E_Z \cF'(t;G_*,Z)=\E_Z\bigg[\frac{1}{2N^2}\sum_{i=1}^N 
\Big\langle \langle \phi_i \bullet \phi_i \rangle_t,\phi_{*i} \bullet \phi_{*i}
\Big\rangle_\cK
-\frac{1}{4N^2}\sum_{i=1}^N \Big\|\langle \phi_i \bullet \phi_i
\rangle_t\Big\|_\cK^2\bigg].\]
Recalling $K(\cG_N)$ in \eqref{eq:KGN}, this implies
$|\E_Z \cF'(t;G_*,Z)| \leq 3K(\cG_N)/(4N)$ for all $t \in [0,1]$. Then,
denoting $\cF(G_*,Z)=\cF(0;G_*,Z)$ and $\widetilde \cF(G_*,Z)=\cF(1;G_*,Z)$ and
integrating over $t \in [0,1]$, we have
\begin{equation}\label{eq:perturbationbound}
|\E_Z \cF(G_*,Z)-\E_Z \widetilde \cF(G_*,Z)| \leq 3K(\cG_N)/(4N).
\end{equation}
Then also $|\cF-\widetilde \cF|=|\E_{G_*,Z}\cF(G_*,Z)
-\E_{G_*,Z} \widetilde \cF(G_*,Z)| \leq 3K(\cG_N)/4N$,
and combining with Lemmas \ref{thm:lower-bound-general-bound}
and \ref{thm:upper-bound-general-model} concludes the proof.
\end{proofof}

\subsection{Overlap concentration}
\label{sec:pf-overlapConc}

\begin{proofof}{Corollary \ref{corollary:overlap_concentration_general_model}}
Define the restricted free energy
\[\widetilde \Phi_\epsilon=\frac{1}{N}\,\E_{G_*,Z}\log \E_G \Big[
\mathbf{1}\{\|Q(G,G_*)-\cL_*(\epsilon)\|_\cL^2>\epsilon\}
\exp \widetilde H(G;G_*,Z)\Big].\]
Recall the $\sqrt{\epsilon}$-cover $\cM$ from the proof of Lemma
\ref{lem:upper-bound-Fn-Phi-general-model}.
If $\norm{Q(G,G_*)-\cL_*(\epsilon)}_\cL^2>\epsilon$, then the point
$\overlapm \in \cM$ for which $\norm{Q(G,G_*)-\overlapm}_\cL \leq
\sqrt{\epsilon}$ must not belong to $\cL_*(\epsilon)$, hence
\[\mathbf{1}\{\norm{Q(G,G_*)-\cL_*(\epsilon)}_\cL^2>\epsilon\}
\leq \sum_{\overlapm \in \cM \setminus \cL_*(\epsilon)}
\mathbf{1}\{\norm{Q(G,G_*)- \overlapm}_\cL^2 \leq \epsilon\}.\]
Then the same argument as that of Lemma
\ref{lem:upper-bound-Fn-Phi-general-model} shows
\begin{align*}
\widetilde \Phi_\epsilon
&\leq \sup_{\overlapm \in \cL \setminus \cL_*(\epsilon)}
\widetilde \Phi_\epsilon(\overlapm)+
D(\cG_N)\sqrt{\frac{L(\epsilon^{1/2};\cG_N)}{N}}+\frac{L(\epsilon^{1/2};\cG_N)}{N}.
\end{align*}
By Lemma \ref{lemma:FPupperbound} and the argument in
Lemma \ref{lem:psi-positive-general-model},
\[\widetilde \Phi_\epsilon(\overlapm) \leq
\Psi(|\overlapm^\top|,\overlapm)+\frac{\epsilon}{2}
\leq \Psi(|\overlapm|)-D_{\mathrm{KL}}(p_{\overlapm_V}\|p_{\overlapm_U})
+\frac{\epsilon}{2}.\]
If $\overlapm \notin \cL_*(\epsilon)$, then either
$\Psi(|\overlapm|) \leq \sup_{\overlap \in \cQ} \Psi(\overlap)-\epsilon$ or
${-}D_{\mathrm{KL}}(p_{\overlapm_V}\|p_{\overlapm_U}) \leq {-}\epsilon$. Hence
\begin{equation}\label{eq:restrictedFEupper}
\widetilde \Phi_\epsilon
\leq \sup_{\overlap \in \cQ} \Psi(\overlap)-\frac{\epsilon}{2}
+D(\cG_N)\sqrt{\frac{L(\epsilon^{1/2};\cG_N)}{N}}
+\frac{L(\epsilon^{1/2};\cG_N)}{N}.
\end{equation}

Now recall the original Hamiltonian $H(G;Y)$ from
\eqref{eq:hamiltonian-general-model}, and define
    \begin{align*}
        \widetilde{\Phi}_\epsilon(G_*,Z)&=\frac{1}{N} \log \E_G \Big[
\mathbf{1}\{\|Q(G,G_*)-\cL_*(\epsilon)\|_\cL^2>\epsilon\}
\exp \widetilde H(G;G_*,Z)\Big],\\
        \Phi_\epsilon(G_*,Z)&=\frac{1}{N} \log \E_G \Big[
\mathbf{1}\{\|Q(G,G_*)-\cL_*(\epsilon)\|_\cL^2>\epsilon\} \exp H(G;Y)\Big],\\
        \widetilde{\cF}(G_*,Z)&=\frac{1}{N} \log \E_G \insquare{\exp
\widetilde H(G;G_*,Z)},\\
        \cF(G_*,Z)&=\frac{1}{N} \log \E_G \insquare{\exp H(G;Y)}.
    \end{align*}
The same argument as \eqref{eq:PhiLipschitz} shows that all four quantities are
$D(\cG_N)/\sqrt{2N}$-Lipschitz in $Z$ for any fixed $G_* \in \cG_N$,
the same argument as \eqref{eq:perturbationbound} shows
\[|\E_Z \Phi_\epsilon(G_*,Z)-\E_Z \widetilde \Phi_\epsilon(G_*,Z)|,
|\E_Z \cF_\epsilon(G_*,Z)-\E_Z \widetilde \cF_\epsilon(G_*,Z)|
\leq \frac{3K(\cG_N)}{4N},\]
and the same argument as \eqref{eq:PhiInvariant} shows
$\E_Z \widetilde{\Phi}_\epsilon(G_*,Z)=\widetilde \Phi_\epsilon$ and
$\E_Z \widetilde \cF(G_*,Z)=\widetilde \cF$ for every $G_* \in \cG_N$.
Then by Gaussian concentration of measure
\cite[Theorem 5.6]{boucheron2003concentration}, for any $u>0$,
    \begin{align*}
        \PP_{G_*,Z}\insquare{\Phi_\epsilon(G_*,Z) \geq \widetilde \Phi_\epsilon + u
+\frac{3K(\cG_N)}{4N}}
\leq \PP_{G_*,Z}\insquare{\widetilde \Phi_\epsilon(G_*,Z) \geq \E_Z
\widetilde \Phi_\epsilon(G_*,Z) + u}
\leq \exp\inparen{ -\frac{4Nu^2}{D(\cG_N)^2}},\\
        \PP_{G_*,Z}\insquare{\cF(G_*,Z) \leq \widetilde{\cF} - u -\frac{3K(\cG_N)}{4N}}
\leq \PP_{G_*,Z}\insquare{\widetilde \cF(G_*,Z) \leq \E_Z \widetilde{\cF}(G_*,Z) - u}
\leq \exp\inparen{ - \frac{4Nu^2}{D(\cG_N)^2}}.
    \end{align*}
Choosing $u=\epsilon/8$, it follows that
    \begin{align*}
        &\E_{G_*,Z}\Big\langle{\mathbf{1}\{\norm{Q(G,G_*)-\cL_*(\epsilon)}_\cL^2
>\epsilon\}\Big\rangle}=
\E_{G_*,Z} \frac{\exp N \Phi_\epsilon(G_*,Z)}
{\exp N \cF(G_*,Z)}\\
        &\leq \exp N\inparen{\widetilde \Phi_\epsilon-\widetilde \cF
+ 2\left(\frac{\epsilon}{8}+\frac{3K(\cG_N)}{4N}\right)}
+ 2\exp \inparen{ - \frac{N\epsilon^2}{16D(\cG_N)^2}}.
\end{align*}
Applying \eqref{eq:restrictedFEupper} and ${-}\widetilde \cF \leq
-\sup_{\overlap \in \cQ} \Psi(\overlap)$ from Lemma
\ref{thm:lower-bound-general-bound}, this gives
\begin{align*}
&\E_{G_*,Z}\Big\langle\mathbf{1}\{\norm{Q(G,G_*)-\cL_*(\epsilon)}_\cL^2 >
\epsilon\} \Big\rangle\\
&\leq \exp N\left({-}\frac{\epsilon}{4}
+D(\cG_N)\sqrt{\frac{L(\epsilon^{1/2};\cG_N)}{N}}
+\frac{L(\epsilon^{1/2};\cG_N)}{N}+\frac{3K(\cG_N)}{2N}\right)
+2\exp \inparen{ - \frac{N\epsilon^2}{16D(\cG_N)^2}},
\end{align*}
implying the corollary.
\end{proofof}

\subsection{Mutual information}\label{subsec:MI}
We verify the mutual information relations \eqref{eq:quadraticMI}
and \eqref{eq:linearMI}. For \eqref{eq:quadraticMI},
\begin{align}
\frac{1}{N}\,I(G_*,Y)&=\frac{1}{N}\,\E_{G_*,Z}\log p(Y \mid G_*)
-\frac{1}{N}\,\E_{G_*,Z}\log p(Y)\notag\\
&=\frac{1}{N}\,\E_{G_*,Z}\left[\sum_{i<j} -\frac{1}{2N}\|\by_{ij}
-\phi_{*i} \bullet \phi_{*j}\|_\cK^2\right]
-\frac{1}{N}\,\E_{G_*,Z}\log \E_G \exp\left(\sum_{i<j} -\frac{1}{2N}\|\by_{ij}
-\phi_i \bullet \phi_j\|_\cK^2\right)\notag\\
&=\frac{1}{N}\,\E_{G_*,Z}\left[\sum_{i<j} -\frac{1}{2N}\,\|\phi_{*i} \bullet
\phi_{*j}\|_\cK^2+\frac{1}{N}\langle \by_{ij},\phi_{*i} \bullet \phi_{*j}
\rangle_\cK\right]-\cF\notag\\
&=\frac{1}{2N^2}\,\E_{G_*}\sum_{i<j} \|\phi_{*i} \bullet  \phi_{*j}\|_\cK^2
-\cF
=\frac{1}{4N^2}\,\E_{G_*}\sum_{i,j=1}^N \|\phi_{*i} \bullet  \phi_{*j}\|_\cK^2
-\cF+O\left(\frac{K(\cG_N)}{N}\right)\notag\\
&=\frac{1}{4}\|Q(\iden,\iden)\|_\cL^2-\cF+O\left(\frac{K(\cG_N)}{N}\right).
\label{eq:MIcalculation}
\end{align}
For \eqref{eq:linearMI},
\begin{align*}
\frac{1}{N}\,i(G_*,Y_\lin)&=\frac{1}{N}\,\E_{G_*,Z}\log p(Y_\lin \mid G_*)
-\frac{1}{N}\,\E_{G_*,Z}\log p(Y_\lin)\\
&=\frac{1}{N}\,\E_{G_*,Z}\left[\sum_{i=1}^N -\frac{1}{2}\|\by_i
-\overlap^{1/2}\phi_{*i}\|_\cH^2\right]-
\frac{1}{N}\,\E_{G_*,Z}\log \E_G \exp\left(\sum_{i=1}^N -\frac{1}{2}\|\by_i
-\overlap^{1/2}\phi_i\|_\cH^2\right)\\
&=\frac{1}{N}\,\E_{G_*,Z}\left[\sum_{i=1}^N
-\frac{1}{2}\|\overlap^{1/2}\phi_{*i}\|_\cH^2
+\langle \by_i,\overlap^{1/2}\phi_{*i} \rangle_\cH\right]
-\left(\Psi(\overlap)+\frac{1}{4}\|\overlap\|_\cL^2\right)\\
&=\frac{1}{2N}\,\E_{G_*}\sum_{i=1}^N \langle \overlap,\phi_{*i} \otimes\phi_{*i}
\rangle_\cL-\left(\Psi(\overlap)+\frac{1}{4}\|\overlap\|_\cL^2\right)
=-\frac{1}{4}\|\overlap\|_\cL^2+\frac{1}{2}
\langle \overlap,Q(\iden,\iden)\rangle_\cL-\Psi(\overlap).
\end{align*}

\section{Proofs for group synchronization}\label{sec:pf-group-sync}

\subsection{Asymptotic mutual information and MMSE}


\begin{proofof}{Theorem
\ref{theorem:free-energy-group-syn-multi-representation}}
Define $\cG_N=\cG^N$,
$\cH=\cK=\cL=\prod_{\ell=1}^L \R^{k_\ell \times k_\ell}$,
the feature map $\phi:\cG^N \to \cH^N$ by \eqref{eq:GSphi}, and the bilinear
maps $\bullet,\otimes$ and inclusion map $\iota$ by \eqref{eq:GSbilinearmaps}.
We check the conditions of Assumption \ref{assumption:general-model}:
The compatibility relation \eqref{eq:compatibility} follows from
\[\langle \ab \bullet \bb, \ab' \bullet \bb' \rangle_\cK
=\sum_{\ell=1}^L \lambda_\ell \tr (\ab_\ell \bb_\ell^\top)^\top
(\ab_\ell'\bb_\ell'^\top)
=\sum_{\ell=1}^L \lambda_\ell \tr (\ab_\ell^\top\ab_\ell')^\top
(\bb_\ell^\top\bb_\ell')=\langle \ab \otimes \ab',\bb \otimes \bb' \rangle_\cL\]
for all $\ab,\ab',\bb,\bb' \in \cH$.
The inclusion relation \eqref{eq:inclusion} follows from
\[\langle \ab,\iota(\overlap)\bb \rangle_\cH
=\sum_{\ell=1}^L \tr \ab_\ell^\top
(\sqrt{\lambda_\ell}\bb_\ell\overlap_\ell^\top)
=\sum_{\ell=1}^L \sqrt{\lambda_\ell}\tr \overlap_\ell^\top \ab_\ell^\top
\bb_\ell=\langle \overlap,\ab \otimes \bb \rangle_\cL.\]
From this form, we see that
$\langle \ab,\iota(\overlap)\bb \rangle_\cH=\langle
\bb,\iota(\overlap)\ab \rangle_\cH$ 
and $\langle \ab,\iota(\overlap)\ab \rangle_\cH \geq 0$
for all $\ab,\bb \in \cH$ if and only if, for every $\ell=1,\ldots,L$,
we have $\tr \overlap_\ell^\top \overlapm_\ell=\tr \overlap_\ell^\top
\overlapm_\ell^\top$ for all $\overlapm_\ell \in \R^{k_\ell \times k_\ell}$
and $\tr \overlap_\ell^\top \overlapm_\ell \geq 0$ for all
$\overlapm_\ell \in \Sym_{\succeq 0}^{k_\ell \times k_\ell}$,
i.e.\ if and only if each
$\overlap_\ell$ is symmetric positive-semidefinite. Thus the set $\cQ$ in
\eqref{eq:Qdef} is
\[\cQ=\Sym_{\succeq 0}:=\prod_{\ell=1}^L \Sym_{\succeq 0}^{k_\ell \times
k_\ell}.\]

Let us write the shorthands
\[I=(I_{k_\ell \times k_\ell})_{\ell=1}^L,
\quad \overlapm^\top=(\overlapm_\ell^\top)_{\ell=1}^L,
\quad \overlapm\overlapm'=(\overlapm_\ell\overlapm_\ell')_{\ell=1}^L
\text{ for any } \overlapm,\overlapm' \in \cL.\]
Note then that inclusion map $\iota$ in \eqref{eq:GSbilinearmaps} satisfies
$\iota(\overlapm)^\top=\iota(\overlapm^\top)$ and
$\iota(I)\iota(\overlapm\overlapm')=\iota(\overlapm)\iota(\overlapm')$.
For any $\overlapm \in \cL$,
defining $|\overlapm|=((\overlapm_\ell^\top\overlapm_\ell)^{1/2})_{\ell=1}^L \in
\cQ$ and
$|\overlapm^\top|=((\overlapm_\ell\overlapm_\ell^\top)^{1/2})_{\ell=1}^L \in
\cQ$,
we then have $\iota(\overlapm)^\top \iota(\overlapm)
=\iota(I)\iota(\overlapm^\top \overlapm)
=\iota(|\overlapm|)\iota(|\overlapm|)$ and similarly
$\iota(\overlapm) \iota(\overlapm)^\top
=\iota(|\overlapm^\top|)\iota(|\overlapm^\top|)$. Furthermore
$\|\overlapm_\ell\|_F^2=\||\overlapm|_\ell\|_F^2=\||\overlapm^\top|_\ell\|_F^2$,
verifying the conditions of \eqref{eq:operatorabsolutevalue}. Finally, 
for any $\gb,\hb \in \cG$ and each $\ell=1,\ldots,L$, we have
$(\hb^{-1}\gb)_\ell=\hb_\ell^\top\gb_\ell$ since
$\gb \mapsto \gb_\ell$ is an orthogonal representation of $\cG$. Hence
\[Q(G,H)=\left(\sqrt{\lambda_\ell} \cdot
\frac{1}{N}\sum_{i=1}^N \gb_\ell^{(i)\top}\hb_\ell^{(i)}
\right)_{\ell=1}^L=Q(H^{-1}G,\iden).\]
This verifies all the conditions of Assumption \ref{assumption:general-model}.

\paragraph{Proof of (a):} 
By the independence of the components $(\gb_*^{(i)},\zb^{(i)})_{i=1}^N$, the
expectation $\E_{G_*,Z}$ in \eqref{eq:RS-potential-general-model} is
separable across samples $i=1,\ldots,N$,
yielding together with the above definitions that
\[\Psi_N(\overlap)
=\sum_{\ell=1}^L {-}\frac{1}{4}\|\overlap_\ell\|_F^2
-\frac{\sqrt{\lambda_\ell}}{2}\tr \overlap_\ell
+\E_{\gb_*,\zb}\log \E_{\gb}
\exp\left(\sqrt{\lambda_\ell} \tr \overlap_\ell \gb_\ell^\top\gb_{*\ell}
+\lambda_\ell^{1/4}\overlap_\ell^{1/2}\gb_\ell^\top\zb_\ell\right).\]
In particular, $\Psi_N(\overlap)$ does not depend on $N$.
Defining the change of variables
$\overlap_\ell=\sqrt{\lambda_\ell}\tilde \overlap_\ell$, we then have
$\Psi_N(\overlap)=\Psi_\gs(\tilde \overlap)$.
The quantities $K(\cG_N)$, $D(\cG_N)$, and $L(\sqrt{\epsilon};\cG_N)$ for any
fixed $\epsilon>0$ are bounded by constants
independent of $N$, so Theorem~\ref{thm:free-energy-general-model} implies
\[\lim_{N \to \infty} \cF_N=\sup_{\overlap \in \cQ} \Psi_N(\overlap)
=\sup_{\tilde\overlap \in \cQ} \Psi_\gs(\tilde\overlap).\]
Part (a) of the theorem then follows from \eqref{eq:quadraticMI}, where
$\|Q(\iden,\iden)\|_\cL^2=\sum_{\ell=1}^L \|\sqrt{\lambda_\ell} I_{k_\ell \times
k_\ell}\|_F^2=\sum_{\ell=1}^L \lambda_\ell k_\ell$.


\paragraph{Proof of (b):} The following I-MMSE relation is standard and follows
from similar arguments as in \cite{guo2005mutual}, but we
include a brief proof for convenience. From \eqref{eq:hamiltonian-general-model}
and \eqref{eq:quadraticMI}, we have
\begin{align*}
H(G;G_*, Z)&=-\frac{1}{2N}\sum_{i<j} \sum_{\ell=1}^L
\lambda_\ell \|\gb_\ell^{(i)}\gb_\ell^{(j)\top}\|_F^2
+\frac{1}{N}\sum_{i<j} \sum_{\ell=1}^L
\lambda_\ell \tr (\gb_\ell^{(i)}\gb_\ell^{(j)\top})^\top
(\gb_{*\ell}^{(i)}\gb_{*\ell}^{(j)\top})\\
&\hspace{1in}
+\frac{1}{\sqrt{N}}\sum_{i<j} \sum_{\ell=1}^L \sqrt{\lambda_\ell}
\tr (\gb_\ell^{(i)}\gb_\ell^{(j)\top})^\top \zb_\ell^{(ij)},\\
\frac{1}{N}I(G_*,Y)&=\sum_{\ell=1}^L \frac{\lambda_\ell k_\ell}{4}
-\frac{1}{N}\E_{G_*,Z} \log \E_G \exp H(G; G_*, Z).
\end{align*}
Taking the derivative with respect to $\lambda_\ell$ and applying Gaussian
integration-by-parts gives
\begin{align*}
\partial_{\lambda_\ell} \frac{1}{N} I(G_*,Y) &= \frac{k_\ell}{4} -
\frac{1}{N}\,\E_{G_*,Z} \Bigg\langle - \frac{1}{2N}\sum_{i<j} \|\gb_\ell^{(i)}
\gb_\ell^{(j)\top}\|_F^2 + \frac{1}{N}\sum_{i < j} 
\tr (\gb_\ell^{(i)}\gb_\ell^{(j)\top})^\top
(\gb_{*\ell}^{(i)}\gb_{*\ell}^{(j)\top})\\
&\hspace{2in}+\frac{1}{2\sqrt{\lambda_\ell N}} \sum_{i < j} \tr
(\gb_\ell^{(i)}\gb_\ell^{(j)\top})^\top \zb_\ell^{(ij)} \Bigg\rangle\\
&=\frac{k_\ell}{4} -
\frac{1}{N}\,\E_{G_*,Z} \left[\frac{1}{N}\sum_{i < j} 
\tr \langle \gb_\ell^{(i)}\gb_\ell^{(j)\top} \rangle^\top
\gb_{*\ell}^{(i)}\gb_{*\ell}^{(j)\top}
-\frac{1}{2N} \sum_{i < j} \|\langle \gb_\ell^{(i)}\gb_\ell^{(j)\top}
\rangle\|_F^2\right].
\end{align*}
Then, completing the square and applying
$\|\gb_{*\ell}^{(i)}\gb_{*\ell}^{(j)}\|_F^2=k_\ell$, we obtain the desired
I-MMSE relation
\begin{align*}
\partial_{\lambda_\ell} \frac{1}{N} I(G_*,Y) 
    &= \frac{1}{2N^2} \sum_{i < j} \E_{G_*,Z}
\|\gb_{*\ell}^{(i)}\gb_{*\ell}^{(j)\top} - \langle \gb_\ell^{(i)}
\gb_\ell^{(j)\top} \rangle\|_F^2
+\frac{k_\ell}{4N} = \frac{1 - N^{-1}}{4} \mmse_\ell + \frac{k_\ell}{4N}.
\end{align*}

Fixing any $\{\lambda_{\ell'}:\ell' \neq \ell\}$, observe by properties of
conditional expectation that $\mmse_\ell$ is non-increasing in $\lambda_\ell$,
so this I-MMSE relation implies
$\lambda_\ell \mapsto -\frac{1}{N}I(G_*,Y)$ is convex. Then its pointwise limit
\[\cI(\lambda_\ell):=
\lim_{N \to \infty}{-}\frac{1}{N}I(G_*,Y)
=\sup_{\overlap \in \Sym_{\succeq
0}} \Psi_\gs(\overlap)-\sum_{\ell'=1}^L \frac{\lambda_{\ell'} k_{\ell'}}{4}\]
is also convex, the set $D \subseteq (0,\infty)$ where $\cI(\cdot)$ is
differentiable has full Lebesgue measure,
and for all $\lambda_\ell \in D$ we have $\lim_{N \to \infty}
\partial_{\lambda_\ell}[{-}\frac{1}{N}I(G_*,Y)]=\cI'(\lambda_\ell)$
\cite[Theorems 10.8, 24.6, 25.3]{rockafellar2015convex}.
Applying a change of variables
$\overlapm=(\lambda_\ell \overlap_\ell)_{\ell=1}^L$, we may express
\[\cI(\lambda_\ell)=\sup_{\overlapm \in \Sym_{\succeq 0}}
\cI(\lambda_\ell,\overlapm),
\qquad \cI(\lambda_\ell,\overlapm):=\sum_{\ell'=1}^L
\Big({-}\frac{\lambda_{\ell'} k_{\ell'}}{4}-
\frac{1}{4\lambda_{\ell'}}\|\overlapm_{\ell'}\|_F^2\Big)
+F(\overlapm)\]
for a function $F(\overlapm)$ not depending on $\lambda_\ell$. It may be checked
(via the gradient calculation in Proposition \ref{prop:GSderivatives})
that this function $F(\overlapm)$ is Lipschitz in $\overlapm$, and hence for any
fixed and bounded range of values $\lambda_\ell>0$ the
supremum $\sup_{\overlapm \in \Sym_{\succeq 0}} \cI(\lambda_\ell,\overlapm)$ is
attained on a compact subset of $\Sym_{\succeq 0}$. Then
by the envelope theorem \cite[Corollary 4]{milgrom2002envelope}, $D$ is
precisely the set where $\partial_{\lambda_\ell}
\cI(\lambda_\ell,\overlapm_*)
=-\frac{k_\ell}{4}+\frac{1}{4\lambda_\ell^2} \|\overlapm_{*\ell}\|_F^2$
takes the same value for all
$\overlapm_* \in \argmax_{\overlapm \in \Sym_{\succeq 0}}
\cI(\lambda_\ell,\overlapm)$, and
$\cI'(\lambda_\ell)=-\frac{k_\ell}{4}+\frac{1}{4\lambda_\ell^2} \|\overlapm_{*\ell}\|_F^2$
for any such $\overlapm_*$. Restating this in terms of the original variable
$\overlap$, $D$ is the set where
$\|\overlap_{*\ell}\|_F^2$ takes the same value for all $\overlap_* \in 
\argmax_{\overlap \in \Sym_{\succeq 0}} \Psi_\gs(\overlap)$, and
for any $\lambda_\ell \in D$ we have
\[\lim_{N \to \infty} \mmse_\ell =  4 \lim_{N \to \infty}
\partial_{\lambda_\ell} \frac{1}{N} I(G_*,Y) = -4 \cI'(\lambda_\ell)
=k_\ell-\|\overlap_{*\ell}\|_F^2,\]
showing part (b).

\paragraph{Proof of (c):} We apply
Corollary \ref{corollary:overlap_concentration_general_model}. Consider any
$\overlapm \in \cL$, and write the singular value decompositions
$\overlapm_\ell=\bu_\ell\bd_\ell\bv_\ell^\top \in \R^{k_\ell \times k_\ell}$.
Then $\iota(\overlapm)$ admits a singular value decomposition
$\iota(\overlapm)=UDV^\top$ where $U,V,D \in B(\cH)$ are orthogonal and 
diagonal linear operators defined by
\[U\ab=(\ab_\ell\ub_\ell^\top)_{\ell=1}^L,
\qquad V\ab=(\ab_\ell\vb_\ell^\top)_{\ell=1}^L,
\qquad D\ab=(\sqrt{\lambda_\ell}\ab_\ell \db_\ell)_{\ell=1}^L.\]
So $\overlapm_U,\overlapm_V \in B(\cH)$ are given by
$\overlapm_U\ab=(\lambda_\ell^{1/4}\ab_\ell\ub_\ell\db_\ell^{1/2})_{\ell=1}^L$ and
$\overlapm_V\ab=(\lambda_\ell^{1/4}\ab_\ell\vb_\ell\db_\ell^{1/2})_{\ell=1}^L$, and
$p_{\overlapm_U}$ is the marginal density of
$Y_\lin=\{\by_\ell^{(i)}\}_{1 \leq i \leq N,\, 1 \leq \ell \leq L}$ in the model
with observations
\[\by_\ell^{(i)}=\lambda_\ell^{1/4} \bg_{*\ell}^{(i)}
\bu_\ell\bd_\ell^{1/2}+\bz_\ell^{(i)}.\]
By independence of components for $i=1,\ldots,N$,
$\frac{1}{N}D_{\mathrm{KL}}(p_{\overlapm_V}\|p_{\overlapm_U})$ is equal to
the Kullback-Liebler divergence between the $N=1$ models
\[\{\by_\ell=\lambda_\ell^{1/4}\bg_{*\ell}\bv_\ell\bd_\ell^{1/2}+\bz_\ell\}_{\ell=1}^L
\qquad \text{ and } \qquad
\{\by_\ell=\lambda_\ell^{1/4}\bg_{*\ell}\bu_\ell\bd_\ell^{1/2}+\bz_\ell\}_{\ell=1}^L.\]
In particular, $\frac{1}{N}D_{\mathrm{KL}}(p_{\overlapm_V}\|p_{\overlapm_U})$
does not depend on $N$.

Consider $\cL_*(0)$ corresponding to \eqref{eq:overlap_concentration_set} with
$\epsilon=0$, and suppose $\overlapm \in \cL_*(0)$ where
$\overlapm_\ell=\bu_\ell\bd_\ell\bv_\ell^\top$. Then
$|\overlapm|=(\bv_\ell\bd_\ell\bv_\ell^\top)_{\ell=1}^L
\in \argmax_{\overlap \in \cQ}\Psi_N(\overlap)$, 
and also $\frac{1}{N}D_{\mathrm{KL}}(p_{\overlapm_V}\|p_{\overlapm_U})=0$.
Since $\lambda_\ell>0$, and the law
of any compactly supported random variable $X \in \R^d$ is uniquely determined
by that of $X+Z \in \R^d$ when $Z \sim \cN(0,I)$,
the above characterization of
$\frac{1}{N}D_{\mathrm{KL}}(p_{\overlapm_V}\|p_{\overlapm_U})$
implies that $(\bg_{*\ell}\bv_\ell\bd_\ell^{1/2})_{\ell=1}^L$ is equal in law to
$(\bg_{*\ell}\bu_\ell\bd_\ell^{1/2})_{\ell=1}^L$. Comparing the supports
of these two laws, there must exist $\bg \in \cG$ for which
$(\bg_\ell\bv_\ell\bd_\ell^{1/2})_{\ell=1}^L=(I_{k_\ell \times k_\ell}
\bu_\ell\bd_\ell^{1/2})_{\ell=1}^L$, so
$\overlapm=(\bu_\ell\bd_\ell\bv_\ell^\top)_{\ell=1}^L
=(\bg_\ell\bv_\ell\bd_\ell\bv_\ell^\top)_{\ell=1}^L
=\bg|\overlapm|$. Thus $\cL_*(0) \subseteq \cS:=
\{\bg\bq_*:\bg \in \cG,\,\bq_* \in \argmax_{\overlap \in
\cQ}\Psi_N(\overlap)\}$. The reverse inclusion $\cS \subseteq \cL_*(0)$ is also
evident from reversing these arguments. By the relation
$\Psi_N(\overlap)=\Psi_\gs(\tilde \overlap)$ shown in part (a) where
$\overlap_\ell=\sqrt{\lambda_\ell}\tilde \overlap_\ell$, we have
$\cS=\{(\sqrt{\lambda_\ell}\overlapm_{*\ell})_{\ell=1}^L:\overlapm_* \in
\cL_{*,\gs}\}$ for $\cL_{*,\gs}$ defined in \eqref{eq:GSoverlapconcset}.
So this establishes
\begin{equation}\label{eq:overlapconcexact}
\cL_*(0)=\left\{(\sqrt{\lambda_\ell} \overlapm_{*\ell})_{\ell=1}^L:
\overlapm_* \in \cL_{*,\gs}\right\}.
\end{equation}

Since $\Psi_N(|\overlapm|)$ and
$\frac{1}{N}D_{\mathrm{KL}}(p_{\overlapm_V}\|p_{\overlapm_U})$ defining
$\cL_*(\cdot)$ are both independent of $N$ and continuous in $\overlapm$,
for any $\epsilon>0$, there
must exist $\delta:=\delta(\epsilon)>0$ independent of $N$ for which
\[\cL_*(\delta) \subseteq
\{\overlapm \in \cL:\|\overlapm-\cL_*(0)\|_\cL^2<\epsilon/2\}.\]
Choosing $\delta:=\delta(\epsilon)$ sufficiently small,
by Corollary \ref{corollary:overlap_concentration_general_model}, there then
exist constants $C,c>0$ for which
\[\E \Big\langle \mathbf{1}\Big\{\|Q(G,G_*)-\cL_*(0)\|_\cL^2 \geq \epsilon \Big\} \Big\rangle
\leq \E \Big\langle \mathbf{1}\Big\{\|Q(G,G_*)-\cL_*(\delta)\|_\cL^2>\delta \Big\}
\Big\rangle \leq Ce^{-cN}\]
Applying $Q(G,G_*)=(\sqrt{\lambda_\ell} \cdot \frac{1}{N}\sum_{i=1}^N
\gb_\ell^{(i)\top}\gb_{*\ell}^{(i)})_{\ell=1}^L$, the characterization of
$\cL_*(0)$ in \eqref{eq:overlapconcexact}, and the definition of
$\cL_{*,\gs}(\epsilon)$ in the theorem statement, we have exactly
\[\Big\{\|Q(G,G_*)-\cL_*(0)\|_\cL^2 \geq \epsilon \Big\}
=\left\{\left(\frac{1}{N}\sum_{i=1}^N
\gb_\ell^{(i)\top}\gb_{*\ell}^{(i)}\right)_{\ell=1}^L \notin
\cL_{*,\gs}(\epsilon)\right\},\]
showing part (c).
\end{proofof}

\subsection{Derivatives of the replica potential}

Throughout this section, we abbreviate $\langle \cdot \rangle \equiv \langle
\cdot \rangle_\qb$.

\begin{proofof}{Proposition \ref{prop:GSderivatives}}
Equip $\Sym=\prod_{\ell=1}^L \Sym^{k_\ell \times k_\ell}$ with the
inner-product $\langle\ab,\bb\rangle \mapsto \sum_{\ell=1}^L \tr
\ab_\ell\bb_\ell$, and
consider the orthonormal basis $\{\be_{ij}^\ell:1 \leq \ell \leq L,\,
1 \leq i \leq j \leq k_\ell\}$ of $\Sym$ given by
\[\be_{ij}^\ell=\begin{cases} \be_i\be_i^\top & \text{ if } i=j \\
\frac{1}{\sqrt{2}}(\be_i\be_j^\top+\be_j\be_i^\top) & \text{ if } i<j \end{cases}
\qquad \text{ where }
\be_1,\ldots,\be_{k_\ell} \text{ are the standard basis vectors in }
\R^{k_\ell}.\]
Denote the partial derivatives of a function $f(\cdot)$ in this basis by
\[\partial_{\ell ij} f(\bq)=\lim_{\delta \to 0}
\frac{1}{\delta}[f(\bq_1,\ldots,\bq_{\ell-1},
\bq_\ell+\delta \be_{ij}^\ell,\bq_{\ell+1},\ldots,\bq_L)- f(\bq_1,\ldots, \bq_L)].\]
For any $\bq$ in the interior of $\Sym_{\succeq 0}$, we then have the basis
representations
\begin{align}
\nabla \Psi_\gs(\bq)[\bx]&=\sum_{\ell=1}^L \sum_{1 \leq i \leq j \leq k_\ell}
(\bx \cdot \be_{ij}^\ell)
\partial_{\ell ij} \Psi_\gs(\bq),\label{eq:GradPsicoords}\\
\nabla^2 \Psi_\gs(\bq)[\bx,\bx']&=\sum_{\ell,\ell'=1}^L
\sum_{1 \leq i \leq j \leq k_\ell}
\sum_{1 \leq i' \leq j' \leq k_{\ell'}}
(\bx \cdot \be_{ij}^\ell)(\bx' \cdot \be_{i'j'}^{\ell'})
\partial_{\ell ij} \partial_{\ell' i'j'} \Psi_\gs(\bq),
\label{eq:HessPsicoords}
\end{align}
so it suffices to compute these first- and second-order partial derivatives of
$\Psi_\gs(\overlap)$.

Recall that the replica potential is
\[\Psi_\gs(\overlap)=-\frac{1}{4}\sum_{\ell=1}^L
\lambda_\ell\|\overlap_\ell\|_F^2-\frac{1}{2}\sum_{\ell=1}^L
\lambda_\ell \tr \overlap_\ell
+\E_{\bg_*,\bz} \log \E_{\bg}\exp\left(\sum_{\ell=1}^L \tr
\Big(\lambda_\ell \overlap_\ell \bg_{*\ell}^\top+\sqrt{\lambda_\ell}\,
\overlap_\ell^{1/2}\bz_\ell^\top\Big)\bg_\ell\right).\]
For part (a), consider any $\bq$ in the interior of $\Sym_{\succeq 0}$.
From the definition of $\partial_{\ell ij}$, we have
$\partial_{\ell ij} \bq_{\ell'}=0$ if $\ell \neq \ell'$. For $\ell=\ell'$, we
have
\[\partial_{\ell ij} \bq_\ell=\be_{ij}^\ell,
\qquad \partial_{\ell ij} \bq_\ell^2=
\be_{ij}^\ell\bq_\ell+\bq_\ell \be_{ij}^\ell.\]
Noting that $\bq \mapsto \bq_\ell^{1/2}$ is smooth on the
interior of $\Sym_{\succeq 0}$, denote its partial derivatives by
$\partial_{\ell ij} [\bq_\ell^{1/2}]$. Then, applying $\tr
\be_{ij}^\ell=\mathbf{1}\{i=j\}$, we have
\[\partial_{\ell ij} \Psi_\gs(\bq)
=-\frac{\lambda_\ell}{2}\tr \be_{ij}^\ell{\bq_\ell}
-\frac{\lambda_\ell}{2}\mathbf{1}\{i=j\}
+\E_{\bg_*,\bz}\left[\tr \Big(\lambda_\ell \be_{ij}^\ell \bg_{*\ell}^\top 
+\sqrt{\lambda_\ell}\partial_{\ell ij}[\bq_\ell^{1/2}]\bz_\ell^\top \Big)
\langle \bg_\ell\rangle\right].\]
Setting $f(\bg)=\sqrt{\lambda_\ell}\,\bg_\ell \partial_{\ell ij}[\bq_\ell^{1/2}]
\in \R^{k_\ell \times k_\ell}$ and applying Gaussian
integration-by-parts in the form
\[\E_{\bz} \tr \bz_\ell^\top \langle f(\bg) \rangle
=\sum_{i,j=1}^{k_\ell} \E_{\bz} \frac{\partial}{\partial z_{\ell ij}}
\langle f_{ij}(\bg) \rangle=\sqrt{\lambda_\ell}\,
\E_{\bz} \tr \Big(\bq_\ell^{1/2} \langle \bg_\ell^\top f(\bg) \rangle
-\bq_\ell^{1/2} \langle \bg_\ell \rangle^\top \langle f(\bg) \rangle\Big)\]
we obtain
\begin{align}
\partial_{\ell ij} \Psi_\gs(\bq)
&=-\frac{\lambda_\ell}{2}\tr \be_{ij}^\ell{\bq_\ell}
-\frac{\lambda_\ell}{2}\mathbf{1}\{i=j\}+\lambda_\ell\,\E_{\bg_*,\bz} 
\tr \be_{ij}^\ell \bg_{*\ell}^\top \langle \bg_\ell \rangle\notag\\
&\hspace{1in}
+\lambda_\ell\,\E_{\bg_*,\bz} \Big\langle \tr \bq_\ell^{1/2}\bg_\ell^\top 
 \bg_\ell \partial_{\ell ij}[\bq_\ell^{1/2}]\Big \rangle
-\lambda_\ell\,\E_{\bg_*,\bz} \tr \Big[\bq_\ell^{1/2}\langle \bg_\ell\rangle^\top 
 \langle \bg_\ell\rangle \partial_{\ell ij}[\bq_\ell^{1/2}]\Big].
\label{eq:firstderivtmp}
\end{align}
Differentiating implicitly
$\bq_\ell^{1/2}\bq_\ell^{1/2}=\bq_\ell$ gives
$\partial_{\ell ij}[\bq_\ell^{1/2}]\bq_\ell^{1/2}
+\bq_\ell^{1/2}\partial_{\ell ij}[\bq_\ell^{1/2}]=\be_{ij}^\ell$.
Thus, since $\bq_\ell^{1/2}$, $\partial_{\ell ij} [\bq_\ell^{1/2}]$, and
$\be_{ij}^\ell$ are all symmetric, for any $\ba \in \R^{k_\ell \times k_\ell}$
we have
\begin{equation}\label{eq:rootqidentity}
\tr \Big(\bq_\ell^{1/2}\ba^\top \ba\partial_{\ell ij}[\bq_\ell^{1/2}]\Big)
=\frac{1}{2}\tr \Big(\big[
\partial_{\ell ij}[\bq_\ell^{1/2}]\bq_\ell^{1/2}
+\bq_\ell^{1/2}\partial_{\ell ij}[\bq_\ell^{1/2}]\big]\ba^\top\ba\Big)
=\frac{1}{2}\tr \be_{ij}^\ell\ba^\top\ba.
\end{equation}
Applying this to (\ref{eq:firstderivtmp}) and noting that
$\bg_\ell^\top\bg_\ell=I_{k_\ell \times k_\ell}$
because $\bg_\ell$ is orthogonal,
the second and fourth terms of (\ref{eq:firstderivtmp}) cancel and we obtain
\[\partial_{\ell ij} \Psi_\gs(\bq)
=-\frac{\lambda_\ell}{2}\tr \be_{ij}^\ell\bq_\ell
+\lambda_\ell\,\E_{\bg_*,\bz} \tr \be_{ij}^\ell
\bg_{*\ell}^\top\langle\bg_\ell \rangle
-\frac{\lambda_\ell}{2}\,\E_{\bg_*,\bz} \tr \be_{ij}^\ell \langle
\bg_\ell \rangle^\top \langle \bg_\ell \rangle.\]
By the Nishimori identity,
\[\E_{\bg_*,\bz} \tr \be_{ij}^\ell \bg_{*\ell}^\top
\langle \bg_\ell \rangle
=\E_{\bg_*,\bz} \tr \be_{ij}^\ell \langle \bg_\ell \rangle^\top
\langle \bg_\ell \rangle.\]
Thus
\[\partial_{\ell ij} \Psi_\gs(\bq)
=-\frac{\lambda_\ell}{2}\tr \be_{ij}^\ell \Big(\bq_\ell-
\E_{\bg_*,\bz}\bg_{*\ell}^\top\langle \bg_\ell \rangle\Big)
=-\frac{\lambda_\ell}{2}\tr \be_{ij}^\ell \Big(\bq_\ell-
\E_{\bg_*,\bz} \langle \bg_\ell \rangle^\top \langle \bg_\ell \rangle\Big).\]
Applying this in (\ref{eq:GradPsicoords}) and using
$\sum_{1 \leq i \leq j \leq k_\ell} (\bx \cdot \be_{ij}^\ell)\be_{ij}^\ell
=\bx_\ell$ gives (\ref{eq:GradPsi}). It is clear that the right side of
(\ref{eq:GradPsi}) extends continuously to the boundary of $\Sym_{\succeq 0}$,
thus establishing (\ref{eq:GradPsi}) for all $\bq \in \Sym_{\succeq 0}$.
Also from this form (\ref{eq:GradPsi}), we have
$\nabla \Psi_\gs(\bq)=0$, i.e.\ $\nabla \Psi_\gs(\bq)[\bx]=0$ for all $\bx \in \Sym$,
if and only if $\bq_\ell=\E_{\bg_*,\bz} \langle \bg_\ell \rangle^\top
\langle \bg_\ell \rangle$ for all $\ell=1,\ldots,L$. This
shows all claims of part (a).

For part (b), let us compute the second partial derivatives from the form
of the first derivative
\[\partial_{\ell' i'j'} \Psi_\gs(\bq)
=-\frac{\lambda_{\ell'}}{2}\tr \be_{i'j'}^{\ell'}\Big(\bq_{\ell'}-
\E_{\bg_*,\bz}\bg_{*\ell'}^\top \langle \bg_{\ell'} \rangle \Big).\]
For $\bq$ in the interior of $\Sym_{\succeq 0}$, taking
$\partial_{\ell ij}$ using the orthogonality relation
\[\partial_{\ell ij}\tr \be_{i'j'}^{\ell'}\bq_{\ell'}
=\mathbf{1}\{\ell=\ell'\}\tr \be_{i'j'}^\ell \be_{ij}^\ell
=\mathbf{1}\{(\ell,i,j)=(\ell',i',j')\}\]
for the first term, we get
\begin{equation}\label{eq:HessPsitmp}
\partial_{\ell ij}\partial_{\ell'i'j'}\Psi_\gs(\bq)
=-\frac{\lambda_\ell}{2}\mathbf{1}\{(\ell,i,j)=(\ell',i',j')\}
+\frac{\lambda_{\ell'}}{2}\,\E_{\bg_*,\bz}\tr \be_{i'j'}^{\ell'}
\bg_{*\ell'}^\top \partial_{\ell ij} \langle \bg_{\ell'} \rangle.
\end{equation}

Let us abbreviate
\[\bbm_\ell=\lambda_\ell \be_{ij}^\ell \bg_{*\ell}^\top \bg_\ell
+\sqrt{\lambda_\ell}\partial_{\ell ij}[\bq_\ell^{1/2}]\bz_\ell^\top \bg_\ell,\]
momentarily write $\otimes$ for the usual vector space tensor product (in this
calculation only, not to be confused with the bilinear map $\otimes$ in the
rest of the paper), and denote the linear maps
$(\tr \otimes \id)(\ba \otimes \bb)=(\tr \ba)\,\bb$ and
$(\tr \otimes \tr)(\ba \otimes \bb)=(\tr \ba)(\tr \bb)$. Then
\begin{align}
\partial_{\ell ij} \langle \bg_{\ell'} \rangle
&=\langle (\tr \bbm_\ell) \bg_{\ell'} \rangle
-\langle \tr \bbm_\ell \rangle \langle \bg_{\ell'} \rangle
=\tr \otimes \id\Big(\langle \bbm_\ell \otimes \bg_{\ell'} \rangle
-\langle \bbm_\ell \rangle \otimes \langle \bg_{\ell'} \rangle\Big).
\label{eq:dgavg}
\end{align}
To simplify the contributions from the second term of $\bbm_\ell$ involving
$\bz_\ell$, we apply Gaussian integration-by-parts in the forms
\begin{align*}
\E_{\bz} \tr \otimes \id \Big\langle
\bz_\ell^\top f(\bg_\ell) \otimes \bg_{\ell'}\Big\rangle
&=\E_{\bz} \sum_{i,j=1}^{k_\ell}
\frac{\partial}{\partial z_{\ell ij}}
\langle f_{ij}(\bg_\ell) \bg_{\ell'} \rangle\\
&=\sqrt{\lambda_\ell}\,\E_{\bz} \tr \otimes \id\left[
\Big\langle \bq_\ell^{1/2} \bg_\ell^\top f(\bg_\ell) \otimes \bg_{\ell'}
\Big\rangle-\Big\langle \bq_\ell^{1/2} \langle \bg_\ell \rangle^\top f(\bg_\ell)
\otimes \bg_{\ell'} \Big\rangle\right],\\
\E_{\bz} \tr \otimes \id \Big[\langle \bz_\ell^\top f(\bg_\ell) \rangle
\otimes \langle \bg_{\ell'} \rangle\Big]
&=\E_{\bz} \sum_{i,j=1}^{k_\ell}
\frac{\partial}{\partial z_{\ell ij}}
[\langle f_{ij}(\bg_\ell) \rangle \langle \bg_{\ell'} \rangle]\\
&=\sqrt{\lambda_\ell}\,\E_{\bz}
\tr \otimes \id \Big[\langle \bq_\ell^{1/2}\bg_\ell^\top f(\bg_\ell)
\rangle \otimes \langle \bg_{\ell'} \rangle
+\Big\langle \bq_\ell^{1/2}\bg_\ell^\top \langle f(\bg_\ell) \rangle
\otimes \bg_{\ell'} \Big\rangle\\
&\hspace{1in}-2\bq_\ell^{1/2}\langle \bg_\ell\rangle^\top \langle f(\bg_\ell)
\rangle \otimes \langle \bg_{\ell'} \rangle\Big].
\end{align*}
Setting $f(\bg_\ell)=\sqrt{\lambda_\ell}\,\bg_\ell
\partial_{\ell ij}[\bq_\ell^{1/2}]$ as before and
taking the difference of the above two expressions,
\begin{align*}
&\E_{\bz} \tr \otimes \id
\Big[\Big\langle \sqrt{\lambda_\ell} \bz_\ell^\top \bg_\ell\partial_{\ell
ij}[\bq_\ell^{1/2}] \otimes \bg_{\ell'}\Big\rangle
-\Big\langle \sqrt{\lambda_\ell} \bz_\ell^\top \bg_\ell\partial_{\ell
ij}[\bq_\ell^{1/2}] \Big \rangle \otimes \langle \bg_{\ell'}\rangle
\Big]\\
&=\lambda_\ell\,\E_{\bz} \tr \otimes \id
\Big[\Big\langle \bq_\ell^{1/2}
(\bg_\ell-\langle \bg_\ell \rangle)^\top
(\bg_\ell-\langle \bg_\ell \rangle) 
\partial_{\ell ij}[\bq_\ell^{1/2}]
\otimes \bg_{\ell'} \Big \rangle\\
&\hspace{1in}-\Big\langle \bq_\ell^{1/2}\bg_\ell^\top \bg_\ell
\partial_{\ell ij}[\bq_\ell^{1/2}]
\Big\rangle \otimes \langle \bg_{\ell'} \rangle
+\bq_\ell^{1/2}\langle \bg_\ell\rangle^\top 
\langle \bg_\ell \rangle \partial_{\ell ij}[\bq_\ell^{1/2}]
\otimes \langle \bg_{\ell'} \rangle \Big]\\
&=\frac{\lambda_\ell}{2}\,\E_{\bz} \tr \otimes \id
\Big[\Big\langle \be_{ij}^\ell
(\bg_\ell-\langle \bg_\ell \rangle)^\top
(\bg_\ell-\langle \bg_\ell \rangle) \otimes \bg_{\ell'} \Big \rangle
-\langle \be_{ij}^\ell \bg_\ell^\top \bg_\ell\rangle
\otimes \langle \bg_{\ell'} \rangle
+\be_{ij}^\ell \langle \bg_\ell \rangle^\top\langle \bg_\ell\rangle
\otimes \langle \bg_{\ell'} \rangle \Big]
\end{align*}
where the last equality applies (\ref{eq:rootqidentity}) with $\ba \in
\{\bg_\ell-\langle \bg_\ell \rangle,\bg_\ell,\langle \bg_\ell \rangle\}$.
Expanding the square in the first term, cancelling the terms involving
$\tr \be_{ij}^\ell \bg_\ell^\top\bg_\ell=\mathbf{1}\{i=j\}$,
and applying
$\tr \be_{ij}^\ell \langle \bg_\ell \rangle^\top \bg_\ell
=\tr \be_{ij}^\ell \bg_\ell^\top \langle \bg_\ell\rangle$, we get
\begin{align*}
&\E_{\bz} \tr \otimes \id
\Big[\Big\langle \sqrt{\lambda_\ell} \bz_\ell^\top \bg_\ell\partial_{\ell
ij}[\bq_\ell^{1/2}] \otimes \bg_{\ell'}\Big\rangle
-\Big\langle \sqrt{\lambda_\ell} \bz_\ell^\top \bg_\ell\partial_{\ell
ij}[\bq_\ell^{1/2}] \Big \rangle \otimes \langle \bg_{\ell'}\rangle
\Big]\\
&=\lambda_\ell\,\E_\bz \tr \otimes \id\Big[{-}\langle
\be_{ij}^\ell \bg_\ell^\top\langle \bg_\ell \rangle \otimes \bg_{\ell'}
\rangle+\be_{ij}^\ell \langle \bg_\ell \rangle^\top\langle \bg_\ell \rangle
\otimes \langle \bg_{\ell'} \rangle\Big]
\end{align*}
Combining this with the contributions from the first term of $\bbm_\ell$
and substituting into (\ref{eq:dgavg}), we arrive at
\[\E_{\bz} \partial_{\ell ij}\langle \bg_{\ell'} \rangle
=\lambda_\ell\,\E_{\bz} \tr \otimes \id\Big[\langle
\be_{ij}^\ell \bg_{*\ell}^\top \bg_\ell \otimes \bg_{\ell'} \rangle
-\be_{ij}^\ell \bg_{*\ell}^\top \langle \bg_\ell \rangle \otimes
\langle \bg_{\ell'} \rangle
-\langle \be_{ij}^\ell \bg_\ell^\top \langle \bg_\ell \rangle \otimes
\bg_{\ell'} \rangle
+\be_{ij}^\ell \langle \bg_\ell \rangle^\top\langle \bg_\ell \rangle
\otimes \langle \bg_{\ell'} \rangle\Big].\]
Applying this back to (\ref{eq:HessPsitmp}) and using again 
$\tr \be_{ij}^\ell \langle \bg_\ell \rangle^\top\bg_\ell
=\tr \be_{ij}^\ell \bg_\ell^\top \langle \bg_\ell\rangle$ and Nishimori's
identity,
\begin{align*}
\partial_{\ell ij}\partial_{\ell'i'j'}\Psi_\gs(\bq)
&=-\frac{\lambda_\ell}{2}\mathbf{1}\{(\ell,i,j)=(\ell',i',j')\}\\
&\qquad
+\frac{\lambda_\ell\lambda_{\ell'}}{2}\,\E_{\bg_*,\bz}\tr \otimes \tr
\Big[\langle \be_{ij}^\ell \bg_{*\ell}^\top \bg_\ell \otimes
\be_{i'j'}^{\ell'} \bg_{*\ell'}^\top \bg_{\ell'} \rangle
-\be_{ij}^\ell \bg_{*\ell}^\top \langle \bg_\ell \rangle \otimes
\be_{i'j'}^{\ell'}\bg_{*\ell'}^\top \langle \bg_{\ell'} \rangle\\
&\hspace{1in}-\langle
\be_{ij}^\ell \bg_\ell \langle \bg_\ell^\top \rangle \otimes 
\be_{i'j'}^{\ell'}\bg_{*\ell'}^\top \bg_{\ell'}\rangle
+\be_{ij}^\ell \langle \bg_\ell \rangle^\top\langle \bg_\ell \rangle
\otimes \be_{i'j'}^{\ell'}\bg_{*\ell'}^\top\langle \bg_{\ell'} \rangle
\Big]\\
&=-\frac{\lambda_\ell}{2}\mathbf{1}\{(\ell,i,j)=(\ell',i',j')\}\\
&\qquad
+\frac{\lambda_\ell\lambda_{\ell'}}{2}\,\E_{\bg_*,\bz} 
\Big[\langle \tr \be_{ij}^\ell \bg_{*\ell}^\top \bg_\ell
\tr \be_{i'j'}^{\ell'} \bg_{*\ell'}^\top \bg_{\ell'} \rangle
-2\tr \be_{ij}^\ell \bg_{*\ell}^\top \langle \bg_\ell \rangle
\tr \be_{i'j'}^{\ell'}\bg_{*\ell'}^\top \langle \bg_{\ell'} \rangle\\
&\hspace{1in}
+\tr \be_{ij}^\ell \langle \bg_\ell \rangle^\top\langle \bg_\ell \rangle
\tr \be_{i'j'}^{\ell'}\langle \bg_{\ell'}
\rangle^\top\langle\bg_{\ell'}\rangle\Big].
\end{align*}
Applying this in (\ref{eq:HessPsicoords}) and using 
$\sum_{1 \leq i \leq j \leq k_\ell} (\bx \cdot \be_{ij}^\ell)
(\bx' \cdot \be_{ij}^\ell)=\tr \bx_\ell\bx_\ell'$,
$\sum_{1 \leq i \leq j \leq
k_\ell} (\bx \cdot \be_{ij}^\ell)\be_{ij}^\ell=\bx_\ell$, and the analogous
identity for $\bx'$ gives (\ref{eq:HessPsi}). Again, the right side of
(\ref{eq:HessPsi}) extends continuously to the boundary of $\Sym_{\succeq 0}$,
establishing (\ref{eq:HessPsi}) for all $\bq \in \Sym_{\succeq 0}$ and showing
part (b).
\end{proofof}

\begin{proofof}{Proposition \ref{prop:GSlocaloptimality}}
Let us write $\E$ for the expectation over independent and uniformly random
elements $\bg,\bh \sim \Haar(\cG)$, with corresponding representations
$(\bg_1,\ldots,\bg_L)$ and $(\bh_1,\ldots,\bh_L)$.

We use the definition of the type of a real-irreducible representation
$\bg_\ell$ following Theorem \ref{thm:realclassification}.
If the representation $\bg_\ell$ is of real type, then it is $\C$-irreducible.
Since it is also non-trivial, Schur orthogonality (Theorem
\ref{thm:schurorthogonality}(a)) implies that
$\E[g_{\ell ij} \cdot 1]=0$ for each entry $(i,j)$ of $\bg_\ell$,
where 1 represents the trivial representation in $\C^{1 \times 1}$;
thus $\E[\bg_\ell]=0$.
If $\bg_\ell$ is of complex or quaternionic type, then the
same argument applies to the entries of the two $\C$-irreducible
sub-representations of $\bg_\ell$. Thus in all cases, $\E[\bg_\ell]=0$.

At $\bq=\bzero$, a sample $\bg$ from the posterior
measure defining $\langle \cdot \rangle_\bg$
is uniform over $\cG$ and independent of $\bg_*$. Thus
\[\nabla \Psi(\bzero)[\bx]=\sum_{\ell=1}^L \frac{\lambda_\ell}{2}
\tr \bx_\ell (\E\bg_\ell)^\top(\E\bg_\ell)=0\]
for any $\bx \in \Sym$, showing the first claim that
$\nabla \Psi(\bzero)=0$. Furthermore, applying $\E \bg_\ell=0$,
\[\nabla^2 \Psi(\bzero)[\bx,\bx']=\sum_{\ell=1}^L -\frac{\lambda_\ell}{2}
\tr \bx_\ell\bx_\ell'+\sum_{\ell,\ell'=1}^L
\frac{\lambda_\ell\lambda_{\ell'}}{2}
\E\Big[(\tr \bx_\ell\bg_\ell^\top\bh_\ell)
(\tr \bx_{\ell'}'\bg_{\ell'}^\top\bh_{\ell'})\Big].\]
By invariance of Haar measure, we have the equality in law
$\bg^\top\bh \overset{L}{=}\bg$.
Furthermore, if $\ell \neq \ell'$, then $\bg_\ell$ and $\bg_{\ell'}$ are
distinct and real-irreducible, so the $\C$-irreducible
sub-representations of $\bg_\ell$ are distinct from those of $\bg_{\ell'}$
(c.f.\ Theorem \ref{thm:realclassification}).
Then Schur orthogonality (Theorem \ref{thm:schurorthogonality}(a)) 
implies $\E[(\tr \bx_\ell\bg_\ell)(\tr
\bx_{\ell'}\bg_{\ell'})]=0$. Thus
\[\nabla^2 \Psi(\bzero)[\bx,\bx']=\sum_{\ell=1}^L
\underbrace{-\frac{\lambda_\ell}{2}
\tr \bx_\ell\bx_\ell'+\frac{\lambda_\ell^2}{2}
\E\Big[(\tr \bx_\ell\bg_\ell)(\tr
\bx_\ell'\bg_\ell)\Big]}_{=:H_\ell[\bx_\ell,\bx_\ell']}\]
This shows that $\nabla^2 \Psi(\bzero)$ is block-diagonal in the $L \times L$
block decomposition with respect to $\bx=(\bx_1,\ldots,\bx_L)$, with blocks
$\{H_\ell\}_{\ell=1}^L$, so its largest eigenvalue satisfies
\[\lambda_{\max}(\nabla^2 \Psi(\bzero))=\max_{\ell=1}^L \lambda_{\max}(H_\ell)
=\max_{\ell=1}^L \sup_{\bx_\ell \in \Sym^{k_\ell \times
k_\ell}:\|\bx_\ell\|_F^2=k_\ell} \frac{1}{k_\ell}\,H_\ell[\bx_\ell,\bx_\ell].\]

If $\bg_\ell$ is of real type, then it is $\C$-irreducible, and
Theorem \ref{thm:schurorthogonality}(a) gives
\[\E[(\tr \bx_\ell\bg_\ell)^2]
=\E[(\tr \bx_\ell\bg_\ell)(\tr \overline{\bx_\ell\bg_\ell})]
=\sum_{i,j,i',j'=1}^{k_\ell} x_{\ell ij} \overline{x_{\ell i'j'}}
\E[g_{\ell ij}\overline{g_{\ell i'j'}}]=\frac{1}{k_\ell}\sum_{i,j=1}^{k_\ell}
x_{\ell ij}\overline{x_{\ell i'j'}}=\frac{1}{k_\ell}\|\bx_\ell\|_F^2.\]
Thus
\[\sup_{\bx_\ell:\|\bx_\ell\|_F^2=k_\ell} \E[(\tr \bx_\ell\bg_\ell)^2]
=\E[(\tr \bg_\ell)^2]=1\]
where the first equality holds because the supremum is attained at any
$\bx_\ell$ satisfying $\|\bx_\ell\|_F^2=k_\ell$, and in particular at
$\bx_\ell=I_{k_\ell \times k_\ell}$.

If $\bg_\ell$ is of complex type, then there exists a
unitary matrix $(\bv_1\;\bv_2) \in \C^{k_\ell \times k_\ell}$ for which
\begin{equation}\label{eq:gelldecomp}
\bg_\ell=\begin{pmatrix} \bv_1 & \bv_2 \end{pmatrix}
\begin{pmatrix} \bg_\ell^{(1)} & \bzero \\ \bzero & \bg_\ell^{(2)}
\end{pmatrix}\begin{pmatrix} \bv_1^* \\ \bv_2^* \end{pmatrix}
\end{equation}
and $\bg_\ell^{(1)},\bg_\ell^{(2)} \in \C^{k_\ell/2 \times k_\ell/2}$ are the
two $\C$-irreducible unitary sub-representations of $\bg_\ell$
(c.f.\ Theorem \ref{thm:completereducibility}). Here,
$\bg_\ell^{(2)}$ is distinct from $\bg_\ell^{(1)}$
and isomorphic to the complex conjugate representation
$\bar{\bg}_\ell^{(1)}$. Then Theorem \ref{thm:schurorthogonality}(a) 
gives, similarly as above,
\[\E[(\tr \bx_\ell \bg_\ell)^2]
=\E[(\tr \bv_1^*\bx_\ell \bv_1 \bg_\ell^{(1)}
+\tr \bv_2^*\bx_\ell \bv_2 \bg_\ell^{(2)})^2]
=\frac{1}{k_\ell/2}\|\bv_1^*\bx_\ell \bv_1\|_F^2
+\frac{1}{k_\ell/2}\|\bv_2^*\bx_\ell \bv_2\|_F^2.\]
We have $\|\bv_1^*\bx_\ell \bv_1\|_F^2+\|\bv_2^*\bx_\ell \bv_2\|_F^2
\leq \|\bx_\ell\|_F^2$, where equality is again attained at
$\bx_\ell=I_{k_\ell \times k_\ell}$. Then
\[\sup_{\bx_\ell:\|\bx_\ell\|_F^2=k_\ell} \E[(\tr \bx_\ell\bg_\ell)^2]
=\E[(\tr \bg_\ell)^2]=2.\]

Finally, if $\bg_\ell$ is of quaternionic type, then again (\ref{eq:gelldecomp})
holds where, now, $\bg_\ell^{(1)},\bg_\ell^{(2)}$ are isomorphic
$\C$-irreducible sub-representations of $\bg_\ell$ (and both isomorphic to
$\bar{\bg}_\ell^{(1)}$). Then there exists a unitary matrix
$\bu \in \C^{k_\ell/2 \times k_\ell/2}$ for which
$\bg_\ell^{(1)}=\bu^*\bg_\ell^{(2)}\bu$ (c.f.\ Proposition
\ref{prop:unitaryequivalence}). Replacing $(\bv_2,\bg_\ell^{(2)})$
by $(\bv_2\bu,\bu^*\bg_\ell^{(2)}\bu)$, we may assume that
$\bg_\ell^{(1)}=\bg_\ell^{(2)}$. Then by
Theorem \ref{thm:schurorthogonality}(a),
\[\E[(\tr \bx_\ell \bg_\ell)^2]
=\E\Big[\Big(\tr (\bv_1^*\bx_\ell \bv_1+\bv_2^*\bx_\ell \bv_2)
\bg_\ell^{(1)}\Big)^2\Big]
=\frac{1}{k_\ell/2}\|\bv_1^*\bx_\ell \bv_1+\bv_2^*\bx_\ell\bv_2\|_F^2.\]
We have $\|\bv_1^*\bx_\ell \bv_1+\bv_2^*\bx_\ell\bv_2\|_F^2
\leq 2\|\bv_1^*\bx_\ell \bv_1\|_F^2+2\|\bv_2^*\bx_\ell\bv_2\|_F^2
\leq 2\|\bx_\ell\|_F^2$, where both equalities are attained at
$\bx_\ell=I_{k_\ell \times k_\ell}$
(since then $\bv_1^*\bx_\ell\bv_1=\bv_2^*\bx_\ell\bv_2=I_{k_\ell/2 \times
k_\ell/2}$). Thus
\[\sup_{\bx_\ell:\|\bx_\ell\|_F^2=k_\ell} \E[(\tr \bx_\ell\bg_\ell)^2]
=\E[(\tr \bg_\ell)^2]=4.\]
Defining $\rho_\ell:=\E[(\tr \bg_\ell)^2]$, this verifies in all cases that
\[\lambda_{\max}(\nabla^2 \Psi(\bzero))
=\max_{\ell=1}^L \lambda_{\max}(H_\ell)
=\max_{\ell=1}^L \frac{1}{k_\ell}\left(-\frac{\lambda_\ell}{2}\,k_\ell
+\frac{\lambda_\ell^2}{2}\,\rho_\ell\right).\]
Then setting
$\tilde \lambda_\ell=\lambda_\ell \rho_\ell/k_\ell$, we have that
$\lambda_{\max}(\nabla^2 \Psi(\bzero))<0$
when $\max_\ell \tilde \lambda_\ell<1$, and
$\lambda_{\max}(\nabla^2 \Psi(\bzero))>0$ when $\max_\ell \tilde
\lambda_\ell>1$, as claimed in parts (a) and (b) of the proposition.

Finally, to conclude the statements about $\bzero$ being a local maximizer of
$\Psi(\bq)$,
observe that since $\Sym_{\succeq 0}$ is a (convex) cone, we have
\[B_\epsilon(\bzero):=\{\bq \in \Sym_{\succeq 0}:\|\bq\|_F \leq \epsilon\}
=\{t\,\bx:\bx \in \Sym_{\succeq 0},\,\|\bx\|_F=1,\,t \in
[0,\epsilon]\}.\]
For any such $\bq=t\bx \in B_\epsilon(\bzero)$, Taylor expansion along the line from
$\bzero$ to $\bq$ gives
\[\Psi(\bq)-\Psi(\bzero)
=\int_0^t \nabla \Psi(s\bx)[\bx]\,ds
=\int_0^t \nabla \Psi(s\bx)[\bx]-\nabla \Psi(\bzero)[\bx]
\,ds=\int_{0 \leq r \leq s \leq t} \nabla^2
\Psi(r\bx)[\bx,\bx]\,dr\,ds\]
where the second equality uses $\nabla\Psi(\bzero)=0$. 
If $\max_\ell \tilde \lambda_\ell<1$, then $\lambda_{\max}(\nabla^2
\Psi(\bzero))<0$, so by continuity there is some $\epsilon>0$ such that
$\lambda_{\max}(\nabla^2 \Psi(r\bx)) \leq -\epsilon$ for all
$r\bx \in B_\epsilon(\bzero)$. The above then implies $\Psi(\bq)<\Psi(\bzero)$
for all $\bq \in B_\epsilon(\bzero)$, so
$\bq=\bzero$ is a local maximizer of $\Psi(\bq)$.
Conversely, if $\tilde \lambda_\ell>1$ for some $\ell$, then choosing
$\bx \in \Sym_{\succeq 0}$ with $\bx_\ell=I_{k_\ell \times k_\ell}/\sqrt{k_\ell}$ and
$\bx_{\ell'}=\bzero$ for all $\ell' \neq \ell$,
the above proof verifies that
$\nabla^2 \Psi(\bzero)[\bx,\bx]>0$. Then by continuity,
$\nabla^2 \Psi(r\bx)[\bx,\bx]>\epsilon>0$ for some $\epsilon>0$ and all $r \in
[0,\epsilon]$. Then the above shows $\Psi(\bq)>\Psi(\bzero)$ for $\bq=t\bx$
and all $t \in (0,\epsilon)$, so $\bq=\bzero$ is not a local maximizer of
$\Psi(\bq)$.

\end{proofof}

\subsection{$\SO(2)$-synchronization}\label{sec:SOk}

We now prove
Theorem \ref{thm:phase-transition-so2}, providing a global analysis of the
optimization problem $\sup_{\overlap \in \cQ} \Psi_\gs(\overlap)$ for the
single-channel $\SO(2)$-synchronization model.

\begin{proofof}{Proposition \ref{prop:abelian}}
If $\overlap \in \Sym_{\succeq 0}$ is a critical point of $\Psi_\gs$, then
Proposition \ref{prop:GSderivatives}(a) shows
$\overlap_\ell=\E\langle \gb_\ell \rangle_\qb^\top\langle \gb_\ell\rangle_\qb$
for all $\ell$. Then $\overlap_\ell$ is a symmetric matrix that commutes with
$\hb_\ell$ for every $\hb \in \cG$ because $\cG$ is abelian,
so it is a multiple of the identity by Schur's lemma
(c.f.\ Theorem \ref{thm:Schurs_lemma_forRealIrredRep}).

To establish the result also for local maximizers on the boundery of
$\Sym_{\succeq 0}$, consider any $\overlap \in \Sym_{\succeq 0}$
for which some $\overlap_\ell$ is not a multiple of the identity.
The above implies
$\E\langle \gb_\ell \rangle_\qb^\top\langle \gb_\ell\rangle_\qb=\mu_\ell I$ for
some $\mu_\ell \geq 0$. If $\overlap_\ell$ has a strictly positive eigenvalue
different from $\mu_\ell$, with eigenvector $\vb_\ell$, then defining
$\bx$ by $\bx_\ell=\vb_\ell\vb_\ell^\top$ and $\bx_{\ell'}=\bzero$ for all
$\ell' \neq \ell$, Proposition \ref{prop:GSderivatives}(a) shows
$\nabla \Psi_\gs(\overlap)[\bx] \neq 0$. Then the point
$\overlap'=\overlap \pm \epsilon \bx$
for some choice of sign $\pm$ and any sufficiently small $\epsilon>0$ satisfies
$\overlap' \in \Sym_{\succeq 0}$ and $\Psi_\gs(\overlap')>\Psi_\gs(\overlap)$.
If $\overlap_\ell$ does not have a strictly positive eigenvalue
different from $\mu_\ell$, then $\overlap_\ell$ must have all eigenvalues equal
to 0 and $\mu_\ell \neq 0$. In this case, let $\vb_\ell$ be an eigenvector
corresponding to 0, and define $\bx$ in the same way.
Proposition \ref{prop:GSderivatives}(a) shows
$\nabla \Psi_\gs(\overlap)[\bx]>0$, so the point
$\overlap'=\overlap+\epsilon \bx$ for any sufficiently small $\epsilon>0$
also satisfies $\overlap' \in \Sym_{\succeq 0}$ and
$\Psi_\gs(\overlap')>\Psi_\gs(\overlap)$. In both cases, $\overlap$ is not a
local maximizer of $\Psi_\gs$, implying the proposition.
\end{proofof}

To show Theorem \ref{thm:phase-transition-so2},
since $\SO(2)$ is abelian, Proposition \ref{prop:abelian} allows us to restrict
attention to the single-letter model \eqref{eq:SO2singleletter} with
mean-squared-error function $\operatorname{mmse}(\gamma)$. The main technical
lemma is the following.

\begin{lemma}\label{lemma:SO2fproperties}
Let $F(\gamma)=1-\frac{1}{2}\operatorname{mmse}(\gamma)$.
Then $F(0)=0$, $F'(0)=1$, and $F(\gamma)$ is strictly increasing and strictly
concave over $\gamma \in (0,\infty)$.
\end{lemma}

\begin{proof}
It will be convenient to work with the complex variable
$u=e^{i\theta} \in \U(1)$ representing
\[\bg=\begin{pmatrix} \cos \theta & -\sin\theta \\ \sin \theta & \cos \theta
\end{pmatrix} \in \SO(2).\]
Observing $\by$ in the single-letter model \eqref{eq:SO2singleletter} is
equivalent to observing the sufficient statistic
$\sqrt{\gamma}(\cos \theta_*,\sin \theta_*)
+(\frac{z_{11}+z_{22}}{2},\frac{z_{21}-z_{12}}{2})$, which we may represent by
the complex observation
\[y=\sqrt{\gamma}u_*+z \in \C\]
where $u_*=e^{i\theta_*} \sim \Haar(\U(1))$ and $\Re z,\Im z \overset{iid}{\sim}
\N(0,\frac{1}{2})$. Then $p(u \mid y) \propto
e^{-|y-\sqrt{\gamma} u|^2}\propto e^{H(u;y)}$ for the Hamiltonian
\[H(u;y)=\sqrt{\gamma}(y\bar u+u \bar y)
=\gamma(u_*\bar u+u\bar u_*)+\sqrt{\gamma}(z \bar u+u\bar z).\]
Abbreviating $\E=\E_{u_*,z}$ and $\langle \cdot \rangle$ for the
posterior mean under $p(\cdot \mid y)$, we have
\begin{align*}
\operatorname{mmse}(\gamma)
&=\E[2(\cos \theta_*-\langle \cos \theta \rangle)^2
+2(\sin \theta_*-\langle \sin \theta \rangle)^2]\\
&=2\E|u_*-\langle u \rangle|^2
=2(1-\E \bar u_* \langle u \rangle)
=2(1-\E \langle \bar u \rangle \langle u \rangle),
\end{align*}
so $F(\gamma)=1-\frac{1}{2}\operatorname{mmse}(\gamma)
=\E \bar u_* \langle u \rangle=\E \langle \bar u \rangle\langle u \rangle$.
At $\gamma=0$ we have $\langle u \rangle=0$, so $F(0)=0$.

Differentiating in $\gamma$ and applying the Gaussian
integration-by-parts formulas $\E zf(z,\bar z)=\E \partial_{\bar z} f(z,\bar z)$
and $\E \bar zf(z,\bar z)=\E \partial_z f(z,\bar z)$, we get
\begin{align*}
F'(\gamma)&=\E \bar u_* \inangle{u\left(u_*\bar u+u\bar u_*
+\frac{1}{2\sqrt{\gamma}}(z\bar u+u\bar z)\right)}
-\E \bar u_* \langle u \rangle \inangle{u_*\bar u+u\bar u_*
+\frac{1}{2\sqrt{\gamma}}(z\bar u+u\bar z)}\\
&=\E\left[1-\langle \bar u \rangle \langle u \rangle
+\bar u_*^2 (\langle u^2 \rangle-\langle u \rangle^2)
+\frac{z\bar u_*}{2\sqrt{\gamma}}(1-\langle \bar u \rangle \langle u \rangle)
+\frac{\bar z \bar u_*}{2\sqrt{\gamma}}
(\langle u^2 \rangle-\langle u \rangle^2)\right]\\
&=\E\left[1-\langle \bar u \rangle \langle u \rangle
+\bar u_*^2 (\langle u^2 \rangle-\langle u \rangle^2)
+\frac{\bar u_*}{2}(-\langle u \rangle
-\langle \bar u \rangle \langle u^2 \rangle
+2\langle \bar u \rangle \langle u \rangle^2)
+\frac{\bar u_*}{2}(-\langle u\rangle-\langle \bar u \rangle
\langle u^2 \rangle+2\langle \bar u \rangle\langle u \rangle^2)\right]\\
&=\E\left[1-\langle \bar u \rangle \langle u \rangle
+\bar u_*^2 (\langle u^2 \rangle-\langle u \rangle^2)
-\bar u_* \langle u \rangle
-\bar u_* \langle \bar u \rangle \langle u^2 \rangle
+2\bar u_* \langle \bar u \rangle \langle u \rangle^2\right]\\
&=\E\left[1-2\langle \bar u \rangle \langle u \rangle
+\langle \bar u \rangle^2 \langle u \rangle^2
+(\langle \bar u^2 \rangle-\langle \bar u \rangle^2)
(\langle u^2 \rangle-\langle u \rangle^2)\right]\\
&=\E\left[(1-|\langle u \rangle|^2)^2
+|\langle u^2 \rangle-\langle u \rangle^2|^2\right].
\end{align*}
Here, both terms are non-negative. For any (finite) $\gamma>0$ the
posterior law of $u$ is not a point mass on the circle $\U(1)$,
so $|\langle u \rangle|<1$ with probability 1 over $y$. Then the first
term is strictly positive, showing that $F(\gamma)$ is strictly increasing.
At $\gamma=0$, we have
$\langle u \rangle=\langle u^2 \rangle=0$, so $F'(0)=1$.

It remains to show that $F(\gamma)$ is strictly concave. For this, 
observe first that the Hamiltonian
$H(u;y)$ defining the posterior mean $\langle \cdot \rangle$
depends on $(\gamma,y)$ only via $\sqrt{\gamma} y$. Observe next that
by rotational symmetry of the model about the origin in the complex plane,
the function $\sqrt{\gamma} y \mapsto (1-|\langle u \rangle|^2)^2
+|\langle u^2 \rangle-\langle u \rangle^2|^2$ depends only on the modulus
$\sqrt{\gamma}|y|$. Thus, setting $x=\sqrt{\gamma}|y|$, we may define
\[f(x)=(1-|\langle u \rangle|^2)^2
+|\langle u^2 \rangle-\langle u \rangle^2|^2
\text{ where } \langle u^j \rangle=\frac{\E_{u \sim \Haar(\U(1))}
u^j e^{x(u+\bar u)}}{\E_{u \sim \Haar(\U(1))} e^{x(u+\bar u)}}\]
for real arguments $x \geq 0$, and we have $F'(\gamma)=\E f(\sqrt{\gamma}|y|)$.
It then suffices to show
\begin{enumerate}
\item For any $\gamma_1>\gamma_2>0$, the law of $x_1=\sqrt{\gamma_1}|y|$
stochastically dominates that of $x_2=\sqrt{\gamma_2}|y|$, in the sense
$\P[x_1 \geq t]>\P[x_2 \geq t]$ for all $t>0$.
\item $f'(x)<0$ for all $x>0$.
\end{enumerate}
Indeed, then there would exist a coupling of $(x_1,x_2)$ so that $x_1>x_2$ with
probability 1, hence $F'(\gamma_1)-F'(\gamma_2)=\E[f(x_1)-f(x_2)]
=\E[\int_{x_2}^{x_1} f'(t)\ud t]<0$, implying strictly concavity of $F(\gamma)$.

To show claim (1), observe that $2|y|^2 \sim \chi_2^2(2\gamma)$ which is
stochastically increasing in the chi-squared non-centrality parameter $2\gamma$
(as this represents the power of a chi-squared statistical test against a
family of alternatives ordered by $\gamma$). Thus
$\P_{\gamma_1}[|y| \geq t]>\P_{\gamma_2}[|y| \geq t]$ for any $t>0$, implying
also $\P[x_1 \geq t]=\P_{\gamma_1}[\sqrt{\gamma_1} |y| \geq t]
>\P_{\gamma_1}[\sqrt{\gamma_2}|y| \geq t]>\P_{\gamma_2}[\sqrt{\gamma_2}|y| \geq
t]=\P[x_2 \geq t]$.

To show claim (2), observe that $\langle u^j \rangle$ is real for any $j \geq
0$, since the law $p(u) \propto e^{x(u+\bar u)}=e^{2x\cos\theta}$ is
conjugation-symmetric. More precisely, $p(u)$ is a von Mises
distribution on the circle, for which
\begin{equation}\label{eq:circularmoments}
u_j:=\langle u^j \rangle=I_j(2x)/I_0(2x)
\end{equation}
where $I_j(\cdot)$ is the modified Bessel function of the first kind
\begin{equation}\label{eq:bessel}
I_j(2x)=\sum_{m \geq 0}
\frac{1}{m!(m+j)!}x^{2m+j}.
\end{equation}
We have $\partial_x u_j=\langle u^{j+1}+u^{j-1} \rangle-\langle u^j \rangle
\langle u+\bar u \rangle=u_{j-1}+u_{j+1}-2u_1u_j$ and
$f(x)=(1-u_1^2)^2+(u_2-u_1^2)^2$, so
\begin{align*}
f'(x)&=-4u_1(1-u_1^2)(1+u_2-2u_1^2)+2(u_2-u_1^2)[(u_1+u_3-2u_1u_2)-2u_1(1+u_2-2u_1^2)]\\
&={-}4u_1(1+ u_2 -2u_1^2)^2+ 2(u_2-u_1^2)(u_1 + u_3 - 2u_1 u_2).
\end{align*}
We then make the following observations:
\begin{itemize}
\item It is clear from definition that $I_j(2x)>0$ for any $x>0$ and $j
\geq 0$, hence $u_1 = I_1(2x)/I_0(2x)>0$.
\item We have $u_2 - u_1^2 =
I_0(2x)^{-2}(I_2(2x)I_0(2x)-I_1(2x)^2)$, where
$I_0(2x)^{-2}>0$ and
\begin{align*}
I_2(2x)I_0(2x)-I_1(2x)^2
&=\sum_{p,q \geq 0} \frac{1}{p!(p+2)!}x^{2p+2}\frac{1}{q!q!}x^{2q}
-\sum_{p,q \geq 0} \frac{1}{p!(p+1)!}x^{2p+1}
\frac{1}{q!(q+1)!}x^{2q+1}\\
&=\sum_{k \geq 0}x^{2k+2}
\sum_{p,q:\,p+q=k} \frac{1}{p!(p+2)!q!q!}-\frac{1}{p!(p+1)!q!(q+1)!}\\
&=\sum_{k \geq 0}x^{2k+2}
\sum_{p,q:\,p+q=k} \frac{1}{k!(k+2)!}\binom{k}{p}\binom{k+2}{q}
-\frac{1}{(k+1)!(k+1)!}\binom{k+1}{p}\binom{k+1}{q}\\
&=\sum_{k \geq 0}x^{2k+2}
\Big(\frac{1}{k!(k+2)!}-\frac{1}{(k+1)!(k+1)!}\Big)\binom{2k+2}{k},
\end{align*}
the last equality using Vandermonde's identity. Here
$\frac{1}{k!(k+2)!}-\frac{1}{(k+1)!(k+1)!}<0$ for every $k \geq 0$,
so $u_2 - u_1^2<0$.
\item We have similarly $u_1+u_3-2u_1u_2=I_0(2x)^{-2}(I_3(2x)I_0(2x)
+I_1(2x)I_0(2x)-2I_2(2x)I_1(2x))$, where $I_0(2x)^{-2}>0$
and
\begin{align*}
&I_3(2x)I_0(2x)+I_1(2x)I_0(2x)-2I_2(2x)I_1(2x)\\
&=\sum_{p,q \geq 0} \frac{1}{p!(p+3)!}x^{2p+3}
\frac{1}{q!q!}x^{2q}+\sum_{p,q \geq 0}\frac{1}{p!(p+1)!}x^{2p+1}
\frac{1}{q!q!}x^{2q}-2\sum_{p,q \geq 0} \frac{1}{p!(p+2)!}x^{2p+2}
\frac{1}{q!(q+1)!}x^{2q+1}\\
&=x+\sum_{k \geq 0} x^{2k+3}
\Bigg(\sum_{p,q:\,p+q=k} \frac{1}{p!(p+3)!q!q!}
+\!\!\!\sum_{p,q:\,p+q=k+1} \frac{1}{p!(p+1)!q!q!}
-2\!\!\!\sum_{p,q:\,p+q=k} \frac{1}{p!(p+2)!q!(q+1)!}\Bigg)\\
&=x+\sum_{k \geq 0} x^{2k+3}
\Bigg(\sum_{p,q:\,p+q=k} \frac{1}{k!(k+3)!}\binom{k}{p}\binom{k+3}{q}
+\!\!\!\sum_{p,q:\,p+q=k+1} \frac{1}{(k+1)!(k+2)!}\binom{k+1}{p}\binom{k+2}{q}\\
&\hspace{1in}
-2\!\!\!\sum_{p,q:\,p+q=k} \frac{1}{(k+1)!(k+2)!}\binom{k+1}{p}\binom{k+2}{q}
\Bigg)\\
&=x+\sum_{k \geq 0} x^{2k+3}
\Bigg(\frac{1}{k!(k+3)!}\binom{2k+3}{k}
+\frac{1}{(k+1)!(k+2)!}\binom{2k+3}{k+1}
-\frac{2}{(k+1)!(k+2)!}\binom{2k+3}{k}\Bigg)\\
&=x+\sum_{k \geq 0} x^{2k+3}
\frac{1}{(k+1)!(k+2)!}\binom{2k+3}{k}
\Bigg(\frac{k+1}{k+3}+\frac{k+3}{k+1}-2\Bigg).
\end{align*}
This summand is positive for every $k \geq 0$, hence
$u_1 + u_3 - 2u_1 u_2 >0$.
\end{itemize}
Combining the above yields $f'(x)<0$ as desired, which concludes the proof.
\end{proof}

\begin{proofof}{Theorem \ref{thm:phase-transition-so2}}
The fixed-point equation \eqref{eq:critical-point-so2} is
$q=F(\lambda q)$, for the function $F(\gamma)$ of
Lemma \ref{lemma:SO2fproperties}. Here $q=0$ is a fixed point because $F(0)=0$.
Since $q \mapsto F(\lambda q)$ is bounded, increasing, and strictly concave,
this is the only fixed point when $1 \geq \partial_q F(\lambda q)|_{q=0}
=\lambda$, and there exists a unique other positive fixed point $q_*>0$
when $1<\partial_q F(\lambda q)|_{q=0}=\lambda$.
Furthermore, $\partial_q \Psi_\gs(qI)
=-\lambda q+\lambda-\frac{\lambda}{2}\operatorname{mmse}(\lambda q)
=\lambda[F(\lambda q)-q]$. When $\lambda \in (0,1]$,
we have $q>F(\lambda q)$ for all $q>0$, so $\Psi_\gs(qI)$ attains its unique
maximum at $q=0$.  When $\lambda>1$, we have $q<F(\lambda q)$ for
$q<q_*$ and $q>F(\lambda q)$ for $q>q_*$, so $\Psi_\gs(qI)$ attains its unique
maximum at $q=q_*$.

Proposition \ref{prop:abelian} then implies that
\eqref{eq:diagonalreduction} holds, and
that $\overlap=\bzero$ and $\overlap=q_*I$ are, respectively, the unique global
maximizer of $\Psi_\gs$ in the two cases $\lambda \in (0,1]$ and $\lambda>1$.
The remaining statements on $I(G_*,Y)$, $\mmse$, and overlap concentration then follow from
Theorem \ref{theorem:free-energy-group-syn-multi-representation}.
\end{proofof}

\paragraph{Example of non-identity critical point for $\SO(k)$-synchronization.}

\begin{proposition}\label{prop:SOkcounterexample}
Consider the single-channel $\SO(k)$-synchronization model of Example
\ref{example:SOk}, with $k \geq 3$.
If $\lambda>\lambda_c:=k$, then there exists a scalar
value $q_*>0$ for which $\nabla \Psi_\gs(\diag(q_*,0,\ldots,0))=0$.
\end{proposition}
\begin{proof}
Write $F(\overlap)=\E_{\gb_*,\zb} \inangle{\gb}_\overlap^\top
\inangle{\gb}_\overlap$, so $\nabla \Psi_\gs(\overlap)=0$ if and only if
$\overlap = F(\overlap)$. 

We claim that for any $\overlap$ of the form $\overlap=\diag(q,0,\ldots,0)$,
we have $F(\overlap)=\diag(q',0,\ldots,0)$ for some other $q' \geq 0$.
To see this, momentarily let $\bg_1 \in \R^k$ and $\bg_{2:k} \in \R^{k \times
(k-1)}$ denote the first and remaining $k-1$ columns of $\bg$. Observe
that when $\overlap=\diag(q,0,\ldots,0)$, $\by$ is independent of $\bg_{2:k}$
given $\bg_1$. Hence, for any fixed $\bh \in \SO(k-1)$, we have
$\E[\bg_{2:k}\bh \mid \by]=\E[\E[\bg_{2:k}\bh \mid \bg_1,\by]
\mid \by]=\E[\E[\bg_{2:k}\bh \mid \bg_1] \mid \by]$. Fixing any $\bg_1$, we
have $\E[\bg_{2:k}\bh \mid \bg_1]=\E[\bg_{2:k} \mid \bg_1]$ by invariance of
Haar measure. Thus $\E[\bg_{2:k}\bh \mid \by]=\E[\bg_{2:k} \mid \by]$ for every
$\bh \in \SO(k-1)$. Then, taking the average
over $\bh \sim \Haar(\SO(k-1))$ which has mean 0 for $k \geq 3$, we get
$\langle \bg_{2:k} \rangle_\bq=\E[\bg_{2:k} \mid \by]=0$, so
$F(\bq)$ is non-zero in only the $(1,1)$ entry, as claimed.

Thus $\diag(q,0,\ldots,0)$ is a fixed point if and only if $q=F_{11}(q)$
where $F_{11}(q)$ denotes the $(1,1)$-entry of $F(\diag(q,0,\ldots,0))$.
We note that $F_{11}(0)=0$.
By specializing Proposition \ref{prop:GSderivatives}(b)
to $L=1$, $\bq=\diag(q,0,\ldots,0)$, and $\bx=\bx'=\diag(1,0,\ldots,0)$, we have
\begin{align*}
F'_{11}(q) = \lambda\,\E_{\bg_*,\bz}\left[\inangle{ (\gb_*^\top
\gb)_{11}^2}_{\bq} - 2 (\gb_*^\top \langle\gb\rangle_{\bq})_{11}^2 +
(\langle \gb \rangle_\bq^\top \langle \gb \rangle_\bq)_{11}^2\right].
\end{align*}
When $\bq=\bzero$,
we have $\langle \bg \rangle_\bq=0$ and $\bg_*^\top \bg$ is equal in 
law to $\bg \sim \Haar(\SO(k))$, so this gives simply
$F_{11}'(0)=\lambda\,\E_{\bg}[(\bg)_{11}^2]=\frac{\lambda}{k}$.
Therefore, if $\lambda>k$, then $F_{11}'(0)>1$. As $F_{11}(q)$ is continuous
and bounded,
there must exist a solution $q_*>0$ to $q=F_{11}(q)$, and hence a fixed point
$\diag(q_*,0,\ldots,0)$ of $\Psi_\gs(\overlap)$.
\end{proof}

\section{Proofs for quadratic assignment}
\label{sec:pf-quadratic-assign}

In this section, we analyze the quadratic assignment model \eqref{eq:QAmodel}.
We start by studying the model with linear observations
\eqref{eq:QAlinearmodel} and showing Lemma \ref{lemma:QAlinearMI}
in Section \ref{sec:QAlinear}. We then prove
Theorem \ref{thm:qua-assgn-mutual-info} in Section \ref{sec:qua-assgn-RS_VP_FE},
applying the general result of Theorem \ref{thm:free-energy-general-model}
and formalizing an approximation of the free energy by that in a model with the
truncated kernel $\kappa^L$.

We will use throughout the following elementary observations: Since $\kappa$ is
continuous and $\cX$ is compact, there exists a constant $K_0<\infty$ for which
\begin{equation}\label{eq:kernelbound}
    |\kappa(x,y)|<K_0 \text{ for all } x,y \in \cX.
\end{equation}
Furthermore, since $f_\ell(x)=\mu_\ell^{-1} \int
\kappa(x,y)f_\ell(y)\rho(\ud y)$ and $\int |\kappa(x,y)f_\ell(y)|\rho(\ud y)
<K_0(\int f_\ell(y)^2\rho(\ud y))^{1/2}<\infty$, by the dominated convergence
theorem $\lim_{x' \to x} f_\ell(x')=f_\ell(x)$. Thus each $f_\ell(x)$ is also
continuous on $\cX$, so there exist constants $C_\ell<\infty$ for which
\begin{equation}\label{eq:eigenfunctionbound}
|f_\ell(x)|<C_\ell \text{ for all } x \in \cX \text{ and } \ell \geq 1.
\end{equation}

\subsection{Mutual information of the linear model}\label{sec:QAlinear}

\begin{proofof}{Lemma \ref{lemma:QAlinearMI}}
We apply the result of \cite{greenshtein2009asymptotic} for Bayesian estimation
in compound decision models. Fixing $\{x_i\}_{i=1}^N$,
let us compare the two observation models
\begin{align}
\by_i &= \sqrt{\lambda}\,\overlap^{1/2}\ub(x_{\pi_*(i)})+\bz_i
\text{ for } i=1,\ldots,N \label{eq:permutation_model} \\
\by_i' &= \sqrt{\lambda}\,\overlap^{1/2}\bv_{*i}+\bz_i'
\text{ for } i=1,\ldots,N \label{eq:bayes_model}
\end{align}
where $\{\bv_{*i}\}_{i=1}^N$ are drawn i.i.d.\ (with replacement)
from the empirical distribution of $\{\ub(x_i)\}_{i=1}^N$, and
$\bz_i,\bz_i' \overset{iid}{\sim} \cN(0,I)$.
Let $i_\lambda(\pi_*,Y_\lin)$ be the mutual
information between $\pi_*$ and $Y_\lin=(\by_i)_{i=1}^N$ in the
model~\eqref{eq:permutation_model}, and
let $i_\lambda(V_*,Y_\lin')$ be the mutual information between
$V_*=(\bv_{*i})_{i=1}^N$ and $Y_\lin'=(\by_i')_{i=1}^N$ in the
model~\eqref{eq:bayes_model}. In \eqref{eq:bayes_model}, the samples
$(\bv_{*i},\by_i')$ are i.i.d.\ given $\{x_i\}_{i=1}^N$, and a direct
calculation gives
\[\frac{1}{N} i_{\lambda}(V_*, Y_\lin')=\E_{\vb_*,\zb'}\left[\frac{\lambda}{2}\,
\vb_*^\top \overlap \vb_* - \log \E_\vb \exp\inparen{-\frac{\lambda}{2} \vb^\top
\overlap \vb + \lambda \vb^\top \overlap \vb_* + \sqrt{\lambda} \vb^\top
\overlap^{1/2}\zb'}\right]\]
where $\E_\vb,\E_{\vb_*}$ are expectations over $\vb,\vb_* \in \R^L$ sampled
uniformly at random from the empirical distribution of $\{\bu(x_i)\}_{i=1}^N$.

By the i-mmse relation \cite{guo2005mutual}, we have 
\begin{align*}
& \frac{\partial}{\partial \lambda} i_\lambda(\pi_*, Y_\lin) = \frac{1}{2}
\sum_{i=1}^N
 \E \norm{ \overlap^{1/2}\bu(x_{\pi_*(i)})-
\overlap^{1/2}\E[\bu(x_{\pi(i)}) \mid Y_\lin]}_2^2=:
\frac{1}{2} \textnormal{mmse}_{\pi_*}(\lambda), \\
& \frac{\partial}{\partial \lambda} i_\lambda(V_*, Y_\lin') = \frac{1}{2}
\sum_{i=1}^N
\E \norm{ \overlap^{1/2}\vb_{*i}- \overlap^{1/2}\E[\vb_i \mid Y_\lin']}_2^2 =:
\frac{1}{2} \textnormal{mmse}_{V_*} (\lambda).
\end{align*}
The analyses of \cite[Theorem 5.1, Corollary 5.2]{greenshtein2009asymptotic}
extend verbatim to a multivariate setting, to show
$|\textnormal{mmse}_{\pi_*}(\lambda)-
\textnormal{mmse}_{V_*} (\lambda)| \leq C_\lambda$
for a constant $C_\lambda$ depending only on $\max_{i=1}^N
\|\sqrt{\lambda}\,\ub(x_i)\overlap^{1/2}\|_2$. 
Then, applying $i_0(\pi_*,Y_\lin)=i_0(V_*,Y_\lin')=0$ and integrating over $\lambda \in
[0,1]$, we obtain
\[\abs{\frac{1}{N} i_1(\pi_*,Y_\lin) - \frac{1}{N} i_1(V_*, Y_\lin') } \le \frac{C}{N}\]
for some constant $C>0$ depending on $(C_\ell)_{\ell \leq L}$ from
\eqref{eq:eigenfunctionbound}.

Here, $i_1(\pi_*,Y_\lin)=i(\pi_*,Y_\lin)$ is the mutual information of interest in the
model \eqref{eq:QAlinearmodel}.
Since the empirical law of $\{x_i\}_{i=1}^N$ converges weakly to
$\rho$, by continuity of $\ub(x)$ we have that the law of $\vb,\vb_*$ (i.e.\
the empirical law of $\{\ub(x_i)\}_{i=1}^N$) converges
weakly to the law of $\ub(x)$ when $x \sim \rho$.
Then by the dominated convergence theorem,
$\lim_{N \to \infty} \frac{1}{N} i_1(V_*,Y_\lin')=i(x,\by)$ as defined in the lemma.
\end{proofof}

\subsection{Mutual information of the quadratic model}
\label{sec:qua-assgn-RS_VP_FE}

We now bound the discrepancy in mutual information due to
truncation of the kernel.

\begin{lemma}
Suppose Assumption \ref{asmpt:kernel} holds, and let $K_0$ satisfy
\eqref{eq:kernelbound}. Let $I(\pi_*,Y)$ be the signal-observation
mutual information in the model \eqref{eq:QAmodel}, and let $I(\pi_*,Y^L)$ be
that in the analogous model with kernel $\kappa^L$ defined by
\eqref{eq:truncated-kernel}. Then for any $\epsilon > 0$, there exists
$L_0 = L_0(\epsilon)$ such that for all $L \geq L_0$ and $N \geq 1$,
        \[\frac{1}{N}\abs{I(\pi_*,Y)-I(\pi_*,Y^L)}
\leq K_0\epsilon.\]
    \label{proposition:truncKernel_fullKernel_FE_close}
\end{lemma}
\begin{proof}
By the uniform convergence of $\kappa^L$ to $\kappa$ given by Mercer's theorem (Theorem \ref{thm:mercer}), for any $\epsilon > 0$, there exists an $L_0 = L_0(\epsilon)$ such that for all $L \geq L_0$,
    \begin{align}
        \sup_{x,y} \abs{ \kappa^L(x,y) - \kappa(x,y)} < \epsilon.
        \label{eq:truncFullKernel_Mercer_uniformConvegence}
    \end{align}
    From here on, fix any $L \geq L_0$. 
Write as shorthand
\[\kappa_{ij}=\kappa(x_{\pi(i)},x_{\pi(j)}),
\quad \kappa_{*ij}=\kappa(x_{\pi_*(i)},x_{\pi_*(j)}),
\quad \kappa_{ij}^L=\kappa^L(x_{\pi(i)},x_{\pi(j)}),
\quad \kappa_{*ij}^L=\kappa^L(x_{\pi_*(i)},x_{\pi_*(j)}).\]
The Hamiltonians associated to the model \eqref{eq:QAmodel} and the one defined
by $\kappa^L$ in place of $\kappa$ are, respectively,
\begin{align}
    \label{eq:Ham_fullKernel}
    H(\pi; \pi_*, Z) &:= -\frac{1}{2N} \sum_{i < j} \kappa_{ij}^2 + \frac{1}{N}
\sum_{i < j} \kappa_{*ij} \kappa_{ij} + \frac{1}{\sqrt{N}} \sum_{i < j}
\kappa_{ij} z_{ij}, \\
    \label{eq:Ham_truncKernel}
    H^L(\pi; \pi_*, Z) &:= -\frac{1}{2N} \sum_{i < j} (\kappa_{ij}^L)^2 +
\frac{1}{N} \sum_{i < j} \kappa_{*ij}^L\kappa_{ij}^L + \frac{1}{\sqrt{N}}
\sum_{i < j} \kappa_{ij}^L z_{ij}.
\end{align}
Let $\cF_N^\infty$ and $\cF_N^L$ denote the free energies associated with these
Hamiltonians,
\begin{align}
    \cF_N^\infty \coloneqq \frac{1}{N} \E_{\pi_*, Z} \log \E_\pi \exp H(\pi; \pi_*, Z) 
    \quad \textnormal{and} \quad 
    \cF_N^L \coloneqq \frac{1}{N} \E_{\pi_*, Z} \log \E_\pi \exp H^L(\pi; \pi_*, Z).
    \label{eq:free_energies_kernel_models}
\end{align}
Then by the same calculations as \eqref{eq:MIcalculation},
\[\frac{1}{N}I(\pi_*,Y)=\frac{1}{2N^2}\E_{\pi_*} \sum_{i<j} \kappa_{*ij}^2
-\cF_N^\infty, \qquad \frac{1}{N}I(\pi_*,Y^L)=\frac{1}{2N^2}\E_{\pi_*} \sum_{i<j}
(\kappa_{*ij}^L)^2-\cF_N^L.\]
Thus, with $H$ and $H^L$ defined in \eqref{eq:Ham_fullKernel} and \eqref{eq:Ham_truncKernel}, we have
    \begin{align}
        \frac{1}{N}\abs{I(\pi_*,Y)-I(\pi_*,Y^L)} \leq
\frac{1}{2N^2}\E_{\pi_*} \sum_{i<j} |\kappa_{*ij}^2-(\kappa_{*ij}^L)^2|
+\frac{1}{N}\,\E_{\pi_*, Z} \sup_{\pi
\in \SS_N} \abs{  H(\pi; \pi_*, Z) - H^L(\pi; \pi_*, Z) }.
        \label{eq:truncFullFE_infinityNormBound}
    \end{align}
By the boundedness of the kernels \eqref{eq:kernelbound}, and
\eqref{eq:truncFullKernel_Mercer_uniformConvegence}, for any $\pi,\pi_*$
we have 
    \begin{align*}
        \abs{ \kappa_{ij}^2 -  (\kappa_{ij}^L)^2},\,
\abs{ \kappa_{*ij} \kappa_{ij} - \kappa_{*ij}^L \kappa_{ij}^L},\,
\abs{ \kappa_{*ij}^2 - (\kappa_{*ij}^L)^2}
        \leq 2K_0 \epsilon.
    \end{align*}
Set $z_{ii}=0$ and $\kappa_{ii}=\kappa(x_{\pi(i)},x_{\pi(i)})$
for all $i=1,\ldots,N$, set $\kappa_{ij}=\kappa_{ji}$,
$\kappa_{ij}^L=\kappa_{ji}^L$, $z_{ij}=z_{ji}$ for all $i>j$, and define
the symmetric matrices $K=(\kappa_{ij})_{i,j=1}^N$,
$K^L=(\kappa_{ij}^L)_{i,j=1}^N$, and $Z=(z_{ij})_{i,j=1}^N$. Then, applying
the von-Neumann trace inequality,
\[\left|\sum_{i<j}
\kappa_{ij}z_{ij}-\kappa_{ij}^Lz_{ij}\right|
=\frac{1}{2}\abs{\tr Z(K-K^L)} \leq \frac{1}{2} \norm{Z}_{\textnormal{op}}
\norm{K-K^L}_{*} = \frac{1}{2} \norm{Z}_{\textnormal{op}} \tr(K-K^L)\]
    where $\norm{\cdot}_{*}$ denotes the nuclear norm, and the last
equality follows because $\kappa-\kappa^L$ remains a positive-semidefinite
kernel, so $K-K^L$ is a positive-semidefinite matrix. Applying $\tr (K-K^L)
\leq K_0N\epsilon$ by \eqref{eq:truncFullKernel_Mercer_uniformConvegence}
and combining the above into \eqref{eq:truncFullFE_infinityNormBound}, we obtain
    \begin{align*}
        \frac{1}{N}\abs{I(\pi_*,Y) - I(\pi_*,Y^L)} \leq 2K_0 \epsilon +
\frac{K_0\epsilon}{2\sqrt{N}} \E\norm{Z}_{\textnormal{op}}.
    \end{align*}

Denote by $\tilde Z$ a copy of $Z$ with diagonal entries replaced by independent
$\cN(0,2)$ variables, and observe that $\E[(u^\top Zu-v^\top Zv)^2]
\leq \E[(u^\top \tilde Zu-v^\top \tilde Zv)^2]$ for any unit vectors $u,v \in
\R^N$. Then by a standard application of the Sudakov-Fernique inequality (see
e.g.\ \cite[Exercise 7.3.5]{vershynin2018}),
$\E\|Z\|_{\textnormal{op}} \leq \E\|\tilde Z\|_{\textnormal{op}}
\leq 2\sqrt{N}$, and the result follows upon adjusting the value of $\epsilon$.
\end{proof}

We are now ready to prove Theorem \ref{thm:qua-assgn-mutual-info}.

\begin{proofof}{Theorem~\ref{thm:qua-assgn-mutual-info}}
We apply Theorem~\ref{thm:free-energy-general-model}. Fixing any $L \geq 1$,
define $\cG_N=\SS_N$, the feature map $\phi:\SS_N \to (\R^L)^N$ by
\eqref{eq:QAfeaturemap}, and the bilinear forms $\bullet,\otimes$ and inclusion
map $\iota(\cdot)$ by \eqref{eq:QAbilinearforms}. It is then direct to check
that all conditions of Assumption \ref{assumption:general-model} hold. 
The quantities $K(\cG_N)$, $D(\cG_N)$, and $L(\epsilon;\cG_N)$ for any fixed
$\epsilon>0$ in Theorem~\ref{thm:free-energy-general-model}
are bounded by a constant due to \eqref{eq:kernelbound}, and as $N \to \infty$,
\begin{align*}
\|Q(\iden,\iden)\|_\cL^2&=\left\|\frac{1}{N}\sum_{i=1}^N
\ub(x_i)\ub(x_i)^\top\right\|_F^2 \to \left\|\E_{x \sim \rho}\ub(x)\ub(x)^\top
\right\|_F^2
=\E_{x,x' \overset{iid}{\sim} \rho} (\ub(x)^\top \ub(x'))^2
=\E_{x,x' \overset{iid}{\sim} \rho} [\kappa^L(x,x')^2],\\
\langle \overlap,Q(\iden,\iden) \rangle_\cL
&=\tr \overlap\left(\frac{1}{N}\sum_{i=1}^N \ub(x_i)\ub(x_i)^\top\right)
\to \E_{x_* \sim \rho} \tr \overlap \ub(x_*)\ub(x_*)^\top
=\E_{x_* \sim \rho} \ub(x_*)^\top\overlap \ub(x_*)
\end{align*}
under Assumption \ref{asmpt:kernel}.
Hence from Theorem~\ref{thm:free-energy-general-model},
Lemma~\ref{lemma:QAlinearMI}, and the forms \eqref{eq:quadraticMI}
and \eqref{eq:linearMI} for the mutual informations, we have
\begin{align}
\label{eq:asym-mutual-information-truncated}
    \lim_{N \to \infty} \frac{1}{N} I(\pi_*,Y^L) =
     \frac{1}{4}\,\E_{x, x' \overset{iid}{\sim} \rho}
     \insquare{\kappa^L(x, x')^2} - \sup_{\overlap \in \Sym_{\succeq 0}^{L \times L}}\Psi_\qs^L (\overlap).
\end{align} 
Here, $\sup_{\overlap \in \Sym_{\succeq 0}^{L \times L}}\Psi_\qs^L (\overlap)$
is non-decreasing in $L$, as a restriction of this supremum to $\overlap \in
\Sym_{\succeq 0}^{L \times L}$ having last row and column equal to 0 gives the
optimization for dimension $L-1$. Thus the limit $\Psi_\infty$ exists in
$(-\infty,\infty]$, and
\[\lim_{L \to \infty} \lim_{N \to \infty} \frac{1}{N} I(\pi_*,Y^L)
=\frac{1}{4}\,\E_{x, x' \overset{iid}{\sim} \rho}
     \insquare{\kappa(x, x')^2} - \Psi_\infty.\]
Finally, Lemma \ref{proposition:truncKernel_fullKernel_FE_close} shows that
$N^{-1}I(\pi_*,Y^L)$ converges to $N^{-1}I(\pi_*,Y)$ as $L \to \infty$,
uniformly over all $N \geq 1$. Thus the limits in $L$ and $N$ on the left side
may be interchanged. Since $I(\pi_*,Y)$ is bounded below by 0 and
$\E\insquare{\kappa(x, x')^2}$ is bounded above due to
\eqref{eq:kernelbound}, this implies that $\Psi_\infty$ is finite,
concluding the proof.
\end{proofof}
\subsection{Implication on matrix MMSE}
\begin{proofof}{Corollary \ref{corollary:QAMMSE}}
    Let $Y^{(L,\lambda)} = \{y_{ij}^{(L, \lambda)}\}_{1 \le i < j \le N}$ be the observations in model \eqref{eq:QAmodel_with_lambda} defined with the 
    truncated kernel $\kappa^L$.
    Then applying Theorem \ref{thm:qua-assgn-mutual-info} with 
    $\sqrt{\lambda}\kappa$ in place of $\kappa$, the mutual information 
    $I(\pi_*, Y^{(\lambda)})$ between $\pi_*$ and $Y^{(\lambda)} = \{y_{ij}^{(\lambda)}\}$
    and the mutual information 
    $I(\pi_*, Y^{(L, \lambda)})$ for the truncated model
    satisfy
    \begin{align*}
        & \lim_{N \to \infty} \frac{1}{N}I(\pi_*, Y^{(L, \lambda)}) =
        \frac{\lambda}{4}\E_{x,x'\overset{iid}{\sim}\rho}[\kappa^L(x,x')^2] 
        - \sup_{\overlap \in \Sym_{\succeq 0}^{L \times L}}
        \Psi_\qs^L (\lambda, \overlap) \\
        & \lim_{N \to \infty} \frac{1}{N}I(\pi_*, Y^{(\lambda)}) = 
        \frac{\lambda}{4}\E_{x,x'\overset{iid}{\sim}\rho}[\kappa(x,x')^2] 
        - \Psi_\infty(\lambda)
    \end{align*}
    where 
    \begin{align*}
        & \Psi_\infty(\lambda) = \lim_{L \to \infty} 
        \sup_{\overlap \in \Sym_{\succeq 0}^{L \times L}}
        \Psi_\qs^L (\lambda, \overlap), \text{ and} \\
        & \Psi_\qs^L (\lambda, \overlap) =-\frac{\lambda}{4}\|\overlap\|_F^2
        + \E_{x_*,\bz} \log \E_x \exp\inparen{-\frac{\lambda}{2}\,\ub(x)^\top \overlap
        \ub(x)+ \lambda \ub(x)^\top \overlap \ub(x_*) + \sqrt{\lambda}\ub(x)^\top \overlap^{1/2}\zb}.
    \end{align*}
    Recall the following standard I-MMSE relation:
    \begin{align*}
        & \frac{\partial}{\partial \lambda} \frac{1}{N}I(\pi_*, Y^{(L, \lambda)}) = 
        \frac{1}{4 n^2} \E \sum_{i,j=1}^n \inparen{\kappa^L(x_{\pi_*(i)}, x_{\pi_*(j)}) 
        - \inangle{\kappa^L(x_{\pi(i)}, x_{\pi(j)})}}^2 
        =\colon \frac{1}{4} \mmse^L_\qs(\lambda) \\
        & \frac{\partial}{\partial \lambda} \frac{1}{N}I(\pi_*, Y^{(\lambda)}) = 
        \frac{1}{4 n^2} \E \sum_{i,j=1}^n \inparen{\kappa(x_{\pi_*(i)}, x_{\pi_*(j)}) 
        - \inangle{\kappa(x_{\pi(i)}, x_{\pi(j)})}}^2 
        = \frac{1}{4} \mmse_\qs(\lambda).
    \end{align*}
    Together with the fact that $\mmse^L_\qs$ is non-increasing in $\lambda$, we know
    $\lambda \mapsto -\frac{1}{N}I(\pi_*, Y^{(L, \lambda)})$ is convex. Then
    its pointwise limit as $N \to \infty$
    \[\cI^L(\lambda) \coloneqq \lim_{N \to \infty} -\frac{1}{N}I(\pi_*, Y^{(L, \lambda)})\]
    is also convex, and the set $D^L \subseteq [0, \infty)$ where $\cI^L(\lambda)$
    is differentiable has full Lebesgue measure. By the same proof as in Theorem
    \ref{theorem:free-energy-group-syn-multi-representation}(b), $\lambda \in D^L$ if and only if all the maximizers $\overlap_*$ of $\Psi_\qs^L(\lambda, \overlap)$ have the same Frobenius norm
    denoted by $q_*^L(\lambda)$, in which case
    \begin{align*}
        \lim_{N \to \infty} \mmse_\qs^L(\lambda) = 
        \E_{x,x'\overset{iid}{\sim}\rho}[\kappa^L(x,x')^2] - q_*^L(\lambda)^2.
    \end{align*}
    Furthermore, the pointwise limit as $L \to \infty$
    \[\cI(\lambda) \coloneqq \lim_{L \to \infty} \cI^L(\lambda)\]
    is also convex and differentiable on a set $D \subseteq [0, \infty)$ with
    full Lebesgue measure, and for any $\lambda \in D$,
    \begin{align*}
        \frac{\partial}{\partial \lambda} \cI(\lambda) = \lim_{L \to \infty} 
         \frac{\partial}{\partial \lambda} \cI^L(\lambda) =
         -\frac{1}{4} \lim_{L \to \infty}\lim_{N \to \infty} \mmse_\qs^L(\lambda).
    \end{align*}
    On the other hand, since $\cI(\lambda)$ is also the pointwise limit of 
    $-\frac{1}{N} I(\pi_*, Y^{(\lambda)})$ as $N \to \infty$, and since $\lambda \mapsto -\frac{1}{N} I(\pi_*, Y^{(\lambda)})$ is convex, for any $\lambda \in D$,
    \begin{align*}
        \frac{\partial}{\partial \lambda} \cI(\lambda) = - \lim_{N \to \infty}
         \frac{\partial}{\partial \lambda} \frac{1}{N} I(\pi_*, Y^{(\lambda)}) 
         = -\frac{1}{4} \lim_{N \to \infty} \mmse_\qs(\lambda).
    \end{align*}
    Hence, $\lim_{L \to \infty} q_*^L(\lambda)$ exists and
    \[\lim_{N \to \infty} \mmse_\qs(\lambda) = 
    \E_{x,x'\overset{iid}{\sim}\rho}[\kappa(x,x')^2] - \lim_{L \to \infty} q_*^L(\lambda)^2.\]
\end{proofof}

\section{Group representations}

We give a brief review of relevant notions from the
representation theory of compact groups, and refer readers to
\cite[Chapter 2]{brocker2013representations} and \cite[Chapter
1]{knapp2001representation} for further background.

Throughout, $\cG$ is a compact group, and representations are always
finite-dimensional and continuous. We will choose to fix the bases and
inner-product structures for $\C^k$ and $\R^k$, thus identifying
representations as $k \times k$ matrices.

\subsection{Complex representations}\label{appendix:complexrepr}

A \emph{(complex) representation} of $\cG$ is a
continuous map $\phi:\cG \to \C^{k \times k}$ satisfying the group homomorphism
properties $\phi(gh)=\phi(g)\phi(h)$, $\phi(g^{-1})=\phi(g)^{-1}$,
and $\phi(\iden)=I_{k \times k}$.
The representation is \emph{trivial} if $\phi(g)=I_{k \times k}$ for all $g \in
\cG$, and \emph{non-trivial} otherwise.

\begin{definition}
Given a representation $\phi:\cG \to \C^{k \times k}$,
a complex linear subspace $W \subseteq \C^k$
is invariant if $\phi(g)w \in W$ for every $g \in \cG$ and
$w \in W$. The representation $\phi$ is \emph{$\C$-irreducible} if there are
no complex invariant subspaces other than $W=\{0\}$ and $W=\C^k$.
\label{definition:G_invariantSubspace,subrepresentation,irreducible}
\end{definition}

\begin{definition}
Given two representations $\phi:\cG \to \C^{k \times k}$ 
and $\phi':\cG \to \C^{k' \times k'}$, a map $U \in
\C^{k' \times k}$ is an \emph{intertwining map} of $\phi$ with $\phi'$
if $U\phi(g)=\phi'(g)U$ for all $g \in \cG$. It is
an \emph{isomorphism} if $k=k'$ and $U$ is invertible. The representations
$\phi,\phi'$ are \emph{isomorphic} (denoted
$\phi \cong \phi'$) if there exists such an isomorphism, i.e.\ an invertible
map $U \in \C^{k \times k}$ such that $\phi(g)=U^{-1}\phi'(g)U$ for all $g
\in \cG$; otherwise $\phi$ and $\phi'$ are \emph{distinct}.
\end{definition}

\begin{theorem}[Schur's Lemma, \cite{brocker2013representations} Theorem 2.1.10]
Let $\phi:\cG \to \C^{k \times k}$ and $\phi':\cG \to \C^{k' \times
k'}$ be two $\C$-irreducible representations of $\cG$.
\begin{enumerate}[(a)]
    \item If $U \in \C^{k' \times k}$ is an intertwining map of
$\phi$ with $\phi'$, then either $U=0$ or $U$ is an isomorphism.
    \item If $U \in \C^{k \times k}$ is an intertwining map of $\phi$ with
itself, then $U=\lambda I_{k \times k}$ for some $\lambda \in \C$.
\end{enumerate}
\label{thm:Schurs_lemma}
\end{theorem}

An intertwining map $U$ of $\phi$ with itself is a map that commutes with
$\phi(g)$ for all $g \in \cG$; part (b) states that when $\phi$ is
$\C$-irreducible, any such map is a multiple
of the identity. If $\cG$ is abelian, then $U=\phi(g_0)$ is such an
intertwining map for any $g_0 \in \cG$, so an immediate consequence is the
following.

\begin{corollary}[\cite{brocker2013representations} Proposition 2.1.13]
\label{cor:onedim}
If $\cG$ is abelian and $\phi:\cG \to \C^{k \times k}$ is $\C$-irreducible, then
$k=1$.
\end{corollary}

A representation $\phi:\cG \to \C^{k \times k}$ is \emph{unitary} if
$\phi(g)$ is a unitary matrix for all $g \in \cG$,
i.e.\ $\phi(g)^*\phi(g)=I_{k \times k}$. For compact $\cG$, any
representation is isomorphic to a unitary representation \cite[Theorem
2.1.7]{brocker2013representations}.

\begin{theorem}[Complete reducibility, \cite{knapp2001representation}
Theorem 1.12(d)]
\label{thm:completereducibility}
Let $\phi:\cG \to \C^{k \times k}$ be a unitary representation of a compact
group $\cG$. Then
there exists a unitary map $U \in \C^{k \times k}$ and $\C$-irreducible
unitary representations $\phi_\ell:\cG \to \C^{k_\ell \times k_\ell}$
for $\ell=1,\ldots,L$ with $k_1+\ldots+k_L=k$, such that
\begin{equation}\label{eq:directsum}
\phi(g)=U \begin{pmatrix} \phi_1(g) & & \\ & \ddots & \\ & & \phi_L(g)
\end{pmatrix} U^{-1}.
\end{equation}
\end{theorem}

If a representation $\phi:\cG \to \C^{k \times k}$
admits a decomposition of the form \eqref{eq:directsum} for some 
invertible map $U \in \C^{k \times k}$ (where $\phi,U$ and
$\phi_1,\ldots,\phi_L$ are not necessarily unitary), then we say that
$\phi_1,\ldots,\phi_L$ are \emph{sub-representations} contained in
$\phi$, and $\phi$ is a \emph{direct sum} of $\phi_1,\ldots,\phi_L$, denoted
$\phi \cong \phi_1 \oplus \ldots \oplus \phi_L$.

\begin{definition}
The \emph{character} $\chi_\phi:\cG \to \C$ of a representation
$\phi:\cG \to \C^{k \times k}$ is the function
\[\chi_\phi(g)=\tr \phi(g).\]
\end{definition}

\begin{theorem}[Schur orthogonality, \cite{knapp2001representation} Theorem
1.12(b), \cite{brocker2013representations} Theorem 2.4.11]
\label{thm:schurorthogonality}
Let $\cG$ be a compact group, and let $\phi_\ell:\cG \to \C^{k_\ell \times
k_\ell}$ be any distinct, $\C$-irreducible, and unitary representations of
$\cG$, with corresponding characters $\chi_\ell:\cG \to \C$.
\begin{enumerate}[(a)]
\item The normalized matrix entry functions
\[\{k_\ell^{1/2}\phi_\ell(\cdot)_{ij}\}_{\ell=1,\ldots,L,\,1 \leq i,j \leq
k_\ell}\]
are orthonormal in the complex inner-product space $L^2(\cG)$ with respect to
Haar measure on $\cG$.
\item The characters $\{\chi_\ell:\ell=1,\ldots,L\}$ are also orthonormal
in $L^2(\cG)$.
\end{enumerate}
\end{theorem}

We remark that if $\phi \cong \phi'$ and $\phi \cong \phi_1 \oplus \ldots \oplus
\phi_L$, then by definition, their characters satisfy
$\chi_\phi=\chi_{\phi'}$ and $\chi_\phi=\chi_{\phi_1}+\ldots+\chi_{\phi_L}$. An
immediate consequence of this and Theorem \ref{thm:schurorthogonality}(b) is
the following.

\begin{corollary}[\cite{brocker2013representations} Theorem 2.4.12]
Two representations $\phi,\phi'$ of $\cG$ are
isomorphic if and only if $\chi_\phi=\chi_{\phi'}$, i.e.\
$\chi_\phi(g)=\chi_{\phi'}(g)$ for all $g \in \cG$.
\end{corollary}

We conclude with a basic proposition showing that if two unitary
representations are isomorphic, then the isomorphism between these
representations may also be taken to be a unitary transform.

\begin{proposition}\label{prop:unitaryequivalence}
Let $\phi,\phi':\cG \to \C^{k \times k}$ be isomorphic unitary
representations. Then there exists a unitary matrix $U \in \C^{k \times k}$
for which $\phi'(g)=U\phi(g)U^*$ for all $g \in \cG$.
\end{proposition}
\begin{proof}
Applying Theorem \ref{thm:completereducibility}, we may write $\phi(g)$ in the
form \eqref{eq:directsum} for some unitary matrix $U \in \C^{k \times k}$ and
unitary $\C$-irreducible sub-representations $\phi_1,\ldots,\phi_L$, and
similarly for $\phi'$ and some $U',\phi_1',\ldots,\phi_M'$. Since
$\chi_{\phi_1}+\ldots+\chi_{\phi_L}=\chi_{\phi_1'}+\ldots+\chi_{\phi_M'}$
and characters of distinct irreducible representations are distinct
orthogonal functions in $L^2(\cG)$,
this implies that $L=M$ and $\phi_1 \cong \phi_1',\ldots,\phi_L
\cong \phi_L'$ under some ordering of these irreducible sub-representations.
Absorbing this ordering as a permutation into $U$ and $U'$,
it suffices to prove the
proposition in the case where $\phi,\phi'$ are isomorphic and $\C$-irreducible.

Since $\phi,\phi'$ are isomorphic, there exists an invertible
matrix $U \in \C^{k \times k}$ for which $\phi(g)=U^{-1}\phi'(g)U$ for all $g
\in \cG$; we must show that we may take $U$ to be unitary.
Since $\phi(g)$ is unitary,
we have $I=\phi(g)\phi(g)^*=U^{-1}\phi'(g)UU^*\phi'(g)^*U^{-*}$.
Then, since $\phi'(g)$ is unitary, rearranging this gives $UU^*\phi'(g)
=\phi'(g)UU^*$. Thus $UU^*$ is an intertwining map of $\phi'$ with itself. Since
$\phi'$ is irreducible,
Schur's lemma implies $UU^*=\alpha I$ for some $\alpha \in \C$.
We must have $\alpha \in \R$ and $\alpha>0$ because $UU^*$ is
Hermitian positive-definite. Thus $\tilde U=U/\sqrt{\alpha}$ is unitary
and $\phi(g)=\tilde U^*\phi'(g)U$ as claimed.
\end{proof}

\subsection{Real representations}\label{appendix:realrepr}

A representation $\phi$ of $\cG$ is a \emph{real representation} if
$\phi(g)$ is real-valued for all $g \in \cG$,
i.e.\ $\phi$ is a map $\phi:\cG \to \R^{k \times k}$. It is
\emph{orthogonal} if furthermore $\phi(g)$ is an orthogonal matrix for all
$g \in \cG$, i.e.\ $\phi(g)^\top \phi(g)=I_{k \times k}$.

\begin{definition}
Given a real representation $\phi:\cG \to \R^{k \times k}$,
a real linear subspace $W \subseteq \R^k$ is invariant if $\phi(g)w \in W$
for every $g \in \cG$ and $w \in W$. The representation is
\emph{real-irreducible} (or \emph{$\R$-irreducible}) if it has no real
invariant subspaces other than $W=\{0\}$ and $W=\R^k$.
\end{definition}

The following statements are analogues of Schur's lemma and complete
reducibility in the real setting. We include their proofs for completeness,
which are similar to their complex counterparts.

\begin{theorem}
\label{thm:Schurs_lemma_forRealIrredRep}
Let $\phi:\cG \to \R^{k \times k}$ and $\phi':\cG \to \R^{k' \times k'}$ be two
real-irreducible representations of $\cG$.
\begin{enumerate}[(a)]
\item If $\U \in \R^{k' \times k}$ is an intertwining map of $\phi$ with
$\phi'$, then either $U=0$ or $U$ is an isomorphism.
\item If $U \in \R^{k \times k}$ is an intertwining map of $\phi$ with itself,
and $U$ has at least one real eigenvalue, then $U=\lambda I_{k \times k}$ for
some $\lambda \in \R$.
\end{enumerate}
\end{theorem}
\begin{proof}
For (a), since $U\phi=\phi'U$, we have that $\operatorname{ker} U \subseteq
\R^k$ is a real invariant subspace of $\phi$, and $\operatorname{im} U
\subseteq \R^{k'}$ is a real invariant subspace of $\phi'$. Thus either
$\operatorname{ker} U=\R^k$ in which case $U=0$, or $\operatorname{ker} U=0$
and $\operatorname{im} U=\R^{k'}$ in which case $k=k'$ and $U$ is an
isomorphism.

For part (b), let $\lambda \in \R$ be an eigenvalue of $U$, and let
$V_\lambda=\operatorname{ker}(U-\lambda I) \subseteq \R^k$ be its corresponding
eigenspace. For any $v \in V_\lambda$ and $g \in \cG$, we have
$U\phi(g)v=\phi(g)Uv=\lambda \phi(g)v$, so $\phi(g)v \in V_\lambda$.
Thus $V_\lambda$ is a real invariant subspace.
We have $V_\lambda \neq \{0\}$ since $\lambda$ is an eigenvalue, so
$V_\lambda=\R^k$ and $U=\lambda I_{k \times k}$.
\end{proof}

\begin{theorem}\label{thm:realreprdecomp}
Let $\phi:\cG \to \R^{k \times k}$ be an orthogonal representation of a compact
group $\cG$. Then it is an orthogonal direct sum of real-irreducible components,
i.e., there exists an orthogonal map $U \in \R^{k \times k}$ and
real-irreducible orthogonal representations
$\phi_\ell:\cG \to \R^{k_\ell \times k_\ell}$
for $\ell=1,\ldots,L$ with $k_1+\ldots+k_L=k$, such that
\[\phi(g)=U\begin{pmatrix} \phi_1(g) && \\ &\ddots& \\ && \phi_L(g)
\end{pmatrix}U^\top\]
\end{theorem}
\begin{proof}
If $\phi$ is real-irreducible, then the statement holds trivially with $L=1$ and
$U=I$. Otherwise, let $W \subset \R^k$ be a real invariant subspace not equal to
$\{0\}$ or $\R^k$, and let $W^\perp$ be its orthogonal complement.
For any $v \in W$, $w \in W^\perp$, and $g \in \cG$ we have $(\phi(g)w)^\top v
=w^\top \phi(g^{-1})v=0$ because $\phi(g^{-1})v \in W$. So $\phi(g)w \in
W^\perp$, implying that $W^\perp$ is also invariant. Thus $\phi(g)$ acts as two
separate linear maps on $W$ and $W^\perp$ for all $g \in \cG$. Choosing $U$
where the first $k_1$ columns and last $k_2=k-k_1$ columns
form orthonormal bases for $W$ and $W^\perp$, respectively, this implies that
each $\phi(g)$ takes the form
\begin{equation}\label{eq:reduction}
\phi(g)=U\begin{pmatrix} \phi_1(g) & \\ & \phi_2(g) \end{pmatrix}
U^\top \qquad \Leftrightarrow
\qquad \begin{pmatrix} \phi_1(g) & \\ & \phi_2(g) \end{pmatrix}
=U^\top \phi(g)U
\end{equation}
for some functions
$\phi_1:\cG \to \R^{k_1 \times k_1}$ and $\phi_2:\cG \to \R^{k_2 \times k_2}$.
Continuity, orthogonality, and the group representation properties of
$\phi_1,\phi_2$ follow from the equality on the right side
of \eqref{eq:reduction} and the corresponding properties for $\phi$.
Thus $\phi_1,\phi_2$ are real orthogonal sub-representations of $\cG$, of lower
dimensionalities $k_1,k_2<k$, and the result follows from induction on $k$.
\end{proof}

Any real representation $\phi:\cG \to \R^{k \times k}$ that is $\C$-irreducible
(when viewed as a complex representation under the embedding $\R^{k \times k}
\subset \C^{k \times k}$) is, by definition, also real-irreducible.
However the converse is not true, and real-irreducible representations may be
reducible in the complex sense. An example is the standard representation
of $\cG=\SO(2)$ in \eqref{eq:SO2}, which has no real invariant subspaces, but
two orthogonal complex invariant subspaces spanned by $(1,i)$ and $(i,1)$. This
example shows also that the extra assumption in
Theorem \ref{thm:Schurs_lemma_forRealIrredRep}(b) of $U$ having a real
eigenvalue cannot, in general, be removed: $\cG=\SO(2)$ is abelian, so any
$U \in \SO(2)$ is an intertwining map of $\SO(2)$ with itself, but $U$ may not
be a multiple of the identity.

\begin{theorem}[Classification of real-irreducible representations,
\cite{brocker2013representations} Table 2.6.2, Theorem 2.6.3]\label{thm:realclassification}
Let $\phi:\cG \to \R^{k \times k}$ be a real-irreducible representation. Then,
as a complex representation in $\C^{k \times k}$, it is either
\begin{enumerate}[(a)]
\item $\C$-irreducible.
\item Isomorphic to the
direct sum $\psi \oplus \bar \psi$ of two $\C$-irreducible
representations $\psi,\bar\psi$ such that
$\psi,\bar\psi$ are distinct (i.e.\ $\psi \not\cong \bar\psi$),
where $\bar \psi$ denotes the complex conjugate
representation $\bar \psi(g)=\overline{\psi (g)}$ for all $g \in \cG$.
\item Isomorphic to the direct sum $\psi \oplus \psi$ of a $\C$-irreducible
representation $\psi$ with itself, such that $\psi \cong \bar \psi$.
\end{enumerate}
\end{theorem}

We remark that since the sub-representations $\psi$ in case (b) are distinct
from those in case (c), the $\C$-irreducible sub-representations
of $\phi$ must be distinct from those of $\phi'$ if $\phi,\phi'$ are
real-irreducible and distinct. This implies also by Theorem
\ref{thm:schurorthogonality} that the corresponding characters
$\chi_\phi,\chi_{\phi'}$ are orthogonal (although not necessarily orthonormal)
in $L^2(\cG)$.

Following the terminology of \cite[Section 2.6]{brocker2013representations}, we call $\phi$ of ``real type'',
``complex type'', and ``quaternionic type'' in these cases (a), (b), and (c)
respectively. From the character relations 
$\chi_\phi=\chi_{\psi}+\chi_{\bar\psi}$ and $\chi_\phi=2\chi_{\psi}$ in the
latter two cases, and the Schur orthogonality of characters for $\C$-irreducible
representations (Theorem \ref{thm:schurorthogonality}), it is readily deduced
that $\rho:=\E_{g \sim \Haar(\cG)}[(\tr \phi(g))^2]$ takes the value 1, 2, or 4
when $\phi$ is of real, complex, or quaternionic type respectively, as stated
in \eqref{eq:irredcharacters}.

A direct consequence of Theorem \ref{thm:realclassification} and Corollary
\ref{cor:onedim} is the following.

\begin{corollary}\label{cor:twodim}
If $\cG$ is abelian and $\phi:\cG \to \R^{k \times k}$ is real-irreducible,
then $k=1$ if $\phi$ is of real type, and $k=2$ if $\phi$ is of complex or
quaternionic type.
\end{corollary}

Finally, the following is an analogue of
Proposition \ref{prop:unitaryequivalence}.

\begin{proposition}\label{prop:orthogonalequivalence}
Let $\phi,\phi':\cG \to \R^{k \times k}$ be real representations that are
isomorphic. Then there exists a (real) invertible map $U \in \R^{k \times k}$
such that $\phi(g)=U^{-1}\phi'(g)U$ for all $g \in \cG$. If furthermore
$\phi,\phi'$ are orthogonal, then there exists such a map $U$ which is also
orthogonal.
\end{proposition}
\begin{proof}
Since characters of distinct real-irreducible representations are distinct and
orthogonal functions of $L^2(\cG)$, the same argument as
in Proposition \ref{prop:unitaryequivalence} shows that $\phi_1 \cong
\phi_1',\ldots,\phi_L \cong \phi_L'$ for some ordering of the real-irreducible
sub-representations of $\phi,\phi'$, so it suffices to prove the statements
when $\phi,\phi'$ are isomorphic and real-irreducible.

Since $\phi,\phi'$ are isomorphic, there exists an invertible matrix
$U \in \C^{k \times k}$ for
which $\phi(g)=U^{-1}\phi'(g)U$; we must show that we may take $U$ to be real.
Writing the real and imaginary parts $U=P+iQ$, we have
$(P+iQ)\phi(g)=\phi'(g)(P+iQ)$. Then $P\phi(g)=\phi'(g)P$ and
$Q\phi(g)=\phi'(g)Q$ since $\phi$ and $\phi'$ are real, so
$(P+\lambda Q)\phi(g)=\phi'(g)(P+\lambda Q)$ for all $\lambda \in \R$. The
complex polynomial $f(\lambda)=\det(P+\lambda Q)$ is not identically 0 because
it is non-zero at $\lambda=i$. Then there exists also some $\lambda \in \R$ for
which $f(\lambda) \neq 0$, implying that $\tilde U=P+\lambda Q \in \R^{k \times
k}$ is invertible,
and $\phi(g)=\tilde U^{-1}\phi'(g)\tilde U$. This shows the first statement.

Now letting $U \in \R^{k \times k}$ be such that
$\phi(g)=U^{-1}\phi'(g)U$, if furthermore $\phi,\phi'$ are orthogonal, then
the same argument as in Proposition \ref{prop:unitaryequivalence} shows
$UU^\top\phi'(g)=\phi'(g)UU^\top$. Here $UU^\top$ is symmetric
positive-semidefinite, having all real eigenvalues, so Schur's lemma in the
form of Theorem \ref{thm:Schurs_lemma_forRealIrredRep}(b) implies
$UU^\top=\alpha I$ for some $\alpha>0$. Thus $\tilde U=U/\sqrt{\alpha}$ is
orthogonal and $\phi(g)=\tilde U^\top \phi'(g)U$, showing the second statement.
\end{proof}

\subsection{Canonical form for the group synchronization model}
\label{appendix:GSreduction}

Consider observations from a
model \eqref{eq:group-synchronization-multi-representation-model} with real
orthogonal representations $\phi_\ell:\cG \to \R^{k_\ell \times k_\ell}$.
By Theorem
\ref{thm:realreprdecomp} and the invariance in law of $\bz_\ell^{(ij)}$ under
the rotation $\bz_\ell^{(ij)} \mapsto
\bu^\top\bz_\ell^{(ij)}\bu$ for any orthogonal matrix
$\bu \in \R^{k_\ell \times k_\ell}$, each observation $\by_\ell^{(ij)}$ is
equivalent to observing
\[\tilde \by_\ell^{(ij)}
=\begin{pmatrix} \phi_{\ell,1}(\bg) && \\ & \ddots & \\ && \phi_{\ell,M}(\bg)
\end{pmatrix}^\top
\begin{pmatrix} \phi_{\ell,1}(\bg) && \\ & \ddots & \\ && \phi_{\ell,M}(\bg)
\end{pmatrix}+\tilde \bz_\ell^{(ij)}\]
where $\phi_{\ell,1},\ldots,\phi_{\ell,M}$ are real-irreducible orthogonal
sub-representations of $\phi_\ell$, and $\{\tilde \bz_\ell^{(ij)}\}$ are standard
Gaussian noise matrices equal in law to $\{\bz_\ell^{(ij)}\}$. The coordinates
of $\tilde \by_\ell^{(ij)}$ outside the $M$ diagonal blocks, as well as the
coordinates of any diagonal block corresponding to a trivial representation
$\phi_{\ell,m}$, carry no information about $\bg$ and hence may be discarded.
Thus the observation model
\eqref{eq:group-synchronization-multi-representation-model} is equivalent to a
model in which each representation $\phi_\ell$ is
real-irreducible and non-trivial.

If two such representations $\phi_\ell,\phi_{\ell'}$ are isomorphic, then
Proposition \ref{prop:orthogonalequivalence} implies that there exists an
orthogonal matrix $\bu \in
\R^{k_\ell \times k_\ell}$ for which $\phi_\ell(\bg)=\bu\phi_{\ell'}\bu^\top$.
Then, replacing $\{\by_\ell^{(ij)}\}$ by the equivalent observations
$\tilde \by_\ell^{(ij)}=\bu^\top \by_\ell^{(ij)}\bu$ as above, we may assume
that $\phi_\ell(\bg)=\phi_{\ell'}(\bg)$ for all $\bg \in \cG$. We may then
replace the observations in the two channels
$\{\by_\ell^{(ij)}\}$ and $\{\by_{\ell'}^{(ij)}\}$ by their sufficient
statistics $\{\frac{\by_\ell^{(ij)}+\by_{\ell'}^{(ij)}}{\sqrt{2}}\}$, 
which yields a new channel for the representation $\gb_\ell$ having the same
standard Gaussian law for the noise, and with a new signal-to-noise parameter
$\sqrt{\lambda}=\frac{\sqrt{\lambda_\ell}+\sqrt{\lambda_{\ell'}}}{\sqrt{2}}$.
Applying this replacement iteratively, the observation model
\eqref{eq:group-synchronization-multi-representation-model} is then
equivalent to a model in which the real-irreducible
representations $\{\phi_\ell\}_{\ell=1}^L$ are also distinct.

\newpage
{\small
\bibliographystyle{alpha}
\bibliography{ref}
}

\newpage

\end{document}